\documentclass{article}
\usepackage[utf8]{inputenc}
\usepackage[T1]{fontenc}
\usepackage{amsmath,amssymb}
\usepackage{graphicx}
\usepackage{amsthm, enumerate}
\theoremstyle{remark}
\usepackage[colorlinks=true]{hyperref}
\usepackage{latexsym}
\usepackage{xypic}
\usepackage{pxfonts}
\usepackage{mathrsfs}
\usepackage{fullpage}
\usepackage{wasysym}
\usepackage{enumitem}
\usepackage[affil-it]{authblk}

\theoremstyle{definition}
\newtheorem{defi}{Definition}[section]
\newtheorem{const}[defi]{Construction}
\newtheorem{expl}[defi]{Example}

\newtheorem{rmk}[defi]{Remark}
\newtheorem*{conv}{Convention}
\newtheorem*{plan}{Organization of the paper}
\newtheorem*{merci}{Acknowlegments}
\newtheorem{notat}[defi]{Notation}

\theoremstyle{plain}
\newtheorem{pro}[defi]{Proposition}
\newtheorem{thm}[defi]{Theorem}
\newtheorem{cor}[defi]{Corollary}
\newtheorem{conj}[defi]{Conjecture}
\newtheorem{lmm}[defi]{Lemma}
\newtheorem*{thm2}{Theorem}

\theoremstyle{remark}

\title{From maps between coloured operads to Swiss-Cheese algebras}
\author{Julien Ducoulombier}
\date{}

\begin{document}
\maketitle 

\abstract \noindent 
In the present work, we extract pairs of topological spaces from maps between coloured operads. We prove that those pairs are weakly equivalent to explicit algebras over the one dimensional Swiss-Cheese operad $\mathcal{SC}_{1}$. Thereafter, we show that the pair formed by the space of long embeddings and the manifold calculus limit of $(l)$-immersions from $\mathbb{R}^{d}$ to $\mathbb{R}^{n}$ is an $\mathcal{SC}_{d+1}$-algebra assuming the Dwyer-Hess' conjecture.

\section*{Introduction}
A multiplicative operad $O$ is a non-symmetric operad under the associative operad $\mathcal{A}s$. In \cite{McClure04}, McClure and Smith build a cosimplicial space $O^{\bullet}$ from a multiplicative operad $O$ and they show that, under technical conditions, its homotopy totalization has the homotopy type of a double loop space. Dwyer and Hess in \cite{Dwyer12}  and independently Turchin in \cite{Tourtchine10} identify the double loop space by proving that, under the assumption $O(0)\simeq O(1) \simeq \ast$, the following weak equivalences hold:
$$
hoTot(O^{\bullet})\simeq \Omega^{2}Operad^{h}_{ns}(\mathcal{A}s\,;\,O)\simeq \Omega^{2}Operad^{h}_{ns}(\mathcal{A}s_{>0}\,;\,O),
$$ 
where $Operad^{h}_{ns}(\mathcal{A}s\,;\,O)$ is the derived mapping space of non-symmetric operads and $\mathcal{A}s_{>0}$ is the non-unital version of the associative operad. In \cite{Ducoulombier14}, we extend this result to the coloured case by using the Swiss-Cheese operad $\mathcal{SC}_{2}$ which is a relative version of the little cubes operad $\mathcal{C}_{2}$. In that case, a typical example of $\mathcal{SC}_{2}$-algebra is a pair of topological spaces of the form
$$
\big(\, \Omega^{2}X\,;\, \Omega^{2}(X\,;\,Y):=\Omega(hofib(f:Y\rightarrow X))\,\big),
$$ 
where $f:Y\rightarrow X$ is a continuous map between pointed spaces. In order to identify typical $\mathcal{SC}_{2}$-algebras, we introduce a two-coloured operad $\mathcal{A}ct=\pi_{0}(\mathcal{SC}_{1})$ whose algebras are pairs of spaces $(A\,;\,B)$ with $A$ a topological monoid and $B$ a left $A$-module. From a pointed operad $O$ (i.e. a two-coloured operad $O$ endowed with a map $\eta:\mathcal{A}ct\rightarrow O$), a pair $(O_{c}\,;\,O_{o})$ of cosimplicial spaces is built and the pair $(hoTot(O_{c})\,;\,hoTot(O_{o}))$ is proved to be weakly equivalent to the explicit $\mathcal{SC}_{2}$-algebra
\begin{equation}\label{g6}
\big(\, \Omega^{2}Operad^{h}_{ns}(\mathcal{A}s_{>0}\,;\,O_{c})\,\,;\,\,\Omega^{2}\big(Operad^{h}_{ns}(\mathcal{A}s_{>0}\,;\,O_{c})\,;\,Operad^{h}_{ns}(\mathcal{A}ct_{>0}\,;\,O)\big)\,\big),
\end{equation}
where $\mathcal{A}ct_{>0}$ is the non-unital version of the operad $\mathcal{A}ct$. 

Pointed operads are helpful in understanding cosimplicial spaces. However, it requires a significant amount of work to identify topological spaces with the homotopy totalization of cosimplicial spaces (coming from multiplicative operads). In many cases, the relations are satisfied up to homotopy.   For instance, as soon as $d>1$, we don't know a cosimplicial model for the space of long embeddings compactly supported in higher dimensions
$$
\overline{Emb}_{c}(\mathbb{R}^{d}\,;\, \mathbb{R}^{n}):=hofib\big(\, Emb_{c}(\mathbb{R}^{d}\,;\, \mathbb{R}^{n})\longrightarrow Imm_{c}(\mathbb{R}^{d}\,;\, \mathbb{R}^{n})\,\big).
$$
\hspace{-250cm}\footnote{ETH Zurich, Ramistrasse 101, 809 \mbox{Zurich}, Switzerland}
\footnote{Key words: coloured operads, loop spaces, space of knots, model category}
\footnote{julien.ducoulombier@math.ethz.ch}
\footnote{Homepage: http://ducoulombier-math.esy.es/}

\noindent However, in the context of symmetric operads, Arone and Turchin develop in \cite{Arone14} a machinery to identify spaces of embeddings with derived mapping space of infinitesimal bimodules. In particular, for $n-d>2$, the two authors in \cite{Arone14} and simultaneously Turchin in \cite{Turchin13} prove that there is the  weak equivalence
\begin{equation}\label{a2}
\overline{Emb}_{c}(\mathbb{R}^{d}\,;\, \mathbb{R}^{n})\simeq Ibimod_{\mathcal{C}_{d}}^{h}(\mathcal{C}_{d}\,;\,\mathcal{C}_{n}).
\end{equation}  
Since the homotopy totalization can be expressed as the derived mapping space of non-symmetric infinitesimal bimodules over $\mathcal{A}s$, it is natural to expect that derived mapping spaces of infinitesimal bimodules over $\mathcal{C}_{d}$ are weakly equivalent to $(d+1)$-iterated loop spaces. For this purpose, we use the category $Bimod_{O}$ of bimodules over an operad $O$ which is an intermediate notion between infinitesimal bimodule and operad in the sense that any operad is a bimodule over itself and any bimodule with a based point in arity $1$ is also an infinitesimal bimodule. So, the Dwyer-Hess' conjecture asserts that if $\eta:\mathcal{C}_{d}\rightarrow O$ is a map of symmetric operads, then the following weak equivalences hold under the assumption $O(0)\simeq O(1)\simeq \ast$:
\begin{equation*}
Ibimod_{\mathcal{C}_{d}}^{h}(\mathcal{C}_{d}\,;\,O)\simeq \Omega^{d}Bimod^{h}_{\mathcal{C}_{d}}(\mathcal{C}_{d}\,;\,O)\simeq \Omega^{d+1}Operad^{h}(\mathcal{C}_{d}\,;\,O).
\end{equation*}
This conjecture is proved by Boavida de Brito and Weiss \cite{Weiss15} for the operad $O=\mathcal{C}_{n}$. More generally, the left weak equivalence is the subject of the paper \cite{Ducoulombier16.2} while the right weak equivalence is proved by the author in \cite{Ducoulombier16}. In particular, we show that if $\eta:O\rightarrow O'$ is a map of operads such that $O$ is a well pointed  $\Sigma$-cofibrant operad and the spaces $O(1)$ and $O'(1)$ are contractible, then the following weak equivalence holds:
\begin{equation}\label{A6}
Bimod_{O}^{h}(O\,;\,O')\simeq \Omega Operad^{h}(O\,;\,O').
\end{equation}

In the present work, we prove a relative version of the Dwyer-Hess' conjecture by extending the weak equivalences (\ref{g6}) and (\ref{A6}) to pairs of topological spaces which don't necessarily arise from homotopy totalizations. For this purpose, we consider the category of ($P$-$Q$) bimodules, denoted by $Bimod_{P\text{-}Q}$, where $P$ and $Q$ are two well pointed coloured operads. Then, we introduce a model category structure on $Bimod_{P\text{-}Q}$ and we adapt the Boardman-Vogt resolution (well known for operads, see \cite{Berger06} and \cite{Boardman73}) to obtain cofibrant replacements. In order to recognize typical $\mathcal{SC}_{1}$-algebras, we introduce a Quillen adjunction between the category of  ($P$-$Q$) bimodules and a subcategory of coloured operads (with set of colours $S=\{o\,;\,c\}$) denoted by $Op[P\,;\,Q]$:
\begin{equation}\label{g7}
\mathcal{L}:Bimod_{P\text{-}Q} \leftrightarrows Op[P\,;\,Q]:\mathcal{R}.
\end{equation}
By using explicit cofibrant resolutions, we prove the main theorem of the paper:

\begin{thm2}{[Theorem \ref{f3}]}
Let $O$ be a well pointed $\Sigma$-cofibrant operad. Let $\eta:\mathcal{L}(O)\rightarrow O'$ be a map in $Op[O\,;\,O]$ and $O_{c}'$ be the restriction of the operad $O'$ to the colour $c$. If the spaces $O(1)$ and $O'_{c}(1)$ are contractible, then the pair of topological spaces 
$$
\big(\, Bimod_{O}^{h}(O\,;\, O'_{c})\,\,;\,\, Bimod_{O}^{h}(O\,;\,\mathcal{R}(O'))\,\big)
$$
is weakly equivalent to the explicit $\mathcal{SC}_{1}$-algebra
$$
\big(\,\Omega Operad^{h}(O\,;\,O'_{c})\,\,;\,\, \Omega\big(\,Operad^{h}(O\,;\,O'_{c})\,;\, Op[O\,;\,\emptyset]^{h}(\mathcal{L}(O)\,;\,O')\,\big)\,\big). 
$$
\end{thm2}

\noindent As a consequence of the above theorem together with the Dwyer-Hess' conjecture, we are able to recognize $\mathcal{SC}_{d+1}$-algebras. In order to do that, we consider the two-coloured operad $\mathcal{CC}_{d}:=\mathcal{L}(\mathcal{C}_{d})$ in the category $Op[\mathcal{C}_{d}\,;\, \mathcal{C}_{d}]$ where $\mathcal{C}_{d}$ is the $d$-dimensional little cubes operad.   According to the notation of Theorem \ref{f3}, $O$ is the operad $\mathcal{C}_{d}$ and one has the following statement:

\begin{thm2}{[Theorem \ref{F6}]}
Assume that the Dwyer-Hess conjecture is true. Let $\eta:\mathcal{CC}_{d}\rightarrow O'$ be a map in $Op[\mathcal{C}_{d}\,;\, \mathcal{C}_{d}]$ such that $O_{c}'(0)$, $O_{c}'(1)$ and $\mathcal{R}(O')(0)$ are contractible. The pair of topological spaces
$$
\big(\, Ibimod_{\mathcal{C}_{d}}^{h}(\mathcal{C}_{d}\,;\, O'_{c})\,;\, Ibimod_{\mathcal{C}_{d}}^{h}(\mathcal{C}_{d}\,;\,\mathcal{R}(O'))\,\big)
$$
is weakly equivalent to the explicit $\mathcal{SC}_{d+1}$-algebra
$$
\left(\,\,
\Omega^{d+1} Operad^{h}(\mathcal{C}_{d}\,;\,O'_{c})\,\,;\,\,\Omega^{d+1}\big(\,\,Operad^{h}(\,\mathcal{C}_{d}\,;\,O'_{c}\,)\,\,;\,\, Op[\mathcal{C}_{d}\,;\,\emptyset]^{h}(\,\mathcal{CC}_{d}\,\,;\,\,O'\,)\,\,\big)\,\, 
\right).
$$
\end{thm2}
\noindent Furthermore, the method used in this paper produces relative deloopings for truncated infinitesimal bimodules. Roughly speaking, $T_{k}Ibimod_{O}$ is the restriction of infinitesimal bimodules to operations with at most $k$ inputs (see Section \ref{A7}). The restriction functors $T_{k}Ibimod_{O}\rightarrow T_{k-1}Ibimod_{O}$ give rise to a tower which \phantom{blablabla} \vspace{-10pt}
\newpage
\noindent plays an important role in understanding the manifold calculus tower associated to the space of embeddings. Under the conditions of Theorem \ref{F6}, if $T_{k}Operad$ denotes the category of $k$-truncated operads, then there are the following weak equivalences: 
$$
\begin{array}{rcl}\vspace{5pt}
T_{k}Ibimod_{\mathcal{C}_{d}}^{h}(T_{k}(\mathcal{C}_{d})\,;\, T_{k}(O'_{c})) & \simeq & \Omega^{d+1}\big(\, T_{k}Operad^{h}(T_{k}(\mathcal{C}_{d})\,;\,T_{k}(O'_{c}))\,\big), \\ 
T_{k}Ibimod_{\mathcal{C}_{d}}^{h}(T_{k}(\mathcal{C}_{d})\,;\,T_{k}(\mathcal{R}(O'))) & \simeq & \Omega^{d+1}\big(\,\,T_{k}Operad^{h}(\,T_{k}(\mathcal{C}_{d})\,;\,T_{k}(O'_{c})\,)\,\,;\,\, T_{k}Op[\mathcal{C}_{d}\,;\,\emptyset]^{h}(\,T_{k}(\mathcal{CC}_{d})\,\,;\,\,T_{k}(O')\,)\,\,\big).
\end{array} 
$$

An application of the previous results concerns the spaces of long embeddings. Due to the weak equivalence (\ref{a2}), we know that the space of long embeddings is related to the map of operads $\eta_{1}:\mathcal{C}_{d}\rightarrow\mathcal{C}_{n}$. In a similar way, Dobrinskaya and Turchin show in \cite{Turchin14} that the manifold calculus limit of $(l)$-immersions is related to the $\mathcal{C}_{n}$-bimodule $\mathcal{C}_{n}^{(l)}$, called the non-$(l)$-overlapping little cubes bimodule. More precisely, one has the following weak equivalences:
$$
T_{k}\overline{Im}m^{(l)}_{c}(\mathbb{R}^{d}\,;\,\mathbb{R}^{n})\simeq T_{k}Ibimod_{\mathcal{C}_{d}}^{h}(T_{k}(\mathcal{C}_{d})\,;\,T_{k}(\mathcal{C}_{n}^{(l)}))
\hspace{15pt}\text{and}\hspace{15pt}
T_{\infty}\overline{Im}m^{(l)}_{c}(\mathbb{R}^{d}\,;\,\mathbb{R}^{n})\simeq Ibimod_{\mathcal{C}_{d}}^{h}(\mathcal{C}_{d}\,;\,\mathcal{C}_{n}^{(l)}),
$$ 
where $T_{k}\overline{Im}m^{(l)}_{c}(\mathbb{R}^{d}\,;\,\mathbb{R}^{n})$ is the $k$-th polynomial approximation of the space of $(l)$-immersions (see Section \ref{F7}). This bimodule doesn't arise from an operad under the little cubes operad. However, there is a map of $\mathcal{C}_{n}$-bimodules $\eta_{2}:\mathcal{C}_{n}\rightarrow \mathcal{C}_{n}^{(l)}$ induced by the inclusion. From the maps $\eta_{1}$ and $\eta_{2}$, we build a two-coloured operad $\Theta$ under $\mathcal{CC}_{d}$ such that $\Theta_{c}=\mathcal{C}_{n}$ and $\mathcal{R}(\Theta)=\mathcal{C}_{n}^{(l)}$. As a consequence of the main theorem of the paper, the pairs of  spaces 
\begin{equation}\label{g8}
\big(\, T_{k}\overline{Emb}_{c}(\mathbb{R}^{d}\,;\, \mathbb{R}^{n}) \,;\, T_{k}\overline{Im}m^{(l)}_{c}(\mathbb{R}^{d}\,;\,\mathbb{R}^{n})\,\big)
\hspace{15pt}\text{and}\hspace{15pt}
\big(\, \overline{Emb}_{c}(\mathbb{R}^{d}\,;\, \mathbb{R}^{n}) \,;\, T_{\infty}\overline{Im}m^{(l)}_{c}(\mathbb{R}^{d}\,;\,\mathbb{R}^{n})\,\big)
\end{equation}
are proved to be weakly equivalent to explicit $\mathcal{SC}_{d+1}$-algebras.

\begin{plan}
The paper is divided into $4$ sections. The \textit{first section} gives an introduction on coloured operads and (infinitesimal) bimodules over coloured operads as well as the truncated versions of these notions. In particular, the little cubes operad, the Swiss-Cheese operad and the non-$(l)$-overlapping little cubes bimodule are defined.

 In the \textit{second section}, we give a presentation of the left adjoint functor to the forgetful functor from the category of ($P$-$Q$) bimodules to the category of $S$-sequences. This presentation is used to endow  $ Bimod_{P\text{-}Q}$ with a cofibrantly generated model category structure. Thereafter, we prove that a Boardman-Vogt type resolution yields explicit and functorial cofibrant replacements in the model category of ($P$-$Q$) bimodules. We also show that similar statements hold true for truncated bimodules.  

 The \textit{third section} is devoted to the proof of the main theorem \ref{f3}. For this purpose, we give a presentation of the functor $\mathcal{L}$ and prove that the adjunction (\ref{g7}) is a Quillen adjunction. Then we change slightly the Boardman-Vogt resolution introduced in Section \ref{G0} in order to obtain explicit cofibrant replacements in the category $Op[O\,;\,\emptyset]$. Finally, by using Theorem \ref{f3} and the Dwyer-Hess' conjecture, we identify $\mathcal{SC}_{d+1}$-algebras from maps of operads $\eta:\mathcal{CC}_{d}\rightarrow O'$. 

 In the \textit{last section} we give an application of our results to the space of long embeddings in higher dimension. We introduce quickly the Goodwillie calculus as well as the relation between the manifold calculus tower and the mapping space of infinitesimal bimodules. Then we show that the pairs (\ref{g8}) are weakly equivalent to explicit typical $\mathcal{SC}_{d+1}$-algebras.  

\end{plan}

\begin{conv}
By a space we mean a compactly generated Hausdorff space and by abuse of notation we denote by $Top$ this category (see e.g. \cite[section 2.4]{Hovey99}). If $X$, $Y$ and $Z$ are spaces, then $Top(X;Y)$ is equipped with the compact-open topology in order to have a homeomorphism $Top(X;Top(Y;Z))\cong Top(X\times Y;Z)$. By using the Serre fibrations, the category $Top$ is endowed with a cofibrantly generated monoidal model  structure. In the paper the categories considered are enriched over $Top$.
\end{conv}

\section{Bimodules and Ibimodules over coloured operads}

In what follows, we cover the notion of operad with the example of the little cubes operad and the notion of (infinitesimal) bimodule together with their truncated versions. For more details about these objects, we refer the reader to \cite{Arone14} and \cite{May72}. For our purpose, we focus on the operads with two colours $S=\{o\,;\,c\}$. In particular, we recall the definition of the $d$-dimensional Swiss-Cheese operad $\mathcal{SC}_{d}$ introduced by Voronov in \cite{Voronov99} (see also \cite{Kontsevich99}). The non-$(l)$-overlapping little cubes bimodule, which is also described below, was introduced in \cite{Turchin14} by Dobrinskaya and Turchin.

\newpage

\subsection{Topological coloured operads} \label{G5} 

\begin{defi}
Let $S$ be a set called \textit{the set of colours}. An \textit{$S$-sequence} is a family of topological spaces
$$M:=\{M(s_{1},\ldots,s_{n};s_{n+1})\}\hspace{15pt}\text{with }s_{i}\in S \text{ and } n\in \mathbb{N},$$
endowed with an action of the symmetric group: for each configuration of $n+1$ elements in $S$ and each permutation $\sigma \in \Sigma_{n}$, there is a continuous map
\begin{equation}\label{A0}
\begin{array}{rcl}
\sigma^{\ast}:M(s_{1},\ldots,s_{n};s_{n+1}) & \longrightarrow & M(s_{\sigma(1)},\ldots,s_{\sigma(n)};s_{n+1}); \\ 
 x & \longmapsto & x\cdot \sigma 
\end{array} 
\end{equation}
satisfying the relation $(x\cdot\sigma)\cdot \tau=x\cdot(\sigma\tau)$ with $\tau\in \Sigma_{n}$. A map between two $S$-sequences is given by a family of continuous maps compatible with the action of the symmetric group. In the rest of the paper, we denote by $Seq(S)$ the category of $S$-sequences. Given an integer $k\geq 1$, we also consider the category of $k$-truncated $S$-sequences $T_{k}Seq(S)$. The objects are family of topological spaces
$$M:=\{M(s_{1},\ldots,s_{n};s_{n+1})\}\hspace{15pt}\text{with }s_{i}\in S \text{ and } n\leq k,$$
endowed an action of the symmetric group (\ref{A0}) with $n\leq k$. A ($k$-truncated) $S$-sequence is said to be pointed if there are distinguished elements $\{\ast_{s}\in O(s\,;\,s)\}_{s\in S}$ called \textit{units}.
\end{defi}

\begin{defi}
An \textit{$S$-operad} is a pointed $S$-sequence $O$ together with operations called \textit{operadic compositions}
\begin{equation}\label{A1}
\circ_{i}:O(s_{1},\ldots,s_{n};s_{n+1})\times O(s'_{1},\ldots,s'_{m};s_{i})\longrightarrow O(s_{1},\ldots,s_{i-1},s'_{1},\ldots,s'_{m},s_{i+1},\ldots,s_{n};s_{n+1}),
\end{equation}
with $s_{j},\,s_{j}'\in S$ and  $1\leq i\leq n$, satisfying compatibility with the action of the symmetric group, associativity and unit axioms. A map between two coloured operads should respect the operadic compositions.  We denote by $Operad_{S}$ the categories of $S$-operads. Given an integer $k\geq 1$, we also consider the category of $k$-truncated $S$-operads $T_{k}Operad_{S}$. The objects are pointed $k$-truncated $S$-sequences endowed with operadic compositions (\ref{A1}) with $n+m-1\leq k$ and $n\leq k$. One has an obvious functor
$$
T_{k}(-):Operad_{S} \longrightarrow T_{k}Operad_{S}.
$$
\end{defi}

\begin{notat}
If the set of colours $S$ has only one element, then an $S$-operad $O$ is said to be uncoloured. In this case, $O$ is a family of topological spaces $\{O(n)\}_{n\geq 0}$ together with a symmetric group action and operadic compositions
$$
\circ_{i}:O(n)\times O(m)\longrightarrow O(n+m-1), \hspace{15pt} \text{with } 1\leq i\leq n.
$$
Let $Operad$ and $T_{k}Operad$ be the categories of uncoloured operads and $k$-truncated uncoloured operads respectively. On the other hand, if $S=\{o\,;\,c\}$, then the "open" colour $o$ will be represented in red whereas the "closed" colour $c$ will be represented in black in the rest of the paper. 
\end{notat}

\begin{expl}\textbf{The uncoloured operad $\mathcal{C}_{d}^{\infty}$}

\noindent A $d$-dimensional little cube is a continuous map $c:[0\,,\,1]^{d}\rightarrow [0\,,\,1]^{d}$ arising from an affine embedding preserving the direction of the axes. The operad $\mathcal{C}_{d}^{\infty}$ is the sequence $\{\mathcal{C}_{d}^{\infty}(n)\}$ whose $n$-th component is given by $n$ little cubes, that is, $n$-tuples $<c_{1},\ldots,c_{n}>$ with $c_{i}$ a $d$-dimensional little cube. The distinguished point in $\mathcal{C}_{d}^{\infty}(1)$ is the identity little cube $id:[0\,,\,1]^{d}\rightarrow [0\,,\,1]^{d}$ whereas $\sigma\in \Sigma_{n}$ permutes the indexation:
$$
\sigma^{\ast}:\mathcal{C}_{d}^{\infty}(n)\longrightarrow \mathcal{C}_{d}^{\infty}(n)\,\,\,;\,\,\, <c_{1},\ldots,c_{n}>  \longmapsto  <c_{\sigma(1)},\ldots,c_{\sigma(n)}>.
$$ 
The operadic compositions are given by the formula
$$
\begin{array}{clcl}
\circ_{i}: & \mathcal{C}_{d}^{\infty}(n)\times \mathcal{C}_{d}^{\infty}(m) & \longrightarrow & \mathcal{C}_{d}^{\infty}(n+m-1); \\ 
 &  <c_{1},\ldots,c_{n}>\,;\, <c'_{1},\ldots,c'_{m}> & \longmapsto & <c_{1},\ldots,c_{i-1}, c_{i}\circ c'_{1},\ldots, c_{i}\circ c'_{m},c_{i+1},\ldots,c_{n}>.
\end{array} 
$$
By convention $\mathcal{C}_{d}^{\infty}(0)$ is the one point topological space and the operadic composition $\circ_{i}$ with this point consists in forgetting the $i$-th little cube.
\end{expl}

\begin{defi}
Let $S$ be a set and $O$ be an $S$-operad. An algebra over the operad $O$, or \textit{$O$-algebra}, is given by a family of topological spaces $X:=\{X_{s}\}_{s\in S}$ endowed with operations
$$
\mu:O(s_{1},\ldots , s_{n};s_{n+1})\times X_{s_{1}}\times \cdots\times X_{s_{n}}\longrightarrow X_{s_{n+1}},
$$ 
compatible with the operadic compositions and the action of the symmetric group.
\end{defi}

\begin{expl}\label{a8}\textbf{The little cubes operad $\mathcal{C}_{d}$}\\
\noindent The $d$-dimensional little cubes operad $\mathcal{C}_{d}$ is the 
sub-operad of $\mathcal{C}_{d}^{\infty}$ whose $n$-th component is the configuration space of $n$ little cubes with disjoint interiors. In other words, $\mathcal{C}_{d}(n)$ is the subspace of $\mathcal{C}_{d}^{\infty}(n)$  formed by configurations $<c_{1},\ldots,c_{n}>$ satisfying the relation
\begin{equation}\label{a1}
Int(Im(c_{i}))\cap Int(Im(c_{j}))=\emptyset, \hspace{15pt} \forall i\neq j.
\end{equation}
The operadic compositions and the action of the symmetric group arise from the operad $\mathcal{C}_{d}^{\infty}$. Furthermore, if  $(X\,;\,\ast)$ is a pointed topological space, then the $d$-iterated loop space $\Omega^{d}X$ is an example of $\mathcal{C}_{d}$-algebra.
\end{expl}

\begin{figure}[!h]
\begin{center}
\includegraphics[scale=0.3]{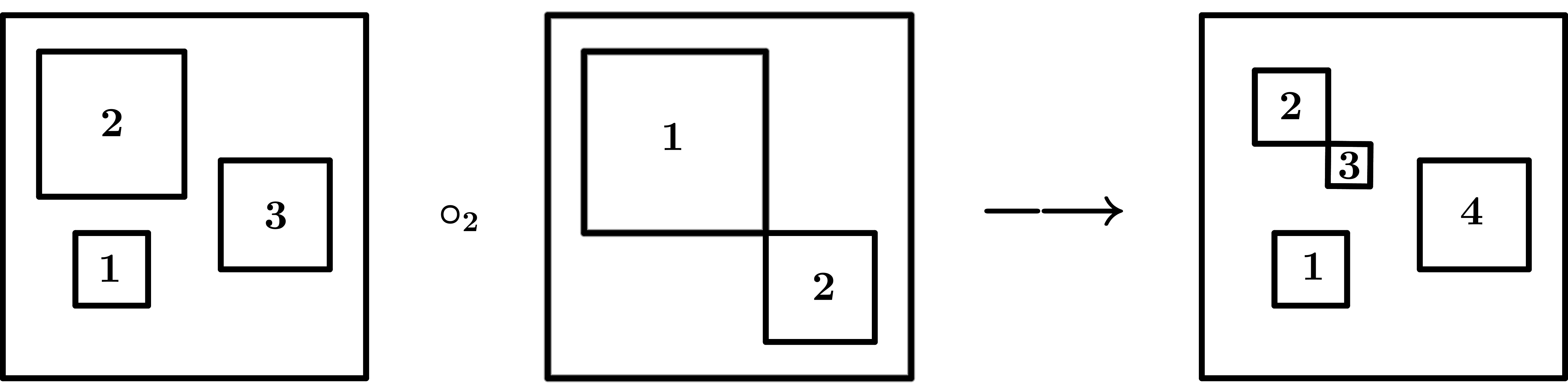}
\caption{The operadic composition $\circ_{2}:\mathcal{C}_{2}(3)\times \mathcal{C}_{2}(2)\rightarrow \mathcal{C}_{2}(4)$.}\vspace{-10pt}
\end{center}
\end{figure}

\begin{expl}\label{a5}\textbf{The Swiss-Cheese operad $\mathcal{SC}_{d}$}\\
The $d$-dimensional Swiss-Cheese operad $\mathcal{SC}_{d}$, considered in this paper, is the alternative version introduced by Kontsevich in \cite{Kontsevich99}. Its restriction to the colour $c$ coincides with the little cubes operad $\mathcal{C}_{d}$:
$$
\mathcal{SC}_{d}(\,\underset{n}{\underbrace{c,\ldots,c}}\,;c)=
\left\{
\begin{array}{cl}\vspace{4pt}
\mathcal{C}_{d}(n) & \text{if}\,\,n\geq 1, \\ 
\ast & \text{if}\,\, n=0.
\end{array} 
\right.
$$
The space $\mathcal{SC}_{d}(s_{1},\ldots,s_{n};o)$ is a subspace of $\mathcal{C}_{d}^{\infty}(n)$ formed by configurations of $n$ little cubes $<c_{1},\ldots,c_{n}>$ satisfying the relation (\ref{a1}) and the  following condition:
$$
s_{i}=o \Rightarrow c_{i}(F_{1})\subset F_{1} \hspace{15pt}\text{with}\hspace{15pt} F_{1}:=\{\,(t_{1},\ldots,t_{d})\in [0\,,\,1]^{d}\,\,|\,\, t_{1}=1\,\}.
$$
By convention the spaces $\mathcal{SC}_{d}(\,\,;\,o)$ is the one point topological space. Furthermore, if there is an integer $i$ so that $s_{i}=o$, then $\mathcal{SC}_{d}(s_{1},\ldots,s_{n};c)=\emptyset$. The operadic compositions and the action of the symmetric group arise from the operad $\mathcal{C}_{d}^{\infty}$.
\begin{figure}[!h]
\begin{center}
\includegraphics[scale=0.3]{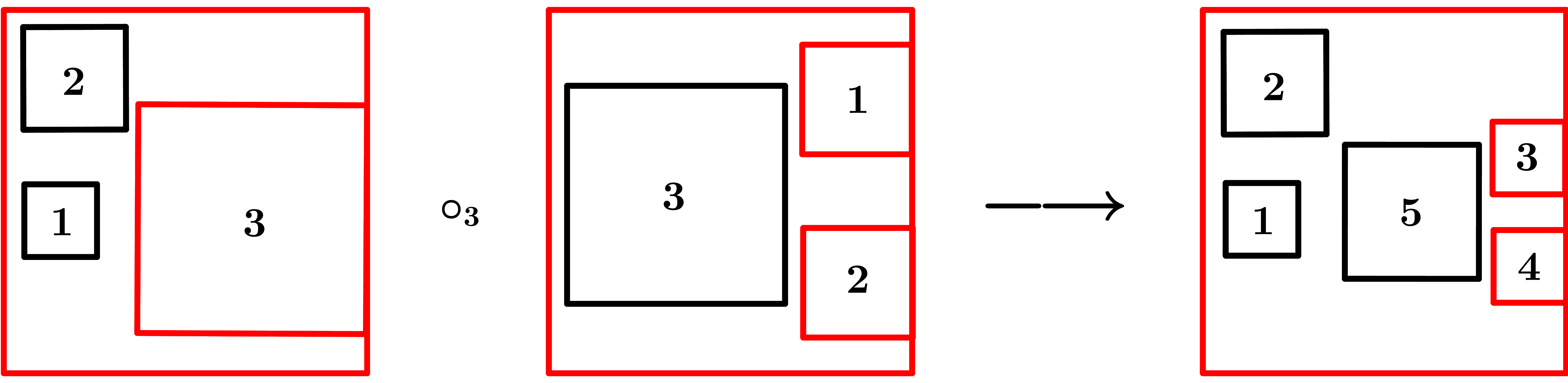}
\caption{The operadic composition $\circ_{3}:\mathcal{SC}_{2}(c,c,o;o)\times \mathcal{SC}_{2}(o,o,c;o)\rightarrow \mathcal{SC}_{2}(c,c,o,o,c;o)$.}\vspace{-10pt}
\end{center}
\end{figure}

\noindent If $f:(Y\,;\,\ast)\rightarrow (X\,;\,\ast)$ is a continuous map between pointed spaces, then the pair $(\Omega^{d}X\,;\,\Omega^{d}(X\,;\,Y))$ is an example of $\mathcal{SC}_{d}$-algebra. Contrary to the little cubes operad, there is no recognition principle for the Swiss-Cheese operad. The conjecture saying that the algebras over $\mathcal{SC}_{d}$ are weakly equivalent to pairs of the form $(\Omega^{d}X\,;\,\Omega^{d}(X\,;\,Y))$ is only proved in the case $d=1$ by Hoefel, Livernet and Stasheff in \cite{Livernet13}.  By $\Omega^{d}(X\,;\,Y)$ we mean the $d$-iterated relative loop space defined as follows:
$$
\begin{array}{rl}\vspace{7pt}
\Omega^{d}(X\,;\,A) & := hofib\big( \, \Omega^{d-1}f:\Omega^{d-1}A\longrightarrow \Omega^{d-1}X\,\big), \\ 
 & := \Omega^{d-1}\big( \,hofib( f:A\longrightarrow X)\,\big).
\end{array} 
$$
\end{expl}

\subsection{Bimodules over coloured operads}

\begin{defi}\label{c7}
Let $P$ and $Q$ be two $S$-operads. A \textit{($P$-$Q$) bimodule} is an $S$-sequence $M$ together with operations 
\begin{equation}\label{A2}
\begin{array}{ll}\vspace{7pt}
\gamma_{r}: & M(s_{1},\ldots,s_{n};s_{n+1})\times Q(s^{1}_{1},\ldots,s^{1}_{m_{1}};s_{1})\times \cdots\times Q(s^{n}_{1},\ldots,s^{n}_{m_{n}};s_{n})  \longrightarrow  M(s^{1}_{1},\ldots,s^{n}_{m_{n}};s_{n+1}), \\ \vspace{7pt}
\gamma_{l}: & P(s_{1},\ldots,s_{n};s_{n+1})\times M(s^{1}_{1},\ldots,s^{1}_{m_{1}};s_{1})\times \cdots\times M(s^{n}_{1},\ldots,s^{n}_{m_{n}};s_{n})  \longrightarrow  M(s^{1}_{1},\ldots,s^{n}_{m_{n}};s_{n+1}),\vspace{-10pt}
\end{array}
\end{equation}
satisfying compatibility with the action of the symmetric group, associativity and unity axioms. In particular, for each $s\in S$ there is a map $\gamma_{s}: P(\,;\, s)\longrightarrow M(\,;\,s)$. A map between ($P$-$Q$) bimodules should respect the operations. We denote by $Bimod_{P\text{-}Q}$ the category of ($P$-$Q$) bimodules. 

Given an integer $k\geq 1$, we also consider the category of $k$-truncated bimodules $T_{k}Bimod_{P\text{-}Q}$. An object is a $k$-truncated $S$-sequence endowed with a bimodule structure (\ref{A2}) for $m_{1}+\cdots + m_{n}\leq k$ (and $n\leq k$ for $\gamma_{r}$). To simplify the notation, the categories $Bimod_{P\text{-}P}$ and $T_{k}Bimod_{P\text{-}P}$ are denoted by $Bimod_{P}$ and $T_{k}Bimod_{P}$ respectively. One has an obvious functor 
$$
T_{k}(-):Bimod_{P\text{-}Q}\longrightarrow T_{k}Bimod_{P\text{-}Q}.
$$
\end{defi}

\begin{notat}
Thanks to the distinguished points in $Q(s\,;\,s)$, the right operations $\gamma_{r}$ can equivalently be defined as a family of continuous maps
$$
\circ^{i}:M(s_{1},\ldots,s_{n};s_{n+1})\times Q(s'_{1},\ldots,s'_{m};s_{i})\longrightarrow M(s_{1},\ldots,s_{i-1},s'_{1},\ldots,s'_{m},s_{i+1},\ldots, s_{n};s_{n+1}),\hspace{15pt}\text{with }1\leq i\leq n.
$$
Furthermore, we use the following notation: 
$$
\begin{array}{ll}\vspace{7pt}
x\circ^{i}y=\circ^{i}(x\,;\,y) & \text{for } x\in M(s_{1},\ldots,s_{n};s_{n+1}) \text{ and } y\in Q(s'_{1},\ldots,s'_{m};s_{i}),  \\ 
x(y_{1},\ldots,y_{n})=\gamma_{l}(x\,;\,y_{1}\,;\ldots;\,y_{n}) & \text{for } x\in P(s_{1},\ldots,s_{n};s_{n+1}) \,\,\text{ and } y_{i}\in M(s^{i}_{1},\ldots,s^{i}_{m_{i}};s_{i}).
\end{array} 
$$
\end{notat}

\begin{expl}\label{d1}
Let $\eta:O_{1}\rightarrow O_{2}$ be a  map of $S$-operads. In that case, the map $\eta$ is also  a bimodule map over $O_{1}$ and the bimodule structure on $O_{2}$ is defined as follows: 
$$
\begin{array}{lrll}\vspace{4pt}
\gamma_{r}: & O_{2}(s_{1},\ldots,s_{n};s_{n+1})\times O_{1}(s^{1}_{1},\ldots,s^{1}_{m_{1}};s_{1})\times \cdots\times O_{1}(s^{n}_{1},\ldots,s^{n}_{m_{n}};s_{n})  & \longrightarrow &  O_{2}(s^{1}_{1},\ldots,s^{n}_{m_{n}};s_{n+1}); \\\vspace{7pt}
& (x\,;\,y_{1},\ldots,y_{n})&\longmapsto & (\cdots(x\circ_{n}\eta(y_{n}))\cdots)\circ_{1}\eta(y_{1})),\\ \vspace{4pt}
\gamma_{l}: &  O_{1}(s_{1},\ldots,s_{n};s_{n+1})\times O_{2}(s^{1}_{1},\ldots,s^{1}_{m_{1}};s_{1})\times \cdots\times O_{2}(s^{n}_{1},\ldots,s^{n}_{m_{n}};s_{n})  & \longrightarrow &  O_{2}(s^{1}_{1},\ldots,s^{n}_{m_{n}};s_{n+1});\\
& (x\,;\,y_{1},\ldots,y_{n})&\longmapsto & (\cdots(\eta(x)\circ_{n}y_{n})\cdots)\circ_{1}y_{1}.
\end{array} 
$$
\end{expl}

\begin{expl}\textbf{The non-$(l)$-overlapping little cubes bimodule $\mathcal{C}_{d}^{(l)}$}\\
The $d$-dimensional non-$(l)$-overlapping little cubes bimodule $\mathcal{C}_{d}^{(l)}$ has been introduced by Dobrinskaya and Turchin in \cite{Turchin14}. The space $\mathcal{C}_{d}^{(l)}(n)$ is the subspace of $ \mathcal{C}_{d}^{\infty}(n)$ formed by configurations of $n$ little cubes $<c_{1},\ldots,c_{n}>$ satisfying the following relation:
\begin{equation}\label{i3}
\forall\, i_{1}< \cdots < i_{l} \in \{1,\ldots, n\}, \hspace{15pt} \underset{1\leq j\leq l}{\bigcap}\, Int(Im(c_{i_{j}}))=\emptyset.
\end{equation}
In particular, $\mathcal{C}_{d}^{(2)}$ coincides with the little cubes operad $\mathcal{C}_{d}$. The action of the symmetric group and the bimodule structure over the little cubes operad $\mathcal{C}_{d}$ arise from the operadic structure of $\mathcal{C}_{d}^{\infty}$. 
\begin{figure}[!h]
\begin{center}
\includegraphics[scale=0.32]{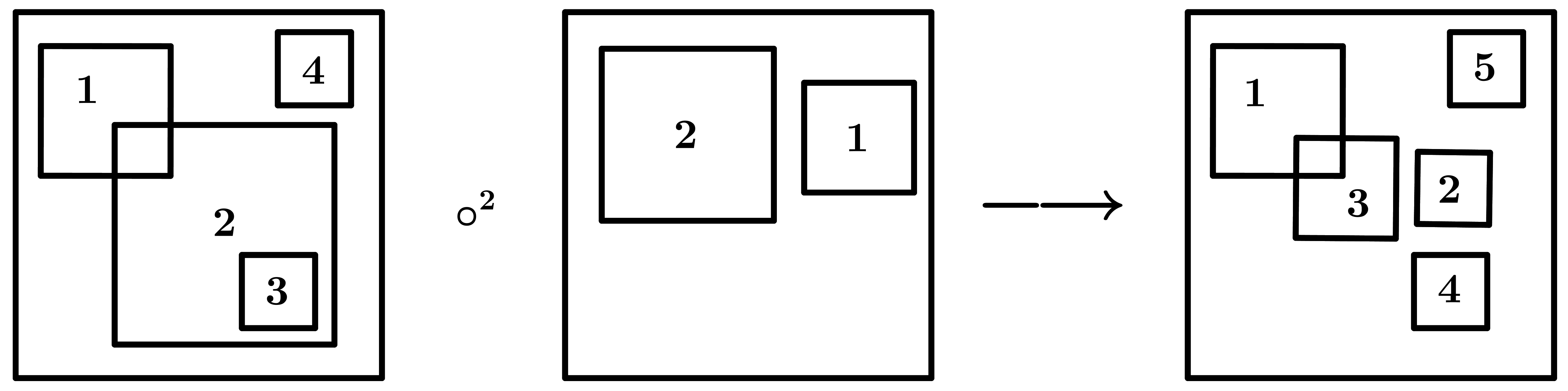}\vspace{-10pt}
\caption{The right module structure $\circ^{2}:\mathcal{C}_{2}^{(3)}(4) \times \mathcal{C}_{2}(2)\rightarrow \mathcal{C}_{2}^{(3)}(5)$.}\vspace{-35pt}
\end{center}
\end{figure}
\end{expl}

\newpage

\subsection{Infinitesimal bimodules over a coloured operad}\label{A7}

\begin{defi}\label{bimodule}
Let $O$ be an $S$-operad. An \textit{infinitesimal bimodule} over the operad $O$, or \textit{$O$-Ibimodule}, is an $S$-sequence $N$ endowed with operations 
\begin{equation}\label{A3}
\begin{array}{ll}\vspace{7pt}
\circ_{i}:O(s_{1},\ldots,s_{n};s_{n+1})\times N(s'_{1},\ldots,s'_{m};s_{i}) 
\rightarrow N(s_{1},\ldots,s_{i-1},s'_{1},\ldots,s'_{m},s_{i+1},\ldots,s_{n};s_{n+1}), & \text{for } 1\leq i\leq n, \\ \vspace{7pt}
\circ^{i}:N(s_{1},\ldots,s_{m};s_{m+1})\times O(s'_{1},\ldots,s'_{n};s_{i})
\rightarrow N(s_{1},\ldots,s_{i-1},s'_{1},\ldots,s'_{n},s_{i+1},\ldots,s_{m};s_{m+1}), & \text{for } 1\leq i\leq n,
\end{array} 
\end{equation}
satisfying compatibility with the action of the symmetric group, associativity, commutativity and unit relations. A map between $O$-Ibimodules should respect the operations. We denote by $Ibimod_{O}$ the category of infinitesimal bimodules over $O$. By convention, $\circ_{i}$ and $\circ^{i}$ are called the left and the right infinitesimal operations respectively. 

Given an integer $k\geq 1$, we also consider the category of $k$-truncated infinitesimal bimodules $T_{k}Ibimod_{O}$. An object is a $k$-truncated $S$-sequence endowed with an infinitesimal bimodule structure (\ref{A3}) for $n+m-1\leq k$ and $m\leq k$.  One has an obvious functor
$$
T_{k}(-):Ibimod_{O}\longrightarrow T_{k}Ibimod_{O}.
$$
\end{defi}

\begin{notat}
We will use the following notation:
$$
\begin{array}{ll}\vspace{7pt}
x\circ_{i}y=\circ_{i}(x\,;\,y) & \text{for } x\in O(s_{1},\ldots,s_{n};s_{n+1}) \,\text{ and } y\in N(s'_{1},\ldots,s'_{m};s_{i}),  \\ 
x\circ^{i}y=\circ^{i}(x\,;\,y) & \text{for } x\in N(s_{1},\ldots,s_{n};s_{n+1}) \text{ and } y\in O(s'_{1},\ldots,s'_{m};s_{i}).
\end{array} 
$$
\end{notat}

\begin{expl}
Let $O$ be an $S$-operad and $\eta:O\rightarrow M$ be a map of $O$-bimodules. In that case, the map $\eta$ is also an infinitesimal bimodule map over $O$. The right operations and the right infinitesimal operations are the same. So, the left infinitesimal operations on $M$ are defined as follows:
$$
\begin{array}{rcl}\vspace{4pt}
\circ_{i}:O(s_{1},\ldots,s_{n};s_{n+1})\times M(s'_{1},\ldots,s'_{m};s_{i}) & \longrightarrow & M(s_{1},\ldots,s_{i-1},s'_{1},\ldots,s'_{m},s_{i+1},\ldots,s_{n};s_{n+1}); \\ 
(\,x\,;\,y) & \longmapsto & \gamma_{l}(\,x\,;\,\eta(\ast_{s_{1}}),\ldots,\eta(\ast_{s_{i-1}}),y,\eta(\ast_{s_{i+1}}),\ldots,\eta(\ast_{s_{n}})\,).
\end{array} 
$$ 
\end{expl}

\section{The Boardman-Vogt resolution for ($P$-$Q$) bimodules}\label{G0}

In \cite{Ducoulombier14}, we introduce a model category structure for bimodules over a non-symmetric $S$-operad $O$ satisfying the condition $O(\,;\,s)=\emptyset$ for any $s\in S$. For this purpose, we give a presentation of the left adjoint to the forgetful functor from bimodules to the category of non-symmetric $S$-sequences. In the present work, we extend this result to the category of ($P$-$Q$) bimodules and the category of $k$-truncated bimodules. Contrary to \cite{Ducoulombier14}, we have to take into account the action of the symmetric group as well as the continuous maps $\gamma_{s}$ in arity $0$. For this reason, we consider the following categories:

\begin{defi}
Let $P$ be an $S$-operad and $k\geq 0$ be an integer. We denote by $Seq(S)_{P_{0}}$ and $T_{k}Seq(S)_{P_{0}}$ the categories of $S$-sequences and $k$-truncated $S$-sequences $M$ endowed with continuous maps $\gamma_{s}:P(\,;\,s)\rightarrow M(\,;\,s)$. In other words, if $P_{0}$ is the $S$-sequence given by $P_{0}(\,;\,s)=P(\,;\,s)$, for $s\in S$, and the empty set otherwise, then
$$
Seq(S)_{P_{0}}:= P_{0}\downarrow Seq(S) \hspace{15pt}\text{and} \hspace{15pt} T_{k}Seq(S)_{P_{0}}:= T_{k}(P_{0})\downarrow T_{k}Seq(S).
$$
\end{defi}

The categories $Seq(S)$ and $T_{k}Seq(S)$ have a model category structure in which a map is a fibration (resp. a weak equivalence) if each of its components is a Serre fibration (resp. a weak homotopy equivalence). This model category structure is cofibrantly generated and each object is fibrant (see \cite{Berger03} or \cite[section 3]{Ducoulombier14}). As a consequence, $Seq(S)_{P_{0}}$ and $T_{k}Seq(S)_{P_{0}}$ inherit a model category structure with the same properties. By using the following adjunctions:
\begin{equation}\label{A8}
\mathcal{F}_{B}:Seq(S)_{P_{0}}\leftrightarrows Bimod_{P\text{-}Q}:\mathcal{U}\hspace{15pt}\text{and} \hspace{15pt} T_{k}\mathcal{F}_{B}:T_{k}Seq(S)_{P_{0}}\leftrightarrows T_{k}Bimod_{P\text{-}Q}:\mathcal{U},
\end{equation}
we are able to define a model category structure on $Bimod_{P\text{-}Q}$ and $T_{k}Bimod_{P\text{-}Q}$. The following theorem is a consequence of the transfer theorem \cite[section 2.5]{Berger03}. Its proof is similar to  \cite[Application 3.5]{Ducoulombier14}.

\begin{thm}\label{b4}
Let $P$ and $Q$ be two $S$-operads. The categories $Bimod_{P\text{-}Q}$ and $T_{k}Bimod_{P\text{-}Q}$ have a cofibrantly generated model category structure  in which every object is fibrant. In particular, a map $f$ in $Bimod_{P\text{-}Q}$ or $T_{k}Bimod_{P\text{-}Q}$ is a fibration (resp. weak equivalence) if and only if the map $\mathcal{U}(f)$ is a fibration (resp. weak equivalence) in $Seq(S)$ or $T_{k}Seq(S)$.  
\end{thm}

\begin{rmk}
It is important to mention that the model category structure on $Bimod_{P\text{-}Q}$ and $T_{k}Bimod_{P\text{-}Q}$ that one gets is the same as one would get from the classical adjunction between $Seq(S)$ and $Bimod_{P\text{-}Q}$ which is the usual model structure on $Bimod_{P\text{-}Q}$. Essentially because the classical adjunction can be factorized as follows: 
$$
Seq(S)\leftrightarrows Seq(S)_{P_{0}} \leftrightarrows Bimod_{P\text{-}Q}.
$$
\end{rmk}

In this section, we give a presentation of the left adjoint functors (\ref{A8}). Thereafter, we adapt the Boardman-Vogt resolution in the context of ($k$-truncated) bimodules to obtain functorial cofibrant replacements. As did Vogt in \cite{Boardman73} for operads (see also \cite{Berger03} and \cite{Vogt03}), we use the language of trees in order to obtain such resolutions. For this reason, we fix some notation.

\begin{defi}\label{C9}
A \textit{planar tree} $T$ is a finite planar tree with one output edge on the bottom and inputs edges on the top. The vertex connected to the output edge, called the root of $T$, is denoted by $r$. Such an element is endowed with an orientation from top to bottom. Furthermore, we introduce the following notation: 
\begin{itemize}[leftmargin=*]
\item[$\blacktriangleright$] The set of its vertices and the set of its edges are denoted by $V(T)$ and $E(T)$ respectively. The set of its internal edges $E^{int}(T)$ is formed by the edges connecting two vertices. Each edge is joined to the trunk by a unique path composed of edges.  
\item[$\blacktriangleright$] According to the orientation of the tree, if $e$ is an internal edge, then its vertex  $t(e)$ towards the trunk is called the \textit{target vertex} whereas the other vertex $s(e)$ is called the \textit{source vertex}. 
\item[$\blacktriangleright$] An edge with no source is called a \textit{leaf} and the leaves are ordered from  left to right. Let $in(T):=\{l_{1},\ldots, l_{|T|}\}$ denote the ordered set of leaves with $|T|$ the number of leaves. 
\item[$\blacktriangleright$]  The set of incoming edges of a vertex $v$ is ordered from left to right. This set is denoted by $in(v):=\{e_{1}(v),\ldots, e_{|v|}(v)\}$ with $|v|$ the number of incoming edges. Moreover, the unique output edge of $v$ is denoted by $e_{0}(v)$.
\item[$\blacktriangleright$] The vertices with no incoming edge are called \textit{univalent vertices} whereas the vertices with only one input are called \textit{bivalent vertices}.
\end{itemize}
Let $S$ be a set. A \textit{planar $S$-tree} is a pair $(T\,;\,f)$ where $T$ is a planar tree and $f:E(T)\rightarrow S$ is a map indexing the edges of $T$ by elements in $S$. If there is no ambiguity about $f$, we will denote by $e_{i}(v)$ the element in $S$ indexing the edge $e_{i}(v)$. 
 \vspace{-15pt}
\begin{center}
\begin{figure}[!h]
\begin{center}
\includegraphics[scale=0.67]{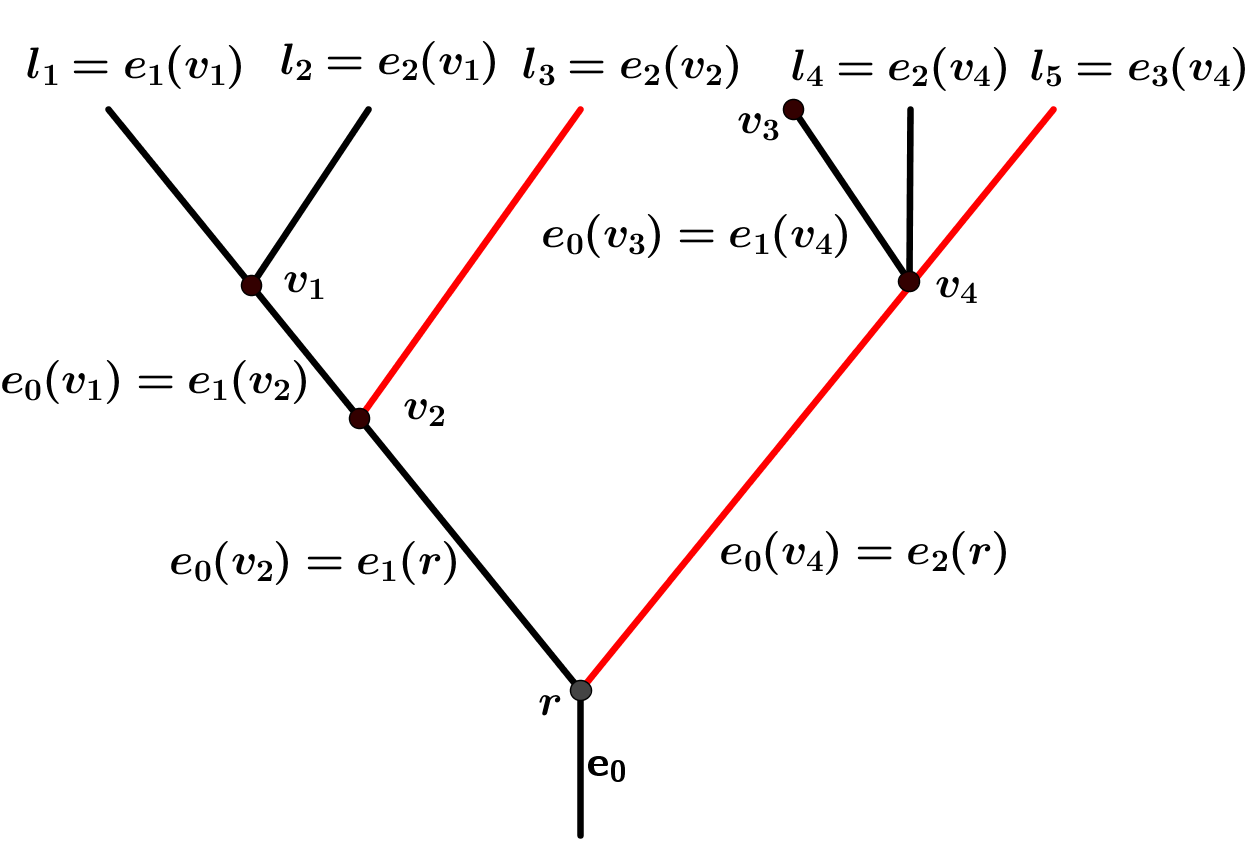}\vspace{-10pt}
\caption{Example of a planar $S$-tree with $S=\{o\,;\,c\}$.}\vspace{-7pt}
\end{center}
\end{figure}
\end{center}

\noindent The \textit{automorphism group} $Aut(T)$, associated to a planar $S$-tree $T$, can be described by induction on the number of vertices. If $|V(T)|=1$, then $Aut(T)$ is the group $\Sigma_{|T|}$. Otherwise, up to non-planar isomorphism, $T$ is of the form
\begin{equation}\label{B2}
T=t_{n}(T_{1}^{1},\ldots,T_{n_{1}}^{1},T_{1}^{2},\ldots,T_{n_{2}}^{2},\ldots, T_{1}^{l},\ldots,T_{n_{l}}^{l}), 
\end{equation}
where $t_{n}$ is a $n$-corolla, the uncoloured trees induced by $T_{1}^{i},\ldots,T_{n_{i}}^{i}$ are the same tree $T^{i}$ and $T^{i}$ is not isomorphic to $T^{j}$ if $i\neq j$. Since $\Sigma_{n_{i}}$ acts on the product $Aut(T^{i})^{\times n_{i}}$ by permuting the factors, the automorphism group of $T$ is a semi-direct product:
\begin{equation}\label{B0}
Aut(T)\cong \big(\, Aut(T^{1})^{\times n_{1}}\times \cdots \times Aut(T^{l})^{\times n_{l}}\big)\rtimes \big(\Sigma_{n_{1}}\times \cdots \Sigma_{n_{l}}\big):= \Gamma_{T}\rtimes \Sigma_{T}.
\end{equation}

An $S$-tree is a triplet $(T\,;\,f\,;\,\sigma)$ where $(T\,;\,f)$ is a planar $S$-tree and $\sigma:\{1,\ldots, |T|\}\rightarrow in(T)$ is a bijection labelling the leaves of $T$. Such an element will be denoted by $T$ if there is no ambiguity about the indexation $f$ and the bijection $\sigma$. We denote by \textbf{$S$-stree} the set of $S$-trees.  The bijection $\sigma$ can be interpreted as an element in the symmetric group $\Sigma_{|T|}$. 
\end{defi}

\subsection{The free ($P$-$Q$) bimodule functor}

\begin{defi}\label{b6}
In the following we introduce the set of trees used to build the free bimodule functor:
\begin{itemize}[leftmargin=*]
\item[$\blacktriangleright$] The \textit{join} $j(v_{1}\,;\,v_{2})$ of two vertices $v_{1}$ and $v_{2}$ is the first common vertex shared by the two paths joining $v_{1}$ and $v_{2}$ to the trunk. If $j(v_{1}\,;\,v_{2})=r$, then $v_{1}$ and $v_{2}$ are said to be \textit{connected to the root} and if $j(v_{1}\,;\,v_{2})\in \{ v_{1};v_{2}\}$, then they are said to be \textit{connected}. In Figure \ref{a9} the vertices
$v_{1}$ and $p_{1}$ are connected whereas the vertices $v_{1}$ and $p_{3}$ are connected to the root.
\item[$\blacktriangleright$] Let $\,d:V(T)\times V(T) \rightarrow \mathbb{N}$ be the distance defined as follows. The integer $d(v_{1}\,;\,v_{2})$ is the number of edges connecting  $v_{1}$ to $v_{2}$ if they are connected, otherwise $d(v_{1}\,;\,v_{2})=d(v_{1}\,;\,v_{3})+d(v_{3}\,;\,v_{2})$ with $v_{3}=j(v_{1}\,;\,v_{2})$. 
In Figure \ref{a9}, $d(v_{1}\,;\,r)=2\,\,$, $\,d(v_{1}\,;\,v_{2})=4$ and $d(v_{1}\,;\,p_{1})=1$.
\item[$\blacktriangleright$]  A \textit{reduced $S$-tree with section}  is a pair $(T,V^{p}(T))$ with $T$ in  \textbf{$S$-tree} and $V^{p}(T)$ a subset of $V(T)$, called the set of pearls, such that each path connecting a leaf or a univalent vertex to the trunk passes through a unique pearl. Furthermore, there is the following condition on the set of pearls:
\begin{equation}\label{g9}
\forall v\in V(T)\setminus V^{p}(T),\, \forall p\in V^{p}(T),\,\, j(v\,;\,p)\in \{v\,;\,p\}\Rightarrow d(v\,;\,p)=1.
\end{equation} 
\end{itemize}
The set of reduced $S$-trees with section is denoted by \textbf{$S$-rstree}. By definition, the set of pearls forms a section cutting the tree into two parts. We denote by $V^{u}(T)$ the set of vertices above the section and by $V^{d}(T)$ the one below the section. Given an integer $k\geq 1$, $\textbf{$S$-rstree}[k]$ is the set of reduced $S$-trees with section having at most $k$ leaves and such that each pearl has at most $k$ incoming edges. For instance, the reduced $S$-tree with section below is an element in the set $\textbf{$S$-rstree}[8]$.
\begin{figure}[!h]
\begin{center}
\includegraphics[scale=0.6]{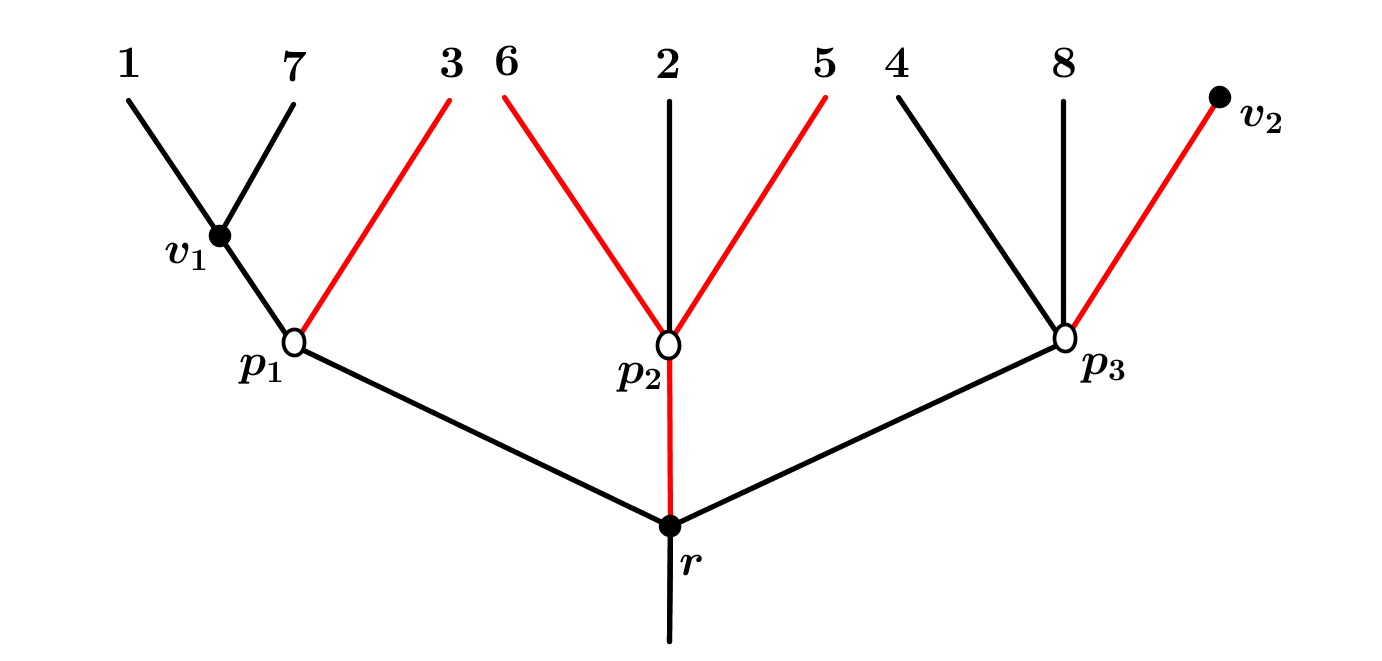}\vspace{-5pt}
\caption{A reduced $S$-tree with section with $S=\{o\,;\,c\}$.}\label{a9}\vspace{-15pt}
\end{center}
\end{figure}
\end{defi}

\begin{const}\label{b5}
From an $S$-sequence $M\in Seq(S)_{P_{0}}$, we build the ($P$-$Q$) bimodule $\mathcal{F}_{B}(M)$. The points are equivalence classes of pairs $[T\,;\,\{a_{v}\}]$ with $T\in S\textbf{-rstree}$ and $\{a_{v}\}$ a family of points labelling the vertices of $T$. The pearls are labelled by points in $M$, the vertices in $V^{u}(T)$ are labelled by points in $Q$ and the vertices in $V^{d}(T)$ are labelled by points in the operad $P$. More precisely, $\mathcal{F}_{B}(M)$ is given by the coproduct
$$
\left.\coprod\limits_{T\in S\textbf{\text{-rstree}}} \,\prod\limits_{v\in V^{d}(T)}\, P(e_{1}(v),\ldots,e_{|v|}(v);e_{0}(v)) \times \prod\limits_{v\in V^{p}(T)}\, M(e_{1}(v),\ldots,e_{|v|}(v);e_{0}(v))\times \prod\limits_{v\in V^{u}(T)}\, Q(e_{1}(v),\ldots,e_{|v|}(v);e_{0}(v))\right/\sim
$$
where $\sim$ is the equivalence relation generated by the following axioms:
\begin{itemize}[leftmargin=*]
\item[$i)$] If a vertex is indexed by a distinguished point $\ast_{s}$ in $P$ or $Q$, then \vspace{-5pt}
\begin{figure}[!h]
\begin{center}
\includegraphics[scale=0.2]{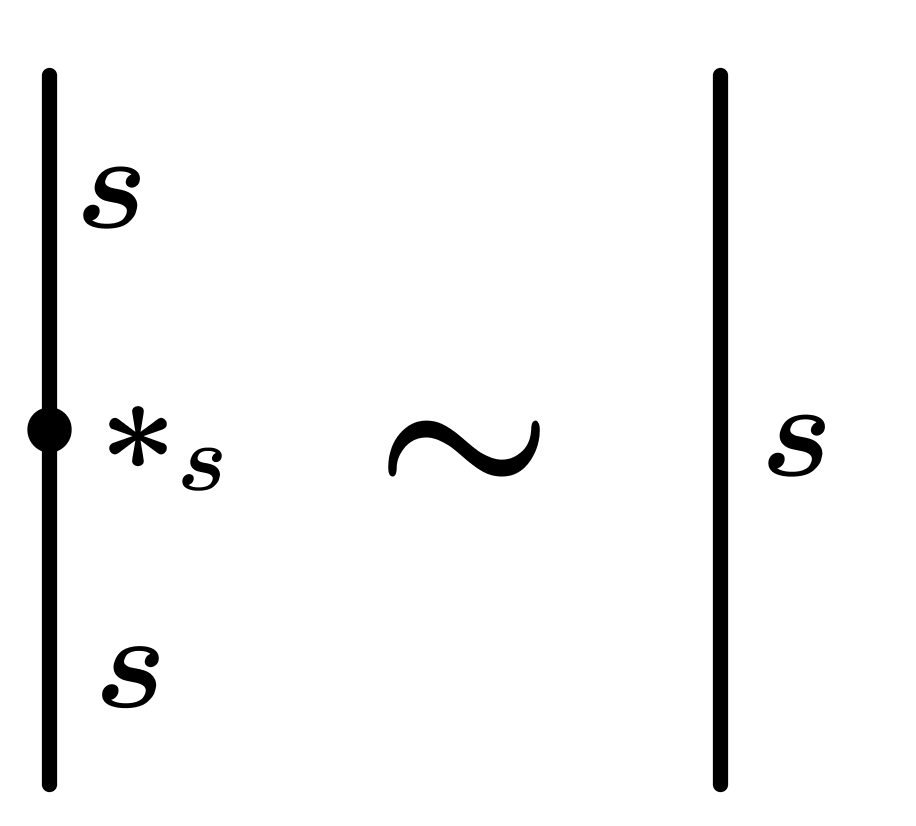}\vspace{-10pt}
\end{center}
\end{figure}
\item[$ii)$]If a vertex is indexed by $a\cdot \sigma$, with $\sigma\in \Sigma_{|v|}$, then \vspace{-5pt}
\begin{figure}[!h]
\begin{center}
\includegraphics[scale=0.35]{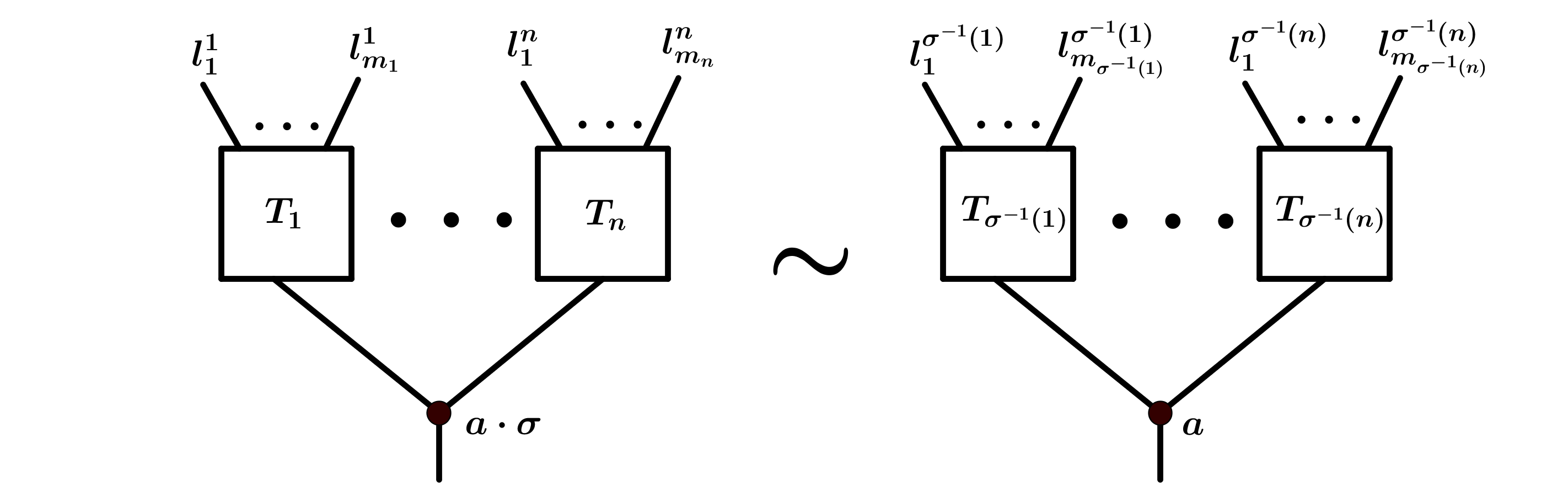}\vspace{-40pt}
\end{center}
\end{figure}
\newpage
\item[$iii)$] If a univalent pearl is indexed by a point of the form $\gamma_{s}(x)$, with $x\in P(\,;\,s)$, then we contract its output edge by using the operadic structure of $P$. In particular, if all the pearls are of the form $\gamma_{s_{v}}(x_{e_{i}(r)})$, then the point is identified with $\gamma_{s}(a_{r}(x_{e_{1}(r)},\ldots, x_{e_{|r|}(r)}))$ where $a_{r}\in P(e_{1}(r),\ldots, e_{|r|}(r)\,;\,s)$ labels the root.
\begin{figure}[!h]
\begin{center}
\includegraphics[scale=0.4]{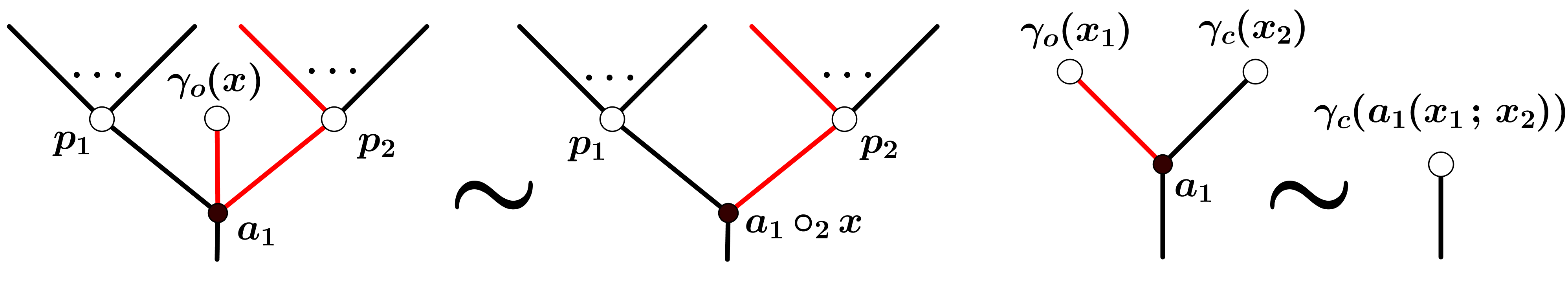}\vspace{-10pt}
\caption{Examples of the relation $(iii)$ for $S=\{o\,;\,c\}$.}\vspace{-10pt}
\end{center}
\end{figure}
\end{itemize}

\noindent A point $a\in Q(s_{1},\ldots,s_{n};s'_{i})$ is interpreted as an $n$-corolla (without pearl) whose incoming edges are indexed by $s_{1},\ldots, s_{n}$ respectively, the output by $s'_{i}$ and the vertex is labelled by the point $a$. If $[T\,;\,\{a_{v}\}]$ is a point in $\mathcal{F}_{B}(M)(s'_{1},\ldots,s'_{m};s_{m+1})$, then the composition $[T\,;\,\{a_{v}\}]\circ^{i}a$ consists in grafting the corolla labelled by $a$ to the $i$-th incoming edge of $T$. Then, we contract the inner edge so obtained if its target is not a pearl by using the operadic structure of $Q$.
\begin{figure}[!h]
\begin{center}
\includegraphics[scale=0.4]{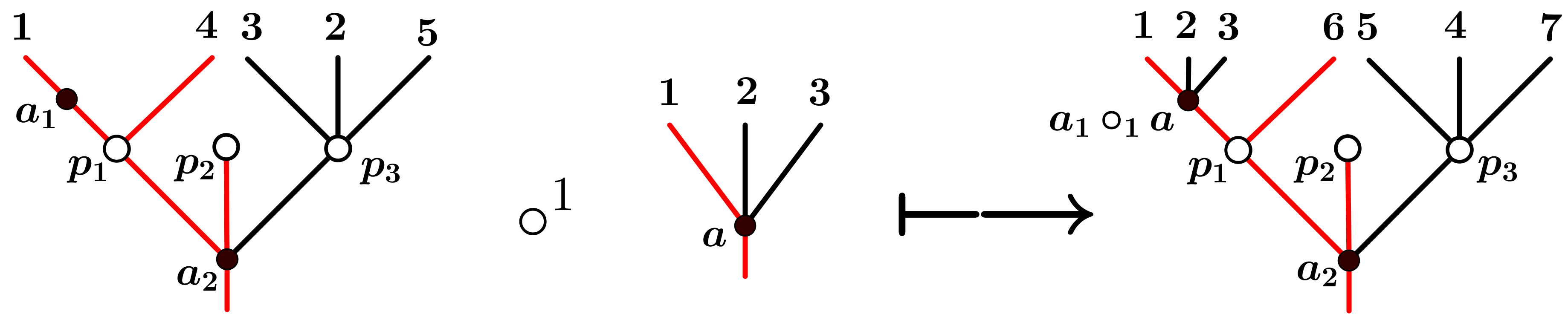}\vspace{-10pt}
\caption{The right module structure over $Q$.}\vspace{-10pt}
\end{center}
\end{figure}

\noindent Let $b\in P(s_{1},\ldots ,s_{n};s_{n+1})$ and $[T_{i}\,;\,\{a_{v}^{i}\}]$ be a family of points in $\mathcal{F}_{B}(M)(s_{1}^{i},\ldots,s_{m_{i}}^{i};s_{i})$. The left module operations over $P$ is defined as follows: each tree of the family is grafted to a leaf of the $n$-corolla from left to right. The inner edges obtained are contracted if their source are not pearl by using the operadic structure of $P$. Moreover, there is a morphism in the category $Seq(S)_{P_{0}}$,
$$
i:M\longrightarrow \mathcal{F}_{B}(M),
$$
sending $m\in M(s_{1},\ldots,s_{n};s_{n+1})$ to the point $[T\,;\,\{m\}]$ where $T$ is the pearl $n$-corolla labelled by $m$ whose leaves are indexed by the colours $s_{1},\ldots, s_{n}$ respectively and the trunk by $s_{n+1}$.
\end{const}

\begin{rmk}
Given an integer $k\geq 1$, the functor $T_{k}\mathcal{F}_{B}$ from $T_{k}Seq(S)$ to the category of $(k)$-truncated bimodules  is defined in the same way as the functor $\mathcal{F}_{B}$. We only have to consider the set $\textbf{$S$-rstree}[k]$ instead of the set of reduced trees with section in Construction \ref{b5}.
\end{rmk}

\begin{pro}\label{b3}
The functors $\mathcal{F}_{B}$ and $T_{k}\mathcal{F}_{B}$ are left adjoints to the forgetful functor $\mathcal{U}$:
\begin{equation}\label{b2}
\mathcal{F}_{B}:Seq(S)_{P_{0}}\leftrightarrows Bimod_{P\text{-}Q}:\mathcal{U} \hspace{15pt}\text{and}\hspace{15pt} T_{k}\mathcal{F}_{B}:T_{k}Seq(S)_{P_{0}}\leftrightarrows T_{k}Bimod_{P\text{-}Q}:\mathcal{U}.
\end{equation}
\end{pro}

\begin{proof}
Let $M'$ be a ($P$-$Q$) bimodule and $f:M\rightarrow M'$ be a morphism in the category $Seq(S)_{P_{0}}$. We have to prove that there exists a unique map of ($P$-$Q$) bimodules $\tilde{f}:\mathcal{F}_{B}(M)\rightarrow M'$ such that the following diagram commutes:
\begin{equation}\label{b1}
\xymatrix{
M \ar[r]^{f} \ar[d]_{i} & M' \\
\mathcal{F}_{B}(M)\ar@{-->}[ru]_{\exists\, !\, \tilde{f}} & 
}
\end{equation}
We define the map $\tilde{f}$ by induction on the cardinal of the set $nb(T)=V(T)\setminus V^{p}(T)$. Let $[(T\,;\,\sigma)\,;\,\{a_{v}\}]$ be a point in $\mathcal{F}_{B}(M)$ such that $|nb(T)|=0$ and $\sigma$ is the permutation indexing the leaves of $T$. By construction, $T$ is necessarily a pearl corolla with only one vertex labelled by $a_{r}\in M$. In order to have the commutative  diagram (\ref{b1}), the following equality has to be satisfied:
$$
\tilde{f}([(T\,;\,\sigma)\,;\,\{a_{v}\}])=f(a_{r})\cdot \sigma.
$$
\newpage
Let $[(T\,;\,\sigma)\,;\,\{a_{v}\}]$ be a point in $\mathcal{F}_{B}(M)$ where $T$ has only one vertex $v$ which is not a pearl. There are two cases to consider. If $v$ is the root of the tree $T$, then the root is labelled by a point $a_{v}\in P$ and $[(T\,;\,id)\,;\,\{a_{v}\}]$ has a decomposition of the form $a_{v}(\,[(T_{1}\,;\,id)\,;\,\{a_{1}\}],\ldots, [(T_{|v|}\,;\,id)\,;\,\{a_{|v|}\}]\,)$ where $T_{i}$ is a pearl corolla labelled by $a_{i}\in M$. Since $\tilde{f}$ has to be a ($P$-$Q$) bimodule map, there is the equality
$$
\tilde{f}([(T\,;\,\sigma)\,;\,\{a_{v}\}])=a_{v}\big(\, f(a_{1}),\ldots, f(a_{|v|})\,\big)\cdot \sigma.
$$
If the root is a pearl, then there exists a unique inner edge $e$ such that $s(e)=v$ and $t(e)=r$. So, the point $[(T\,;\,id)\,;\,\{a_{v}\}]$ has a decomposition on the form $[(T_{1}\,;\,id)\,;\,\{a_{t(e)}\}]\circ^{i}a_{s(e)}$ with $a_{s(e)}\in Q$ and $a_{t(e)}\in M$. Since $\tilde{f}$ has to be a ($P$-$Q$) bimodule map, there is the equality
$$
\tilde{f}([(T\,;\,\sigma)\,;\,\{a_{v}\}])= \big(\,f(a_{t(e)})\circ^{i}a_{s(e)}\,\big)\cdot\sigma.
$$

Assume $\tilde{f}$ has been defined for $|nb(T)|\leq n$. Let $[(T\,;\,\sigma)\,;\,\{a_{v}\}]$ be a point in $\mathcal{F}_{B}(M)$ such that $|nb(T)|= n+1$. By definition, there is an inner edge $e$ whose target vertex is a pearl. So, the point $[(T\,;\,id)\,;\,\{a_{v}\}]$ has a decomposition of the form $[(T_{1}\,;\,id)\,;\,\{a_{v}\}\setminus \{a_{s(e)}\} ]\circ^{i} a_{s(e)}$ where $T_{1}$ is a planar $S$-tree with section such that $|nb(T_{1})|= n$. Since $\tilde{f}$ has to be a ($P$-$Q$) bimodule map, there is the equality
$$
\tilde{f}([(T\,;\,\sigma)\,;\,\{a_{v}\}])=\big(\,\tilde{f}([(T_{1}\,;\,id)\,;\,\{a_{v}\}\setminus \{a_{s(e)}\}])\circ^{i}a_{s(e)}\,\big)\cdot\sigma.
$$ 
Due to the ($P$-$Q$) bimodule axioms, $\tilde{f}$ does not depend on the choice of the decomposition and $\tilde{f}$ is a ($P$-$Q$) bimodule map. The uniqueness follows from the construction. Similarly, we can prove that the functor $T_{k}\mathcal{F}_{B}$ is the left adjoint to the forgetful functor. 
\end{proof}

\subsection{The Boardman-Vogt resolution for bimodules}

By convention, if $\mathcal{C}$ is a model category enriched over $Top$, then the derived mapping space $\mathcal{C}^{h}(A;B)$ is the space $\mathcal{C}(A^{c};B^{f})$ with $A^{c}$ a cofibrant replacement of $A$ and $B^{f}$ a fibrant replacement of $B$. Since every object is fibrant in $Bimod_{P\text{-}Q}$, it is sufficient to determine cofibrant replacements in order to compute derived mapping spaces. In this section, we give a functorial way to obtain such cofibrant replacements.

\begin{defi}\label{e6}
Let \textbf{$S$-stree} be the set of pairs $(T\,;\, V^{p}(T))$ where $T$ is an $S$-tree and $V^{p}(T)$ is a subset of $V(T)$, called the \textit{set of pearls}. Similarly to Definition \ref{b6}, each path connecting a leaf or univalent vertex to the trunk passes through a unique pearl. However, a tree in \textbf{$S$-stree} doesn't necessarily satisfy Condition (\ref{g9}). The set of pearls forms a section cutting the tree $T$ into two parts. We denote by $V^{u}(T)$ and $V^{d}(T)$ the set of vertices above and below the section respectively. Elements in \textbf{$S$-stree} are called \textit{$S$-trees with section}. Analogously to Definition \ref{C9}, one can talk about non-planar isomorphism for $S$-trees with section and more particularly about the automorphism group $Aut(T\,;\,V^{p}(T))$ associated to an element $(T\,;\,V^{p}(T))$.  
\begin{figure}[!h]
\begin{center}
\includegraphics[scale=0.38]{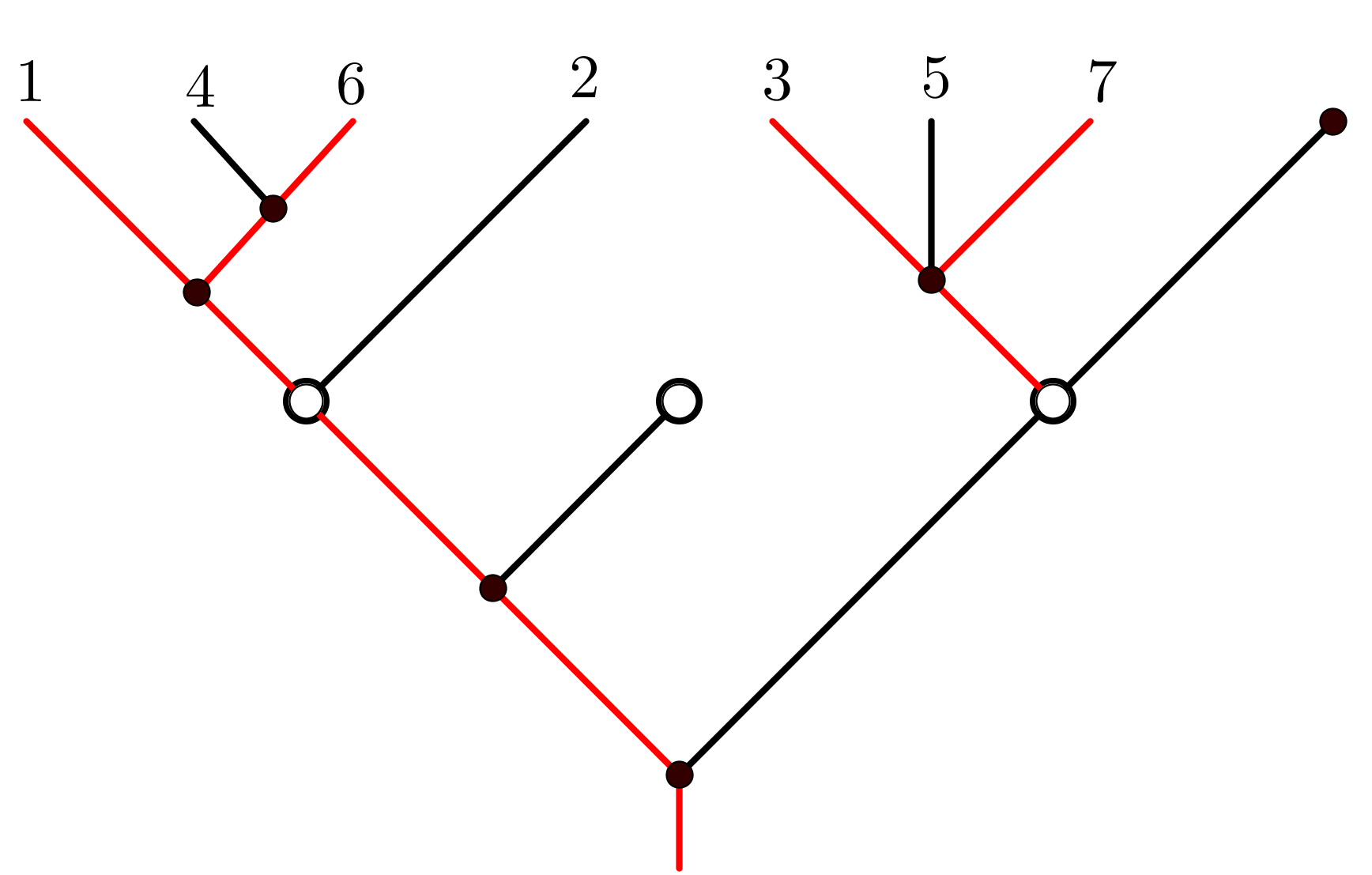}
\caption{An $S$-tree with section with $S=\{o\,;\,c\}$.}\vspace{-10pt}
\end{center}
\end{figure}
\end{defi}

\begin{const}\label{b7} 
Let $P$ and $Q$ be two $S$-operads. From a ($P$-$Q$) bimodule $M$, we build the ($P$-$Q$) bimodule $\mathcal{B}(M)$. The points are equivalence classes $[T\,;\, \{t_{v}\}\,;\, \{a_{v}\}]$ with $T\in S\textbf{-stree}$ and $\{a_{v}\}_{v\in V(T)}$ is a family of points labelling the vertices of $T$. The pearls are labelled by points in $M$ whereas the vertices in $V^{u}(T)$ (resp. the vertices in $V^{d}(T)$) are labelled by points in the operad $Q$ (resp. the operad $P$).  Furthermore, $\{t_{v}\}_{v\in V(T)\setminus V^{p}(T)}$ is a family of real numbers in the interval $[0\,,\,1]$ indexing the vertices which are not pearls. If $e$ is an inner edge above the section, then $t_{s(e)}\geq t_{t(e)}$. Similarly, if $e$ is an inner edge below the section, then $t_{s(e)}\leq t_{t(e)}$. In other words, $\mathcal{B}(M)$ is given by the quotient of the sub-$S$-sequence 
$$
\left.
\underset{\textbf{S\text{-stree}}}{\coprod}\,\underset{v\in V^{d}(T)}{\prod}\,\big[\,P(e_{1}(v),..,e_{|v|}(v);e_{0}(v))\times I\,\big]\times\!\!\!\!\underset{v\in V^{p}(T)}{\prod}\!M(e_{1}(v),..,e_{|v|}(v);e_{0}(v))\times\!\!\!\!\underset{ v\in V^{u}(T)}{\prod}\,\big[\,Q(e_{1}(v),..,e_{|v|}(v);e_{0}(v))\times I\big]
\right/ \!\sim 
$$
coming from the restrictions on the families $\{t_{v}\}$. The equivalence relation is generated by the following:

\begin{itemize}[ topsep=0pt, leftmargin=*]
\item[$i)$] If a vertex is labelled by a distinguished point $\ast_{s}$ in $P$ or $Q$, then 
\begin{center}
\includegraphics[scale=0.2]{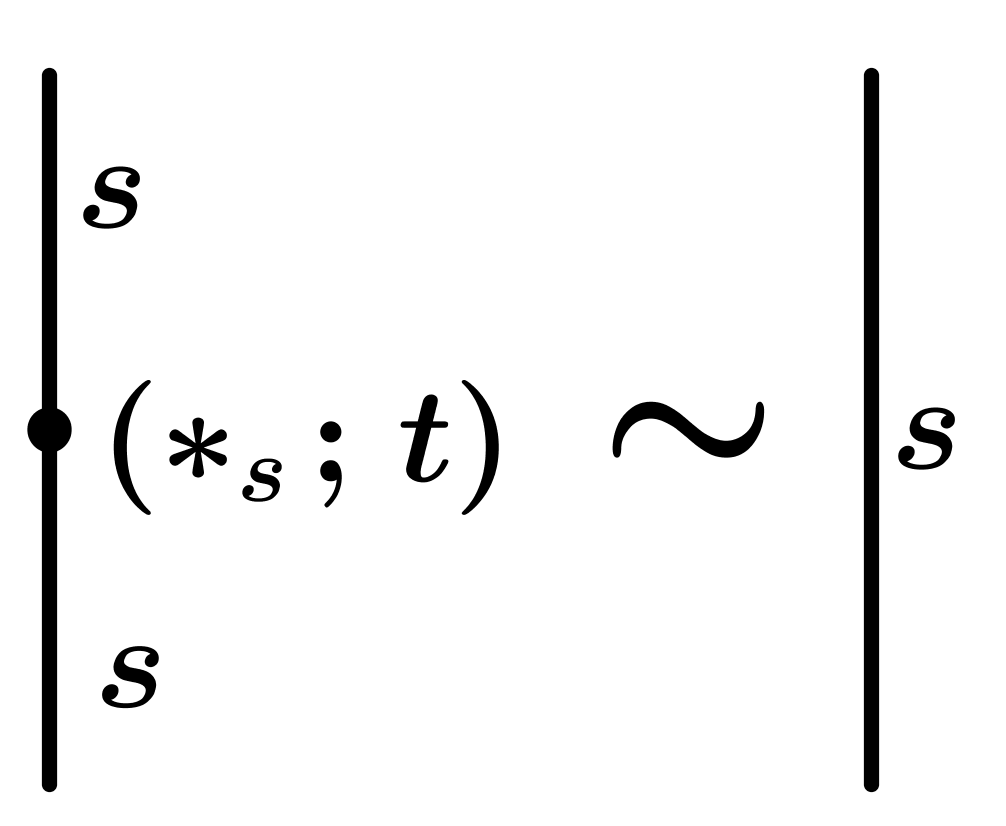}\vspace{-2pt}
\end{center}
\item[$ii)$] If a vertex is labelled by $a\cdot \sigma$, with $\sigma\in \Sigma_{|v|}$, then 
\begin{center}
\includegraphics[scale=0.35]{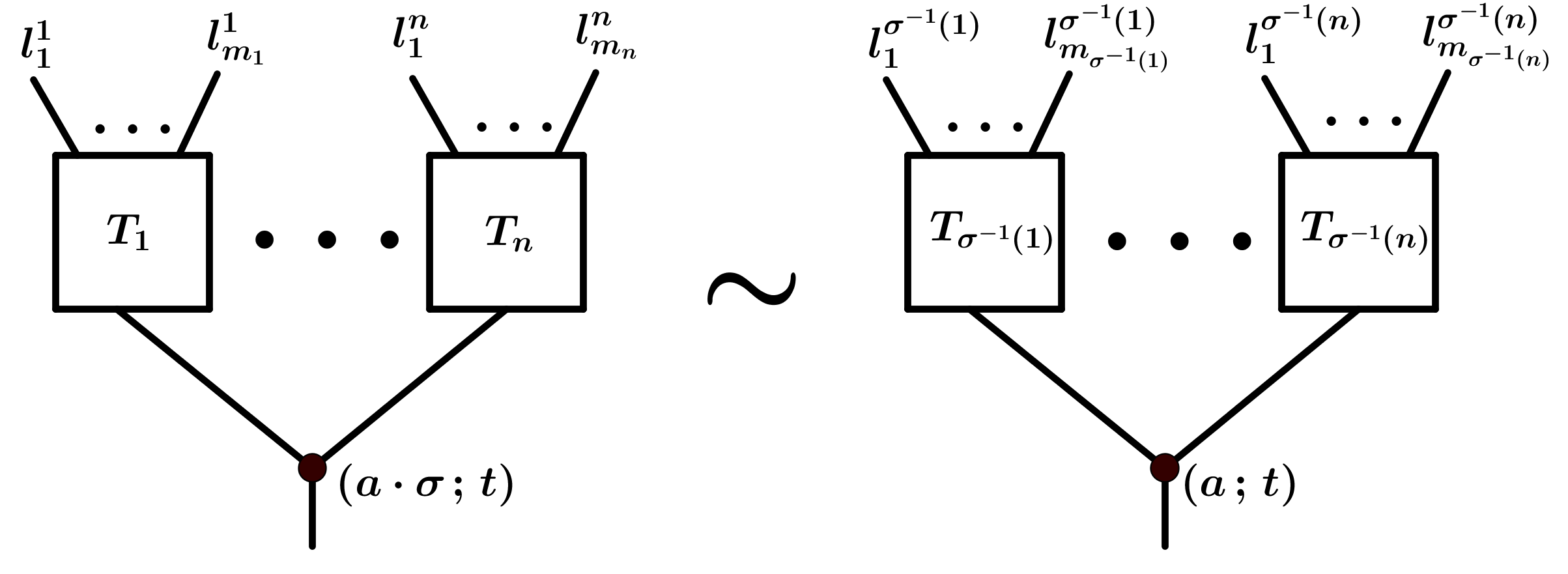}\vspace{-5pt}
\end{center}

\item[$iii)$] If an univalent pearl is indexed by a point of the form $\gamma_{s}(x)$, with $x\in P(\,;\,s)$, then we contract its output edge by using the operadic structure of $P$. In particular, if all the pearls connected to a vertex $v$ are univalent and of the form $\gamma_{s_{v}}(x_{v})$, then the vertex is identified to the pearl corolla with no input.
\begin{figure}[!h]
\begin{center}
\includegraphics[scale=0.4]{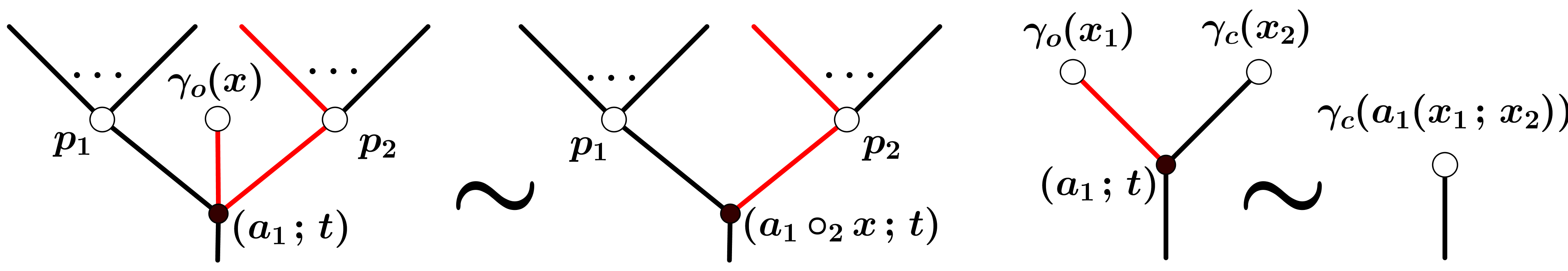}
\caption{Example of the relation $(iii)$ for $S=\{o\,;\,c\}$.}\vspace{-15pt}
\end{center}
\end{figure}
\item[$iv)$] If two consecutive vertices, connected by an edge $e$, are indexed by the same real number $t\in [0\,,\,1]$, then $e$ is contracted by using the operadic structures. The vertex so obtained is indexed by the real number $t$.\vspace{-5pt}

\begin{figure}[!h]
\begin{center}
\includegraphics[scale=0.48]{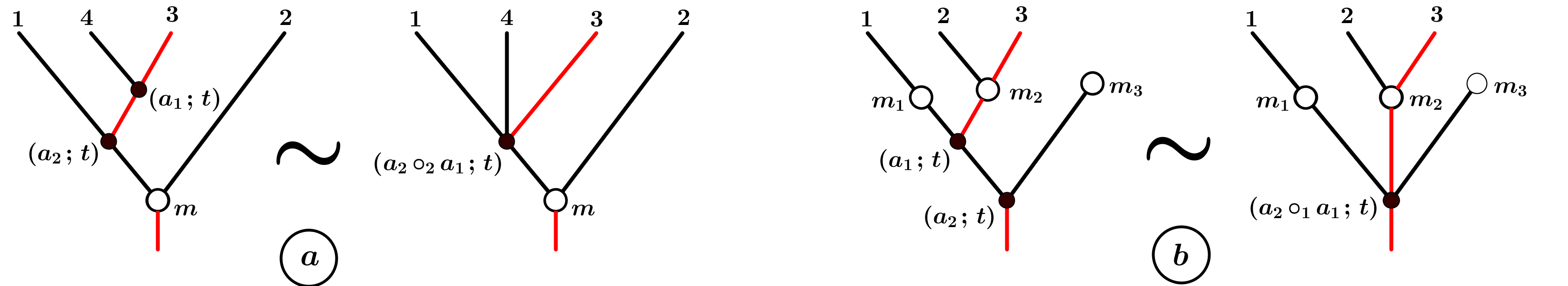}\vspace{-5pt}
\caption{Examples of the relation $(iii)$.}\vspace{-15pt}
\end{center}
\end{figure}

\item[$v)$]  If a vertex above the section is indexed by $0$, then its output edge is contracted by using the right module structures. Similarly, if a vertex below the section is indexed by $0$ then all its incoming edges are contracted by using the left module structure.
\begin{figure}[!h]
\begin{center}
\includegraphics[scale=0.48]{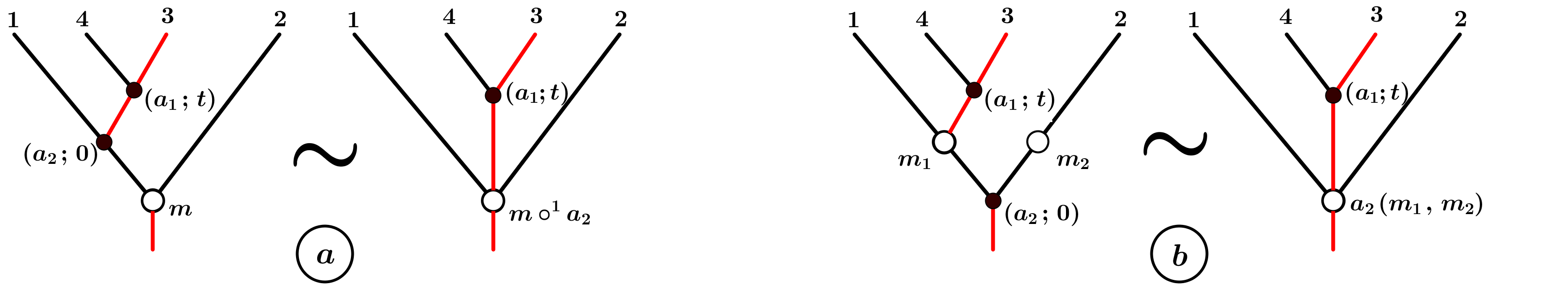}\vspace{-5pt}
\caption{Examples of the relation $(iv)$.}\vspace{-35pt}
\end{center}
\end{figure}
\end{itemize}

\newpage

\noindent Let us describe the ($P$-$Q$) bimodule structure. Let $a\in Q(s_{1},\ldots,s_{n};s'_{i})$ and $[T\,;\,\{t_{v}\}\,;\,\{a_{v}\}]$ be a point in $\mathcal{B}(M)(s'_{1},\ldots,s'_{m};s_{m+1})$. The composition $[T\,;\,\{t_{v}\}\,;\,\{a_{v}\}]\circ^{i}a$ consists in grafting the $n$-corolla labelled by $a$ to the $i$-th incoming edge of $T$ and indexing the new vertex by $1$. Similarly, let $b\in P(s_{1},\ldots,s_{n};s_{n+1})$ and $[T^{i}\,;\,\{t^{i}_{v}\}\,;\,\{a^{i}_{v}\}]$ be a family of points in the spaces $\mathcal{B}(M)(s^{i}_{1},\ldots,s^{i}_{n_{i}};s_{i})$. The left module structure over $P$ is defined as follows: each tree of the family is grafted to a leaf of the $n$-corolla labelled by $b$ from left to right. The new vertex, arising from the $n$-corolla, is indexed by $1$.\vspace{-7pt}
\begin{center}
\begin{figure}[!h]
\begin{center}
\includegraphics[scale=0.4]{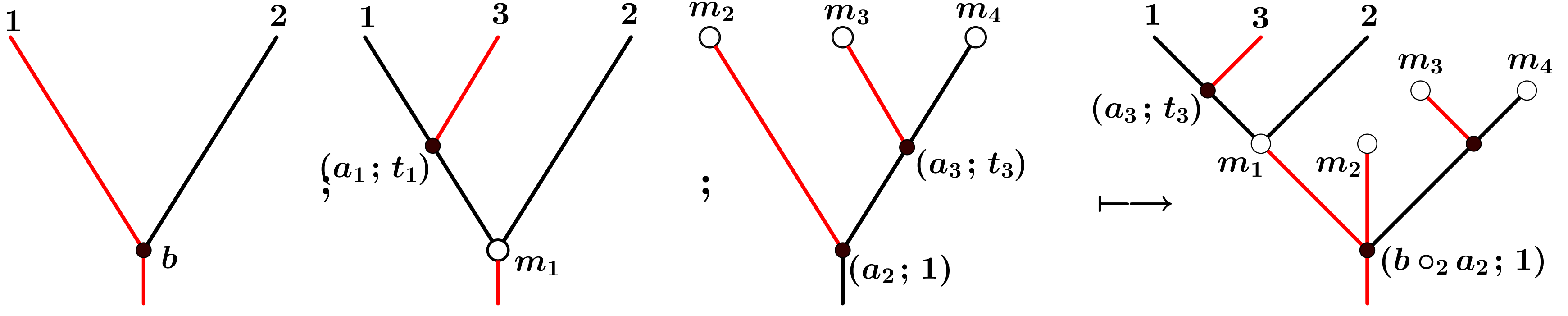}
\caption{Illustration of the left module structure.}\vspace{-7pt}
\end{center}
\end{figure}
\end{center}
\end{const}

\noindent Let us recall that the free ($P$-$Q$) bimodule $\mathcal{F}_{B}(M)$ is the space of equivalence classes $[T\,;\,\{a_{v}\}]$ with $T\in S\textbf{-rstree}$. Since \textbf{$S$-rstree} is a subset of \textbf{$S$-stree}, there is a map
\begin{equation}\label{b8}
\tau:\mathcal{F}_{B}(M)\rightarrow\mathcal{B}(M)\,\,;\,\, [T\,;\,\{a_{v}\}]\mapsto [T\,;\,\{1_{v}\}\,;\, \{a_{v}\}],
\end{equation}
indexing the vertices in $V^{u}(T)$ and $V^{d}(T)$ by $1$. Due to axioms $(iii)$ and $(iv)$ of Construction \ref{b7}, $\tau$ is a  ($P$-$Q$) bimodule map. Furthermore, one has the map
\begin{equation}\label{b9}
\mu:\mathcal{B}(M)\rightarrow M\,\,;\,\, [T\,;\,\{t_{v}\}\,;\,\{a_{v}\}]\mapsto [T\,;\,\{0_{v}\}\,;\, \{a_{v}\}],
\end{equation}
sending the real numbers indexing the vertices to $0$. The element so obtained is a pearl corolla labelled by a point in $M$. Due to axioms $(iv)$ and $(v)$ of Construction \ref{b7}, $\mu$ is a ($P$-$Q$) bimodule map.

\subsection{Cofibrant replacements for $k$-truncated bimodules}\label{B6}

In this section, $P$ and $Q$ are two $S$-operads whereas $M$ is a ($P$-$Q$) bimodule. In order to show that the Boardman-Vogt resolution $\mathcal{B}(M)$ is a cofibrant replacement of $M$, we introduce a filtration according to the number of geometrical inputs which is the number of leaves plus the number of univalent vertices above the section. As we will see, this filtration also produces cofibrant replacements for the truncated bimodules $T_{k}(M)$ with $k\geq 1$.

A point in the bimodule $\mathcal{B}(M)$ is said to be \textit{prime} if the real numbers labelling its vertices are strictly smaller than $1$. Besides, a point is said to be \textit{composite} if one of its vertex is labelled by $1$. A composite point can be decomposed into \textit{prime components} as shows the picture below. More precisely, the prime components of a point indexing by a planar $S$-tree with section are obtained by forgetting the vertices indexing by $1$. Otherwise, the prime components of a point of the form  $[(T\,;\,\sigma)\,;\,\{t_{v}\}\,;\,\{a_{v}\}]$ coincide with the prime components of $[(T\,;\,id)\,;\,\{t_{v}\}\,;\,\{a_{v}\}]$.
\begin{figure}[!h]
\begin{center}
\includegraphics[scale=0.38]{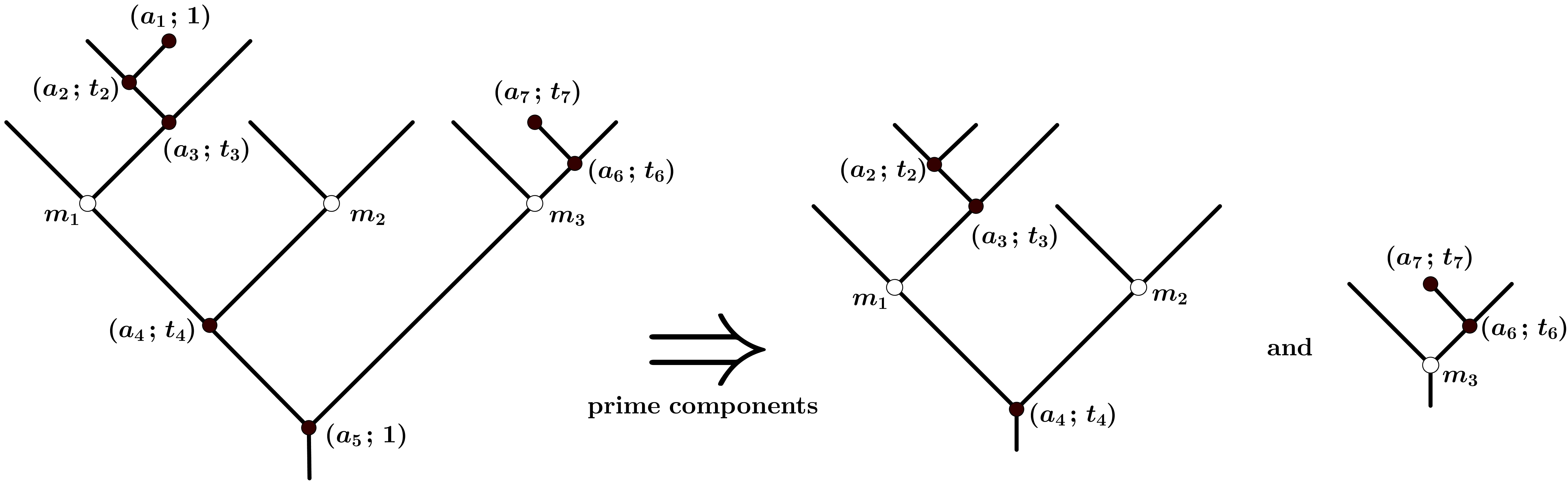}
\caption{A composite point and its prime components.}\label{A4}\vspace{-15pt}
\end{center}
\end{figure}
 
\noindent A prime point is in the $k$-th filtration term $\mathcal{B}_{k}(M)$ if the number of its geometrical inputs is smaller than $k$. Similarly, a composite point is in the $k$-th filtration term if all its prime components are in the $k$-th filtration term. For instance, the composite point in Figure \ref{A4} is in the $6$-th filtration term. For each $k\geq 0$, $\mathcal{B}_{k}(M)$ is a ($P$-$Q$) bimodule and they produce the following filtration of $\mathcal{B}(M)$:
\begin{equation}\label{A9}
\xymatrix{
P_{0}\ar[r] & \mathcal{B}_{0}(M) \ar[r] & \mathcal{B}_{1}(M)\ar[r] & \cdots \ar[r] & \mathcal{B}_{k-1}(M) \ar[r] & \mathcal{B}_{k}(M) \ar[r] & \cdots \ar[r] & \mathcal{B}(M)
}
\end{equation}

\begin{notat}
A pointed $S$-sequence $M$ is said to be well pointed if the maps $\ast_{s}\rightarrow M(s\,;\,s)$ are cofibrations. Furthermore, a bimodule or operadic map $f$ is said to be $\Sigma$-cofibrant if $\mathcal{U}(f)$ is cofibrant in the category of $S$-sequences. We also recall that the category of spaces together with a right action of a group $G$, denoted by $G\text{-}Top$, is endowed with a model category structure coming from the adjunction $G[-]:Top\leftrightarrows G\text{-}Top:\mathcal{U}$ where $G[-]$ sends the space $X$ to $G[X]=\coprod_{G}X$. By convention, a map in $G\text{-}Top$ is called a $G$-equivariant map whereas a $G$-cofibration is a cofibration in $G\text{-}Top$.   
\end{notat}

\begin{thm}\label{B7}
Assume that the maps $\gamma_{s}$ are cofibrations, the $S$-sequences $M$, $P$ and $Q$ are $\Sigma$-cofibrant and the operads $P$ and $Q$ are well pointed. Then, the objects $\mathcal{B}(M)$ and $T_{k}(\mathcal{B}_{k}(M))$ are cofibrant replacement of $M$ and $T_{k}(M)$ in the categories $Bimod_{P\text{-}Q}$ and $T_{k}Bimod_{P\text{-}Q}$ respectively. 
\end{thm}

In the first version of the paper, we show that the maps (\ref{b8}) and (\ref{b9}) are respectively a cofibration and a weak equivalence in the category of ($P$-$Q$) bimodules, under the assumptions of the theorem. Nevertheless, we give an alternative proof since we need to show that the filtration (\ref{A9}) is composed of cofibrations in the category of ($P$-$Q$) bimodules. In what follows, the method used is based on the paper of Turchin \cite[Section 11]{Tourtchine10} as well as the paper of Berger and Moerdijk \cite{Berger06}. In particular, we use the following two statements which are special cases of \cite[Lemma 2.5.3]{Berger06} and \cite[Lemma 2.5.2]{Berger06} respectively.

\begin{lmm}\label{B3}
Let $1\rightarrow G_{1} \rightarrow G_{1}\rtimes G_{2}\rightarrow G_{2}\rightarrow 1$ be a short exact sequence of groups. Let $A\rightarrow B$ be a $G_{2}$-cofibration and $X\rightarrow Y$ be a $G_{1}\rtimes G_{2}$-equivariant $G_{1}$-cofibration. The pushout product map $(A\times Y)\cup_{A\times X}(B\times X)\rightarrow B\times Y$ is a $G_{1}\rtimes G_{2}$-cofibration.
\end{lmm}

\begin{lmm}\label{B4}
Let $G$ be a group. Let $A\rightarrow B$ and $X\rightarrow Y$ be two maps in $G\text{-}Top$ which are cofibrations as continuous map. If one of them is cofibrant in  $G\text{-}Top$ then the pushout product map $(A\times Y)\cup_{A\times X}(B\times X)\rightarrow B\times Y$ is a cofibration in $G\text{-}Top$. Moreover the latter is acyclic if $A\rightarrow B$ or $X\rightarrow Y$ is.
\end{lmm}

\begin{proof}[Proof of Theorem \ref{B7}]
First, we show that the map $\mu:\mathcal{B}(M)\rightarrow M$, which sends the real numbers indexing the vertices to $0$, is a weak equivalence in the category of ($P$-$Q$) bimodules. Indeed, the map $\mu$ is a homotopy equivalence in the category of $S$-sequences where the homotopy consists in bringing the parameters to $0$. Unfortunately, we cannot do the same for the map $\mu_{k}:T_{k}(\mathcal{B}_{k}(M))\rightarrow T_{k}(M)$ since the previous homotopy doesn't necessarily preserve the number of geometrical inputs. For instance, the composite point $[T\,;\,\{t_{v}\}\,;\, \{a_{v}\}]$ in Figure \ref{A4} is in the $6$-th filtration term whereas the point $[T\,;\,\{tt_{v}\}\,;\,\{a_{v}\} ]$, with $t\in ]0\,,\,1[$, is in the $9$-th filtration term. To solve this problem, we start by contracting the output edge of univalent vertices. For this, we use the homotopy
$$
\begin{array}{crcl}\vspace{4pt}
C_{k}: & T_{k}(\mathcal{B}_{k}(M))\times [0\,,\,1] & \longrightarrow & T_{k}(\mathcal{B}_{k}(M)); \\ 
 & [T\,;\,\{t_{v}\}\,;\,\{a_{v}\} ]\,;\, t & \longmapsto & [T\,;\,\{tt_{D(v)}+(1-t)t_{v}\}\,;\, \{a_{v}\}],
\end{array} 
$$ 
where $D(v)=v$ for $v\notin V^{u}(T)$. Otherwise, $D(v)$ is the first vertex in the path joining $v$ to its pearl such that $D(v)$ is connected to a leaf. By convention, if such a vertex doesn't exist, then $D(v)$ is the pearl and $t_{D(v)}$ is fixed to $0$. So, the $k$-truncated $S$-sequence $T_{k}(\mathcal{B}_{k}(M))$ is weakly equivalent to the sub-object formed by points without univalent vertices above the section. Then, we use the homotopy bringing the parameters to $0$ in order to conclude that $\mu_{k}$ is a weak equivalence.   

In order to show that the map from $\mathcal{B}_{k-1}(M)$ to $\mathcal{B}_{k}(M)$ is a cofibration, we introduce another filtration according to the number of vertices. A prime point is said to be in $\mathcal{B}_{k}(M)[l]$ if it has at most $k-1$ geometrical inputs or if it has exactly $k$ geometrical inputs and at most $l$ vertices. A composite point is said to be in $\mathcal{B}_{k}(M)[l]$ if all its prime components are in $\mathcal{B}_{k}(M)[l]$. Thus, we get a family of ($P$-$Q$) bimodule maps
$$
\xymatrix{
\mathcal{B}_{k-1}(M) \ar[r] & \mathcal{B}_{k}(M)[1]\ar[r] & \cdots \ar[r] & \mathcal{B}_{k}(M)[l-1] \ar[r] & \mathcal{B}_{k}(M)[l] \ar[r] & \cdots \ar[r] & \mathcal{B}_{k}(M).
}
$$
In particular, the map $P_{0}\rightarrow \mathcal{B}_{0}(M)[1]$ is a cofibration. Indeed, let $M_{0}$ be the $S$-sequence given by $M_{0}(\,;\,s)=M(\,;\,s)$, for $s\in S$, and the empty set otherwise. Due to the axiom $(iii)$ of Construction \ref{b7}, $P_{0}$ and $\mathcal{B}_{0}(M)[1]$ are the free ($P$-$Q$) bimodules $\mathcal{F}_{B}(P_{0})$ and $\mathcal{F}_{B}(M_{0})$ respectively. Consequently, the map from $P_{0}$ to $\mathcal{B}_{0}(M)[1]$ coincides with $\mathcal{F}_{B}(\{\gamma_{s}\}):\mathcal{F}_{B}(P_{0})\rightarrow \mathcal{F}_{B}(M_{0})$ which is a cofibration in the category of ($P$-$Q$) bimodules since the maps $\gamma_{s}$ are cofibrations in the category of topological spaces. 

\newpage

In the general case, in order to prove that the map from $\mathcal{B}_{k}(M)[l-1]$ to $\mathcal{B}_{k}(M)[l]$ is a cofibration in the category of ($P$-$Q$) bimodules, we consider the set $S\text{-}\textbf{\text{stree}[k\,;\,l]}$ of $S$-trees with section having exactly $k$ geometrical inputs and $l$ vertices. Let $X_{k}[l]$ be the quotient of the sub-$S$-sequence 
$$
\left.\underset{S\text{-}\textbf{\text{stree}[k\,;\,l]}}{\coprod}\,\,\underset{v\in V^{d}(T)}{\prod}\,\big[\,P(e_{1}(v),..,e_{|v|}(v);e_{0}(v))\!\times\! I\,\big]\,\times\!\!\!\underset{v\in V^{p}(T)}{\prod}\,M(e_{1}(v),..,e_{|v|}(v);e_{0}(v))\,\times\!\!\!\underset{ v\in V^{u}(T)}{\prod}\,\big[\,Q(e_{1}(v),..,e_{|v|}(v);e_{0}(v))\!\times\! I\big]\right/ \sim
$$ 
obtained by taking the restriction on the set of real numbers indexing the vertices. The equivalent relation is generated by the axiom $(ii)$ of Construction \ref{b7}. Similarly, $\partial X_{k}[l]$ is the $S$-sequence formed by points in $X_{k}[l]$ satisfying one of the following conditions called \textit{the boundary conditions}:
\begin{enumerate}
\item \label{R1} there is a vertex indexed by $0$ or $1$,
\item \label{R2} there are two consecutive vertices indexed by the same real number,
\item \label{R3} there is a bivalent vertex labelled by a distinguished points in $P$ or $Q$,
\item \label{R4} there is a univalent pearl labelled by a point of the form $\gamma_{s}(x)$, with $x\in P(\,;\,s)$. 
\end{enumerate}
The $S$-sequences $X_{k}[l]$ and $\partial X_{k}[l]$ are not objects in the category $Seq(S)_{P_{0}}$. To solve this problem, we consider the following $S$-sequences which are obviously endowed with maps from $P_{0}$:
$$
\tilde{X}_{k}[l](s_{1},\ldots,s_{n}\,;\,s_{n+1}):=
\left\{
\begin{array}{ll}\vspace{4pt}
X_{k}[l](\,;\,s_{n+1})\sqcup P(\,;\,s_{n+1}) & \text{if } n=0, \\ 
X_{k}[l](s_{1},\ldots,s_{n}\,;\,s_{n+1}) & \text{otherwise},
\end{array} 
\right.
$$

$$
\partial\tilde{X}_{k}[l](s_{1},\ldots,s_{n}\,;\,s_{n+1}):=
\left\{
\begin{array}{ll}\vspace{4pt}
\partial X_{k}[l](\,;\,s_{n+1})\sqcup P(\,;\,s_{n+1}) & \text{if } n=0, \\ 
\partial X_{k}[l](s_{1},\ldots,s_{n}\,;\,s_{n+1}) & \text{otherwise}.
\end{array} 
\right.
$$ 

\noindent Furthermore, there is the  pushout diagram 
\begin{equation}\label{E6}
\xymatrix{
\mathcal{F}_{B}(\partial\tilde{X}_{k}[l]) \ar[r] \ar[d] & \mathcal{F}_{B}(\tilde{X}_{k}[l])\ar[d]\\
\mathcal{B}_{k}(M)[l-1] \ar[r] & \mathcal{B}_{k}(M)[l]
}
\end{equation}
where the left vertical map consists in: contracting the output edge (resp. incoming edges) of vertices above the section (resp. below the section) indexed by $0$ ; contracting the inner edges connecting vertices indexed by the same real number ; forgetting the bivalent vertices labelled by distinguished points ; contracting the output edge of univalent pearl labelled by $\gamma_{s}(x)$ with $x\in P_{0}$ ; taking the inclusion for points having a vertex indexed by $1$. Since the functor $\mathcal{F}_{B}$ and the pushout diagrams preserve the cofibrations, $\mathcal{B}_{k-1}(M)\rightarrow \mathcal{B}_{k}(M)$ is a cofibration in the category of ($P$-$Q$) bimodules if the inclusion from $\partial X_{k}[l]$ to $X_{k}[l]$ is a $\Sigma$-cofibration.  

Let $\textbf{$S$-stree}^{p}\textbf{[k\,;\,l]}$ be the set of planar $S$-trees with section having $k$ geometrical inputs and $l$ vertices (by planar we mean without the bijection labelling the leaves). If $T\in \textbf{$S$-stree}^{p}\textbf{[k\,;\,l]}$, then $H(T)$ is the space of parametrizations of the set $V(T)\setminus V^{p}(T)$ by real numbers in the internal $[0\,,\,1]$ satisfying the restriction of Construction \ref{b7}. So, $H(T)$ is a sub-polytope of $[0\,,\,1]^{m}$, with $m=|V(T)\setminus V^{p}(T)|$, arising from a gluing of simplices. For instance, if $T$ has only bivalent vertices, then $H(T)=\Delta^{|V^{u}(T)|}\times \Delta^{|V^{d}(T)|}$. We denote by $H^{-}(T)$ the sub-polytope of $H(T)$ formed by points satisfying the axioms (\ref{R1}) or (\ref{R2}) of the boundary conditions. In other words, $H^{-}(T)$ is formed by faces of the polytope $H(T)$. Consequently, the inclusion (\ref{B1}) is a cofibration in the category of spaces and preserves the action of the automorphism group $Aut(T\,;\,V^{p}(T))$:
\begin{equation}\label{B1}
H^{-}(T)\longrightarrow H(T).
\end{equation}

Similarly, let $\underline{M}(T)$ be the space of indexations of $V(T)$ by points in $P$, $Q$ and $M$ satisfying the relation of Construction \ref{b7}. Since the objects $P$, $Q$ and $M$ are $\Sigma$-cofibrant, the space $\underline{M}(T)$ is $Aut(T\,;\,V^{p}(T))$-cofibrant. To show that, we adapt the proof introduced by Berger and Moerdijk in \cite{Berger06} for operads. We prove the result by induction on the set of vertices. Assume that $T$ is of the form (\ref{B2}), then there are two cases to consider. If the root is a pearl, then the trees $T^{i}$ are elements in the set $S\text{-}\textbf{tree}$ and $\underline{M}(T^{i})$ is the space of indexations of $V(T^{i})$ by points in the operad $Q$. In \cite{Berger06}, the authors show that $\underline{M}(T^{i})$ is $Aut(T^{i})$-cofibrant. Consequently, $\underline{M}(T^{1})^{\times k_{1}}\times \cdots \times \underline{M}(T^{j})^{\times k_{j}}$  is $\Gamma_{T}$-cofibrant and is equipped with an action of $\Gamma_{T}\rtimes \Sigma_{T}$. Since $\underline{M}(t_{n})=P(e_{1}(r),\ldots,e_{|r|}(r);e_{0}(r))$ is $\Sigma_{T}$-cofibrant, Lemma \ref{B3}, applied to the short exact sequence $1\rightarrow \Gamma_{T} \rightarrow \Gamma_{T}\rtimes \Sigma_{T}\rightarrow \Sigma_{T}\rightarrow 1$, shows that $\underline{M}(T)$ is $Aut(T\,;\,V^{p}(T))$-cofibrant.\vspace{-10pt}

\newpage

In the second case, the root is not a pearl and the trees $T^{i}$ are $S$-trees with section. By induction, $\underline{M}(T^{i})$ is $Aut(T^{i}\,;\,V^{p}(T^{i}))$-cofibrant. Consequently, $\underline{M}(T^{1})^{\times k_{1}}\times \cdots \times \underline{M}(T^{j})^{\times k_{j}}$  is $\Gamma_{T}$-cofibrant and is equipped with an action of $\Gamma_{T}\rtimes \Sigma_{T}$. Since $\underline{M}(t_{n})=M(e_{1}(r),\ldots,e_{|r|}(r);e_{0}(r))$ is $\Sigma_{T}$-cofibrant, Lemma \ref{B3}, applied to the short exact sequence $1\rightarrow \Gamma_{T} \rightarrow \Gamma_{T}\rtimes \Sigma_{T}\rightarrow \Sigma_{T}\rightarrow 1$, shows that $\underline{M}(T)$ is also $Aut(T\,;\,V^{p}(T))$-cofibrant.

The space $\underline{M}^{-}(T)$ is the subspace of $\underline{M}(T)$ formed by points satisfying the axioms (\ref{R3}) or (\ref{R4}) of the boundary conditions. An invariant form of the pushout product lemma \ref{B3} together with an induction on trees shows that the inclusion from $\underline{M}^{-}(T)$ to $\underline{M}(T)$ is an $Aut(T\,;\,V^{p}(T))$-cofibration since the operads $P$ and $Q$ are well pointed and the maps $\gamma_{s}$ are cofibrations. From now on, we denote by $(H\times \underline{M})^{-}(T)$ the pushout product
$$
(H^{-}(T)\times \underline{M}(T))\, \underset{H^{-}(T)\,\times\, \underline{M}^{-}(T)}{\coprod}\,(H(T)\times \underline{M}^{-}(T)).
$$ 
Lemma \ref{B4} implies that the inclusion from $(H\times \underline{M})^{-}(T)$ to $H(T)\times \underline{M}(T)$ is an $Aut(T\,;\,V^{p}(T))$-cofibration. An element in $Aut(T\,;\,V^{p}(T))$ can be associated to a permutation of the leaves of the tree $T$. Thus, the group $Aut(T\,;\,V^{p}(T))$ acts on $\Sigma_{|T|}$. Consequently, the following map is a $\Sigma_{|T|}$-cofibration:
\begin{equation}\label{A5}
(H\times \underline{M})^{-}(T) \underset{Aut(T\,;\,V^{p}(T))}{\times} \Sigma'_{|T|} \longrightarrow (H(T)\times \underline{M}(T)) \underset{Aut(T\,;\,V^{p}(T))}{\times} \Sigma_{|T|}.
\end{equation}
Hence, the horizontal maps of the following diagram are $\Sigma$-cofibrations:
$$
\xymatrix{
\underset{[T\,;\,V^{p}(T)]}{\coprod} (H\times \underline{M})^{-}(T) \underset{Aut(T\,;\,V^{p}(T))}{\times} \Sigma_{|T|}  \ar[r]\ar@{=}[d] & \underset{[T\,;\,V^{p}(T)]}{\coprod} (H(T)\times \underline{M}(T)) \underset{Aut(T\,;\,V^{p}(T))}{\times} \Sigma_{|T|}\ar@{=}[d] \\
\partial X_{k}[l] \ar[r] & X_{k}[l]
}
$$
where the disjoint union is along the isomorphism classes of planar $S$-trees with section in $\textbf{$S$-stree}^{p}\textbf{[k\,;\,l]}$. Finally, the map $\mathcal{B}_{k-1}(M)\rightarrow \mathcal{B}_{k}(M)$ is a cofibration in the category of ($P$-$Q$) bimodules. In the same way, we can check that the truncated bimodule $T_{k}(\mathcal{B}_{k}(M))$ is cofibrant using the functor $T_{k}\mathcal{F}_{B}$ instead of $\mathcal{F}_{B}$ in the above arguments.
\end{proof}

\begin{rmk}
The Boardman-Vogt resolution induces an endofunctor in the category $Bimod_{P\text{-}Q}$ or $T_{k}Bimod_{P\text{-}Q}$. Note that this construction is also functorial in the operads $P$ and $Q$. Indeed, let $f_{p}:P_{1}\rightarrow P_{2}$ and $f_{q}:Q_{1}\rightarrow Q_{2}$ be two maps of $S$-operads. If $f$ is a map from a ($P_{1}$-$Q_{1}$) bimodule $M_{1}$ to a ($P_{2}$-$Q_{2}$) bimodule $M_{2}$ such that the following diagrams commute:
$$
\xymatrix{
P_{1}\,\times\, M_{1}^{1}\times \cdots \times M_{1}^{k} \ar[r]^{\,\,\,\,\,\,\,\,\,\,\,\,\,\,\,\,\,\,\,\,\,\,\,\,\,\,\,\,\gamma_{l}} \ar[d]^{f_{p}\times f\times \cdots \times f} & M_{1} \ar[d]^{f} \\
P_{2}\,\times\, M_{2}^{1}\times \cdots \times M_{2}^{k} \ar[r]_{\,\,\,\,\,\,\,\,\,\,\,\,\,\,\,\,\,\,\,\,\,\,\,\,\,\,\,\,\gamma_{l}} & M_{2}
}
\,\,\,\,\,\,\,\,\,\,\,\,\,\,\,\,\,\,\,\,\,\,\,\,\,\,\,\,
\xymatrix{
M_{1}\,\times \, Q_{1} \ar[r]^{\circ^{i}} \ar[d]^{f\times f_{q}} & M_{1} \ar[d]^{f}\\
M_{2}\,\times\, Q_{2} \ar[r]_{\circ^{i}} & M_{2}
}
$$
then $f$ induces a map of ($P_{1}$-$Q_{1}$) bimodules $\tilde{f}:\mathcal{B}(M_{1})\rightarrow \mathcal{B}(M_{2})$ where $\mathcal{B}(M_{1})$ and $\mathcal{B}(M_{2})$ are the Boardman-Vogt resolutions in the categories $Bimod_{P_{1}\text{-}Q_{1}}$ and $Bimod_{P_{2}\text{-}Q_{2}}$ respectively:
$$
\tilde{f}([T\,;\,\{t_{v}\}\,;\,\{a_{v}\}])=[T\,;\,\{t_{v}\}\,;\,\{a'_{v}\}]\hspace{15pt}\text{with} \hspace{15pt} a'_{v}:=\left\{
\begin{array}{ll}\vspace{3pt}
f_{q}(a_{v}) & \text{if } v\in V^{u}(T), \\ \vspace{3pt}
f(a_{v}) & \text{if } v\in V^{p}(T), \\ 
f_{p}(a_{v}) & \text{if } v\in V^{d}(T).
\end{array} 
\right.
$$
\end{rmk}

\begin{rmk}
From a $k$-truncated bimodule $M_{k}$, we consider the $k$-free bimodule $\mathcal{F}_{B}^{k}(M_{k})$ whose $k$ first components coincide with $M_{k}$. The functor $\mathcal{F}_{B}^{k}$, from truncated bimodules to bimodules, can be described using the set of reduced trees with section in which the sum of the incoming inputs of any two consecutive vertices is bigger than $k+2$. We can check that $\mathcal{F}_{B}^{k}$ is the left adjoint to the truncated functor $T_{k}$:
$$
\mathcal{F}_{B}^{k}:T_{k}Bimod_{P\text{-}Q}\leftrightarrows Bimod_{P\text{-}Q}:T_{k}.
$$
In particular, one has $\mathcal{F}_{B}^{k}(T_{k}(\mathcal{B}_{k}(M)))=\mathcal{B}_{k}(M)$ since $\mathcal{B}_{k}(M)$ is defined as the sub-bimodule of $\mathcal{B}(M)$ generated its $k$ first components. As a consequence of this adjunction together with Theorem \ref{B7}, one has the following identifications:
$$
T_{k}Bimod_{O}^{h}(T_{k}(M)\,;\,T_{k}(O'))\cong T_{k}Bimod_{O}(T_{k}(\mathcal{B}_{k}(M))\,;\,T_{k}(O'))\cong Bimod_{O}(\mathcal{B}_{k}(M)\,;\,O').
$$
\end{rmk}

\subsection{The Boardman-Vogt resolution for coloured operads}\label{C4}

In this section, we recall the Boardman-Vogt resolution for topological operads and we introduce the notation needed for the proof of the main theorem of the paper. Since this construction is similar to the Boardman-Vogt resolution considered in the previous subsections, we skip some details and we refer the reader to \cite{Berger06, Boardman73} for a complete description.

\begin{const}\label{e7}
From an $S$-operad $O$, we build the $S$-operad $\mathcal{BV}(O)$. The points are equivalent classes $[T\,;\,\{t_{e}\}\,;\,\{a_{v}\}]$ where $T$ is an $S$-tree, $\{a_{v}\}_{v\in V(T)}$ is a family of points in $O$ labelling the vertices of $T$ and $\{t_{e}\}_{e\in V^{int}(T)}$ is a family of real numbers in the interval $[0\,,\,1]$ indexing the inner edges. In other words, $\mathcal{BV}(O)$ is the quotient of the coproduct
$$
\left.
\underset{T\in \,S\text{-tree}}{\coprod} \,\,\underset{v\in V(T)}{\prod}\,O(e_{1}(v),\ldots,e_{|v|}(v);e_{0}(v)) \,\,\times \,\,\underset{e\in\, E^{int}(T)}{\prod}\, [0\,,\,1]\,\,
\right/\!\sim\,\,.
$$
The equivalence relation is generated by the following axioms: 
\begin{itemize}[itemsep=-10pt, topsep=3pt, leftmargin=*]
\item[$i)$] If a vertex is labelled by a distinguished point $\ast_{s}\in O(s;s)$, then

\begin{center}
\includegraphics[scale=0.25]{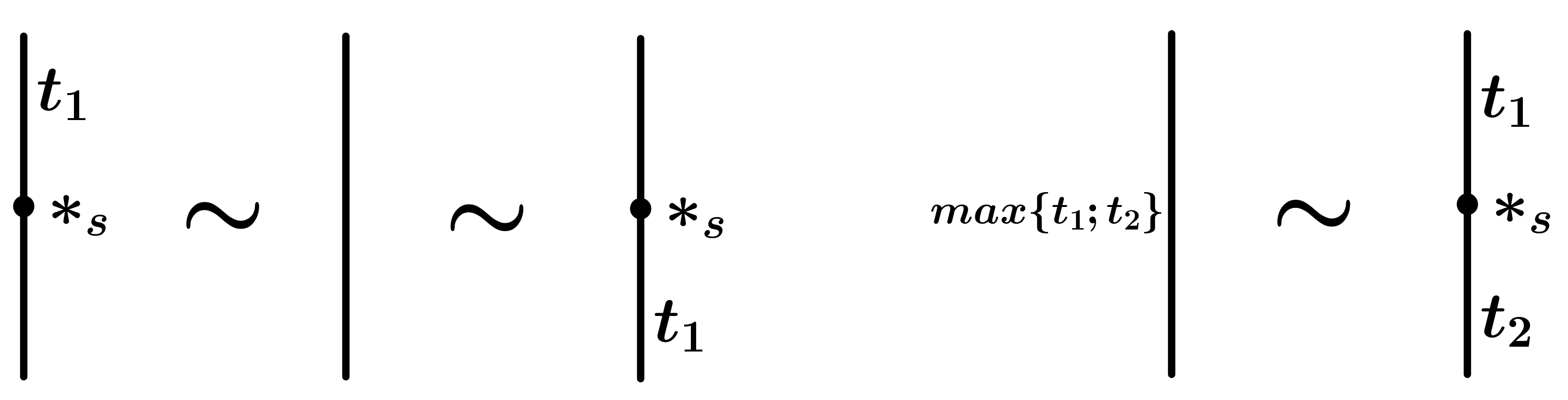}
\end{center}\vspace{9pt}

\item[$ii)$] If a vertex is labelled by $a\cdot \sigma$, with $\sigma\in \Sigma_{|v|}$, then 
\begin{center}
\includegraphics[scale=0.35]{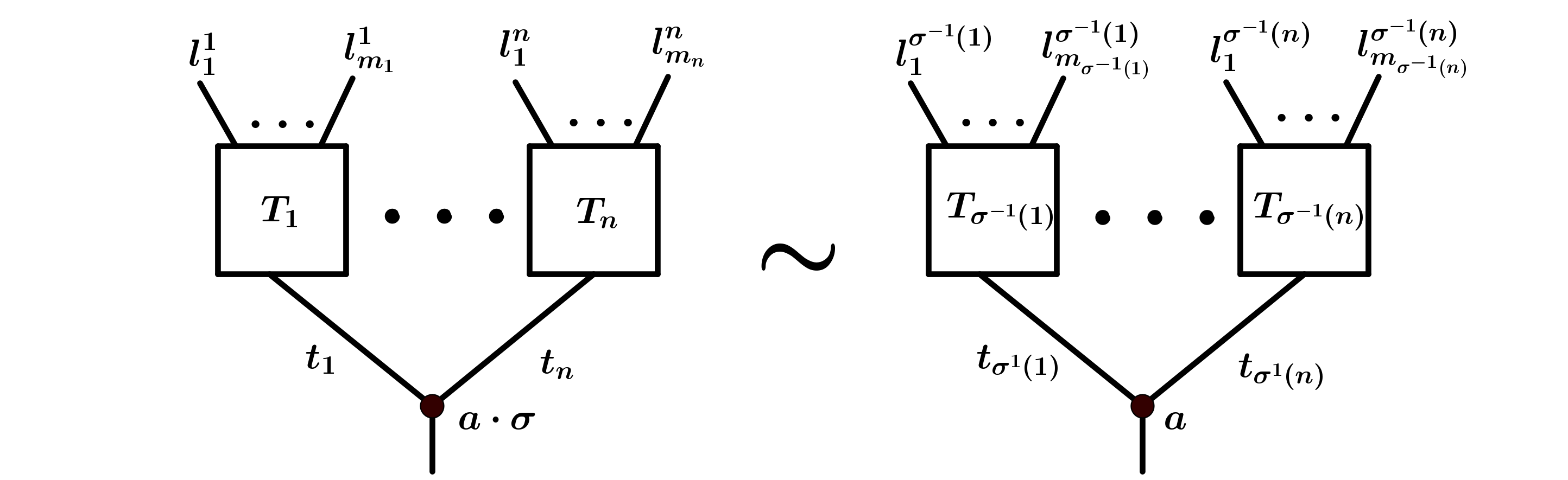}
\end{center}\vspace{10pt}

\item[$iii)$] If an inner edge is indexed by $0$, then we contract it by using the operadic structure of $O$.
\begin{figure}[!h]
\begin{center}
\includegraphics[scale=0.08]{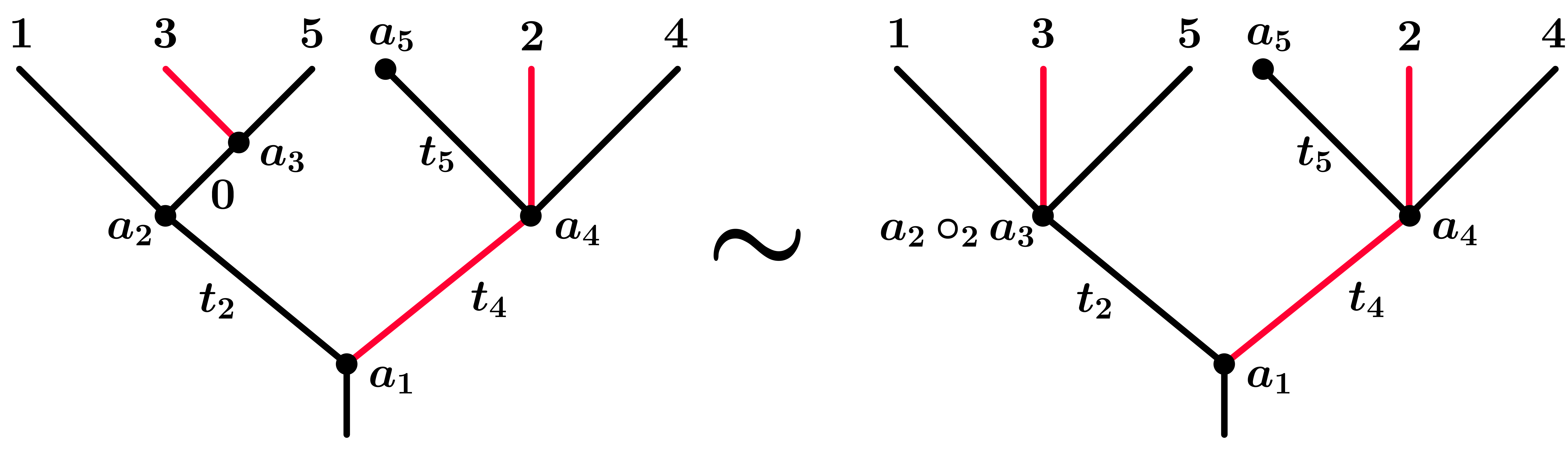}
\caption{Illustration of the relation $(iii)$.}
\end{center}
\end{figure}
\end{itemize}\vspace{-10pt}

\noindent Let $[T\,;\,\{t_{e}\}\,;\,\{a_{v}\}]$ be a point in $\mathcal{BV}(O)(s_{1},\ldots,s_{n};s_{n+1})$ and $[T'\,;\,\{t'_{e}\}\,;\,\{a'_{v}\}]$ be a point in $\mathcal{BV}(O)(s'_{1},\ldots,s'_{m};s_{i})$. The operadic composition $[T\,;\,\{t_{e}\}\,;\,\{a_{v}\}]\circ_{i}[T'\,;\,\{t'_{e}\}\,;\,\{a'_{v}\}]$ consists in grafting $T'$ to the $i$-th incoming edge of $T$ and indexing the new inner edge by $1$. Furthermore, there is a map of pointed $S$-sequences,
\begin{equation}\label{B8}
\iota:O\longrightarrow \mathcal{BV}(O)\,\,;\,\,a\longmapsto [t_{|a|}\,;\,\emptyset\,;\,\{a\}],
\end{equation} 
sending a point $a$ to the corolla labelled by $a$. There is also a map of operads sending the real numbers indexing the inner edges to $0$,
\begin{equation}\label{d8}
\mu:\mathcal{BV}(O)\rightarrow O\,\,;\,\, [T\,;\,\{t_{e}\}\,;\,\{a_{v}\}] \mapsto [T\,;\,\{0_{e}\}\,;\,\{a_{v}\}].
\end{equation}
\end{const}

From now on, we introduce a filtration of the resolution $\mathcal{BV}(O)$ according to the number of geometrical inputs which is the number of leaves plus the number of univalent vertices. Similarly to the bimodule case, a point in $\mathcal{BV}(O)$ is said to be prime if the real numbers indexing the set of inner edges are strictly smaller than $1$. Besides, a point is said to be composite if one of its inner edges is indexed by $1$ and such a point can be decomposed into prime components. More precisely, the prime components of a point indexed by a planar tree are obtained by cutting the inner edges indexed by $1$ as illustrated in Figure \ref{H0}. Otherwise, the prime components of a point of the form $[(T\,;\,\sigma)\,;\,\{t_{e}\}\,;\,\{a_{v}\}]$, with $\sigma\neq id$, coincide with the prime components of $[(T\,;\,id)\,;\,\{t_{e}\}\,;\,\{a_{v}\}]$.\\ \vspace{-10pt}

\begin{figure}[!h]
\begin{center}
\includegraphics[scale=0.27]{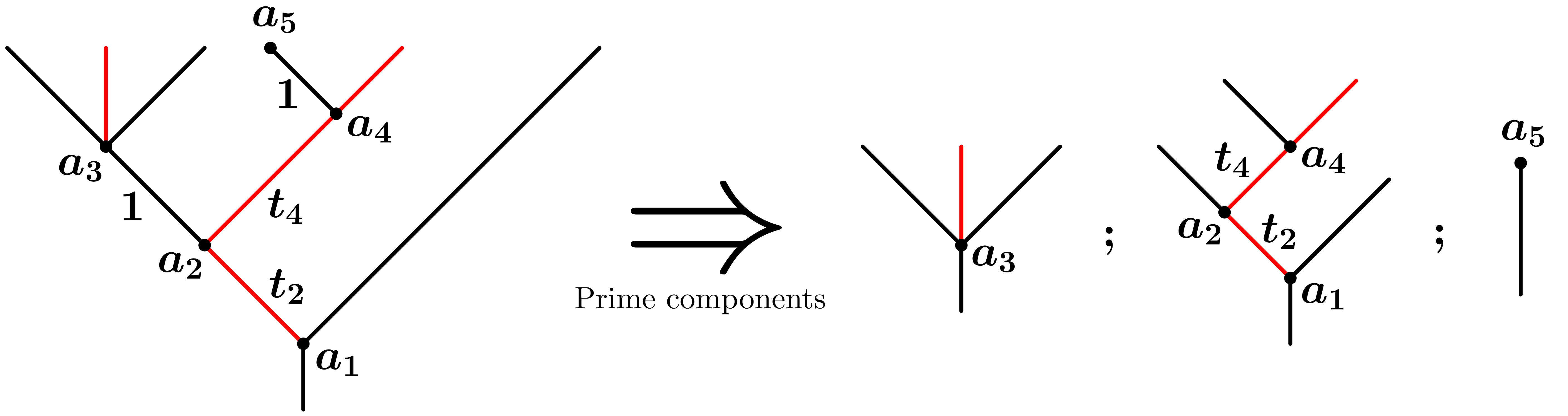}
\caption{Illustration of a composite point together with its prime components.}\label{H0}\vspace{-15pt}
\end{center}
\end{figure}

\noindent A prime point is in the $k$-th filtration term $\mathcal{BV}_{k}(O)$ if the number of its geometrical inputs is smaller than $k$. Then, a composite point is in the $k$-th filtration term if its prime components are in $\mathcal{BV}_{k}(O)$. For instance, the composite point in Figure \ref{H0} is an element in the filtration term $\mathcal{BV}_{4}(O)$. By convention, $\mathcal{BV}_{0}(O)$ is the initial object in the category of $S$-operads. For each $k\geq 0$, $\mathcal{BV}_{k}(O)$ is an $S$-operad and the family $\{\mathcal{BV}_{k}(O)\}$ gives rise to a filtration of $\mathcal{BV}(O)$, 
\begin{equation*}
\xymatrix{
\mathcal{BV}_{0}(O)\ar[r] & \mathcal{BV}_{1}(O) \ar[r] &  \cdots \ar[r] & \mathcal{BV}_{k-1}(O) \ar[r] & \mathcal{BV}_{k}(O) \ar[r] & \cdots \ar[r] & \mathcal{BV}(O).
}
\end{equation*}

\begin{thm}{\cite{Berger03,Vogt03}}\label{B9}
Assume that $O$ is a well pointed $\Sigma$-cofibrant $S$-operad. The objects $\mathcal{BV}(O)$ and $T_{k}(\mathcal{BV}_{k}(O))$ are cofibrant replacements of $O$ and $T_{k}(O)$ in the categories $Operad_{S}$ and $T_{k}Operad_{S}$ respectively. In particular, the map (\ref{d8}) is a weak equivalence. 
\end{thm}

\begin{proof}
In what follows, we recall only the proof that the inclusion from $\mathcal{BV}_{k-1}(O)$ to $\mathcal{BV}_{k}(O)$ is a cofibration in the category of $S$-operads. For this purpose, we consider another filtration according to the number of vertices. Let \textbf{$S$-tree[k\,;\,l]} be the set of $S$-trees having exactly $k$ geometrical inputs and $l$ vertices. Then, the $S$-sequence $Y_{k}[l]$ is the quotient of the coproduct 

$$
\left.
\underset{T\in \,S\textbf{-tree[k\,;\,l]}}{\coprod} \,\,\underset{v\in V(T)}{\prod}\,O(e_{1}(v),\ldots,e_{|v|}(v);e_{0}(v)) \,\,\times \,\,\underset{e\in\, E^{int}(T)}{\prod}\, [0\,,\,1]\,\,
\right/\!\sim 
$$ 
where the equivalent relation is generated by the axiom $(ii)$ of Construction \ref{e7}. The $S$-sequence $\partial Y_{k}[l]$ is formed by points in $Y_{k}[l]$ having a bivalent vertex labelled by the unit of the operad $O$ or having an inner edge indexed by $0$ or $1$. For $(k\,;\,l)\neq (1\,;\,1)$, the $S$-sequences $Y_{k}[l]$ and $\partial Y_{k}[l]$ are not pointed. In order to use the operad functor $\mathcal{F}$ from pointed $S$-sequences to $S$-operads, we consider the $S$-sequences $\tilde{Y}_{k}[l]$ and $\partial \tilde{Y}_{k}[l]$ obtained by adding based points in arity $1$:
$$
\tilde{Y}_{k}[l](s_{1},\ldots,s_{n}\,;\,s_{n+1}):=
\left\{
\begin{array}{ll}\vspace{4pt}
Y_{k}[l](s\,;\,s)\sqcup \ast_{s} & \text{if } n=1 \text{ and } s_{1}=s_{2}=s, \\ 
Y_{k}[l](s_{1},\ldots,s_{n}\,;\,s_{n+1}) & \text{otherwise},
\end{array} 
\right.
$$

$$
\partial\tilde{Y}_{k}[l](s_{1},\ldots,s_{n}\,;\,s_{n+1}):=
\left\{
\begin{array}{ll}\vspace{4pt}
\partial Y_{k}[l](s\,;\,s)\sqcup \ast_{s} & \text{if } n=1 \text{ and } s_{1}=s_{2}=s, \\ 
\partial Y_{k}[l](s_{1},\ldots,s_{n}\,;\,s_{n+1}) & \text{otherwise}.
\end{array} 
\right.\vspace{5pt}
$$ 

\noindent Then, we consider the following pushout diagrams: 
$$
\xymatrix{
\mathcal{F}(\mathcal{BV}_{0}(O)) \ar[r] \ar@{=}[d] & \mathcal{F}(Y_{1}[1]) \ar@{=}[d] \\
\mathcal{BV}_{0}(O) \ar[r] & \mathcal{BV}_{1}(O)[1]
}
\hspace{15pt}\text{and}\hspace{15pt}
\xymatrix{
\mathcal{F}(\partial \tilde{Y}_{k}[l]) \ar[r] \ar[d] & \mathcal{F}(\tilde{Y}_{k}[l]) \ar[d] \\
\mathcal{BV}_{k}(O)[l-1] \ar[r] & \mathcal{BV}_{k}(O)[l]
}
$$ 
Similarly to the proof of Theorem \ref{B7}, we can show that the inclusion from $\partial Y_{k}[l]$ to $Y_{k}[l]$ is a $\Sigma$-cofibration. As a consequence, the horizontal maps of the above diagrams are cofibrations in the category of $S$-operads. Since $lim_{l}\,\mathcal{BV}_{k}(O)[l]$ is the bimodule $\mathcal{BV}_{k}(O)$, the inclusion from $\mathcal{BV}_{k-1}(O)$ to $\mathcal{BV}_{k}(O)$ is also a cofibrations in the category of $S$-operads. 
\end{proof}

\begin{rmk}
Analogously to the bimodule case, from a $k$-truncated operad $O_{k}$, we consider the $k$-free operad $\mathcal{F}^{k}(O_{k})$ whose $k$ first components coincide with $O_{k}$. The functor $\mathcal{F}^{k}$, from truncated operads to operads, can be described using the set of trees in which the sum of the incoming inputs of any two consecutive vertices is bigger than $k+2$. We can check that $\mathcal{F}^{k}$ is the left adjoint to the truncated functor $T_{k}$:
$$
\mathcal{F}^{k}:T_{k}Operad^{h}\leftrightarrows Operad^{h}:T_{k}.
$$
In particular, one has $\mathcal{F}^{k}(T_{k}(\mathcal{BV}_{k}(O)))=\mathcal{BV}_{k}(O)$ since $\mathcal{BV}_{k}(O)$ is defined as the sub-operad of $\mathcal{BV}(O)$ generated its $k$ first components. As a consequence of this adjunction together with Theorem \ref{B9}, one has the following identifications:
$$
T_{k}Operad^{h}(T_{k}(O)\,;\,T_{k}(O'))\cong T_{k}Operad(T_{k}(\mathcal{BV}_{k}(O))\,;\,T_{k}(O'))\cong Operad(\mathcal{BV}_{k}(O)\,;\,O').
$$
\end{rmk}


\section{\texorpdfstring{$\mathcal{SC}_{1}$}{Lg}-algebra arising from a map of two-coloured operads}

From a map of operads $\eta:O\rightarrow O'$, we prove in \cite{Ducoulombier16} that the derived mapping space of bimodules $Bimod_{O}^{h}(O\,;\,O')$ is an algebra over the operad $\mathcal{C}_{1}$. Furthermore, we have been able to identify the corresponding loop space using an explicit cofibrant replacement of the operad $O$ in the model category $Bimod_{O}$ which differs from the cofibrant resolution introduced in this paper. More precisely, one has the theorem below which is a generalization of results obtained by Dwyer-Hess \cite{Dwyer12} and independently by Turchin \cite{Tourtchine10} in the context of non-symmetric operads and when the source object is the associative operad $\mathcal{A}s$.

\begin{thm}{\cite[Theorem 3.1]{Ducoulombier16}}\label{H1}
Let $O$ be a well pointed $\Sigma$-cofibrant operad and $\eta:O\rightarrow O'$ be a map of operads. If the spaces $O(1)$ and $O'(1)$ are contractible, then there are explicit weak equivalences of $\mathcal{C}_{1}$-algebras:
$$
\begin{array}{rcl}\vspace{5pt}
\xi:\Omega Operad^{h}(O\,;\,O') & \longrightarrow & Bimod_{O}^{h}(O\,;\,O'), \\ 
\xi_{k}:\Omega\big(\, T_{k}Operad^{h}(T_{k}(O)\,;\,T_{k}(O'))\,\big) & \longrightarrow & T_{k}Bimod_{O}^{h}(T_{k}(O)\,;\,T_{k}(O')).
\end{array} 
$$
\end{thm} 
\noindent In what follows, we give a similar statement in the relative case using the language of coloured operads with set of colours $S=\{o\,;\,c\}$. In particular, for $\eta_{1}:O\rightarrow O'$ an operadic map and $\eta_{2}:O'\rightarrow M$ a bimodule map over $O'$, we prove that the pair of  spaces  $(\,Bimod_{O}^{h}(O\,;\,O')\,;\,Bimod_{O}^{h}(O\,;\,M)\,)$ is weakly equivalent to an explicit $\mathcal{SC}_{1}$-algebra. For this purpose, we consider an adjunction between the category of ($P$-$Q$) bimodules and a subcategory of two-coloured operads described below.

\begin{notat}
Let $O$ be an $\{o\,;\,c\}$-operad. We denote by $O_{c}$ and $O_{o}$ the  operads coming from the restriction to the colour $c$ and $o$ respectively. In other words, $O_{c}$ and $O_{o}$ are defined as follows:
$$
O_{c}(n)=O(\,\,\underset{n}{\underbrace{c,\ldots,c}}\,\,;c)\, \text{ for } n\geq 0 \hspace{30pt} \text{and} \hspace{30pt} O_{o}(n)=O(\,\,\underset{n}{\underbrace{o,\ldots,o}}\,\,;o)\, \text{ for } n\geq 0.
$$
Conversely, from two operads $P$ and $Q$, we build the $\{o\,;\,c\}$-operad $P\oplus Q$ given by
$$
P\oplus Q(\,\,\underset{n}{\underbrace{c,\ldots,c}}\,\,;c)=Q(n)\, \text{ for } n\geq 0,   \hspace{30pt} P\oplus Q(\,\,\underset{n}{\underbrace{o,\ldots,o}}\,\,;o)=P(n)\, \text{ for } n\geq 0,
$$
and the empty set otherwise. Consequently, a map of two-coloured operads $f:P\oplus Q\rightarrow O$ is equivalent to a pair of operadic maps $f_{c}:Q\rightarrow O_{c}$ and $f_{o}:P\rightarrow O_{o}$.
\end{notat}

\begin{defi}\label{d7}
Let $P$ and $Q$ be two operads. We consider the category of $\{o\,;\,c\}$-operads under $P\oplus Q$
$$
Op[P\,;\,Q]:= (P\oplus Q) \downarrow Operad_{\{o\,;\,c\}}.
$$
An object $(O\,;\,\tau_{O})$ is given by an $\{o\,;\,c\}$-operad $O$ together with an $\{o\,;\,c\}$-operadic map $\tau_{O}:P\oplus Q \rightarrow O$. A morphism $f:(O\,;\,\tau_{O})\rightarrow (O'\,;\,\tau_{O'})$ is an $\{o\,;\,c\}$-operadic map $f:O\rightarrow O'$ such that the following diagram commutes:
$$
\xymatrix@R=10pt{
 & P\oplus Q \ar[dl]_{\tau_{O}} \ar[dr]^{\tau_{O'}} & \\
 O \ar[rr]_{f} & & O'
}\vspace{-20pt}
$$

\newpage

There is an obvious functor from the category of operads under $P\oplus Q$ to the category of ($P$-$Q$) bimodules. Given $(O\,;\,\tau_{O})\in Op[P\,;\,Q]$, the sequence $\mathcal{R}(O\,;\,\tau_{O})$ (also denoted by $\mathcal{R}(O)$ if there is no ambiguity on the maps $\tau_{O}$) is defined as follows:
$$
\mathcal{R}(O)(n):=O(\,\,\underset{n}{\underbrace{c,\ldots,c}}\,\,;o)\,  \text{for }n\geq 0. 
$$
By definition, $\mathcal{R}(O)$ is endowed with an ($O_{o}$-$O_{c}$) bimodule structure in which the map $\gamma:O_{o}(0)\rightarrow \mathcal{R}(O)(0)$ is the identity map. Due to the $\{o\,;\,c\}$-operadic map $\tau_{O}$, $\mathcal{R}(O)$ is also a ($P$-$Q$) bimodule. Thus, one has a functors
$$
\mathcal{R}:Op[P\,;\,Q]\longrightarrow Bimod_{P\text{-}Q}.
$$
\end{defi}

In Section \ref{f4}, we give a presentation of the left adjoint $\mathcal{L}$ to the functor $\mathcal{R}$. We also prove that the pair of functors so obtained is a Quillen adjunction. As a consequence, $\mathcal{L}$ preserves cofibrations and Construction \ref{b7} produces cofibrant objects in the category $Op[P\,;\,Q]$ in the particular case $P=Q=O$. Then, in Section \ref{h1}, we modify slightly Construction \ref{b7} in order to get cofibrant replacements in the category $Op[O\,;\,\emptyset]$ in which $\emptyset$ is the initial object in the category of operads. Section \ref{F2} is devoted to the proof of the main theorem which is the following one:
\begin{thm}\label{d9}
Let $O$ be a well pointed operad and $\eta:\mathcal{L}(O)\rightarrow O'$ be a map in the category $Op[O\,;\,O]$. If $O$ is $\Sigma$-cofibrant, then the following weak equivalence holds:
\begin{equation}\label{e3}
Bimod_{O}^{h}(O\,;\,\mathcal{R}(O'))  \simeq  \Omega\big(\, Operad^{h}(O\,;\,O'_{c})\,;\,Op[O\,;\,\emptyset]^{h}(\mathcal{L}(O)\,;\,O')\,\big).
\end{equation}
\end{thm}

\begin{proof}
It is a direct consequence of Theorem \ref{d3} together with Propositions \ref{d4}, \ref{d5} and \ref{d6}.
\end{proof}

\noindent In Section \ref{H2}, we prove a truncated version of the above theorem. Finally, the last subsection introduces the two-coloured operad $\mathcal{CC}_{d}$ and we identify explicit $\mathcal{SC}_{d+1}$-algebras from maps of coloured operads $\eta:\mathcal{CC}_{d}\rightarrow O$ using the Dwyer-Hess' conjecture.

\subsection{The left adjoint to the functors \texorpdfstring{$\mathcal{R}$}{Lg} }\label{f4}

Let $M$ be an $\{o\,;\,c\}$-sequence and $I_{c}$, $I_{o}$ be a partition of the set $\{1,\ldots,n\}$. In order to simplify the notation, we denote by $M(I_{c},I_{o};s_{n+1})$ the space
$$
M(s_{1},\ldots,s_{n};s_{n+1})\hspace{15pt}\text{in which}\hspace{15pt}
\left\{\begin{array}{ll}
s_{i}=c & \text{if } i\in I_{c}, \\ 
s_{i}=o & \text{if } i\in I_{o}.
\end{array} \right.
$$

\begin{const}\label{i7}
Let $P$ and $Q$ be two operads. From a ($P$-$Q$) bimodule $M$, we build the  $\{o\,;\,c\}$-operad $\mathcal{L}(M;P;Q)$ as follows:
$$\mathcal{L}(M;P;Q)(c,\ldots,c;c)=Q(n)\, \text{ for } n\geq 0 \hspace{30pt} \text{and} \hspace{30pt} \mathcal{L}(M;P;Q)(c,\ldots,c;o)=M(n)\, \text{ for } n\geq 0.$$
In order to describe the spaces $\mathcal{L}(M;P;Q)(I_{c},I_{o};o)$, with $|I_{c}|\geq 0$ and $|I_{o}|\geq 1$, we introduce the set $\Psi(I_{c};I_{o})$ formed by $\{o\,;\,c\}$-trees whose trunk has the colour $o$, the leaves indexed by the set $I_{c}$ have colour $c$ and the leaves indexed by the set $I_{o}$ have colour $o$. Moreover, the elements in $\Psi(I_{c};I_{o})$ are two-levels trees (i.e each path connecting a leaf or a univalent vertex to the trunk passes through at most two vertices) and satisfy the following conditions:
\begin{itemize}
\item[$\blacktriangleright$] the incoming edges of the vertices distinct from the root are indexed by $c$,
\item[$\blacktriangleright$] the incoming edges of the root are indexed by the colour $o$,
\item[$\blacktriangleright$] the root has at least one incoming leaf.
\end{itemize}
\begin{figure}[!h]
\begin{center}
\includegraphics[scale=0.2]{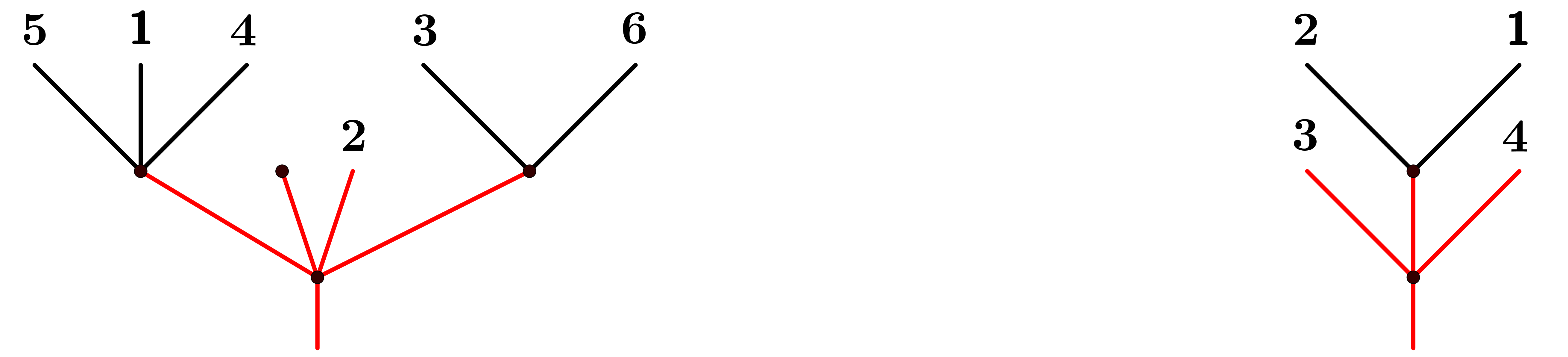}
\caption{Elements in $\Psi(\{1,3,4,5,6\}\,;\,\{2\})$ and $\Psi(\{1,2\}\,;\,\{3,4\})$ respectively.}\vspace{-30pt}
\end{center}
\end{figure}

\newpage

\noindent To define the $\{o\,;\,c\}$-sequence $\mathcal{L}(M;P;Q)$, we label the root by a point in the operad $P$ whereas the other vertices are labelled by points in the bimodule $M$. In other words, one has
\begin{equation}\label{E3}
\mathcal{L}(M;P;Q)(I_{c},I_{o};o):= \underset{T\in \Psi(I_{c},I_{o})}{\coprod}\,\,\left.\left[ \,P(|r|)\,\,\times\,\, \underset{v \neq r}{\prod} \,M(|v|)\,\,\right]\right/ \sim\,.
\end{equation}

\noindent The equivalence relation $\sim$ is generated by the compatibility with the symmetric group action (axiom $(ii)$ of Construction \ref{b5}), the contraction of edges the source of which are of the form $\gamma(x)$ with $x\in P(0)$ (axiom $(iii)$ of Construction \ref{b5}) as well as the relation defined as follows: if a vertex of a tree $T\in \Psi(I_{c}\,,\, I_{o})$ other than the root is labelled by a point of the form $a(m_{1},\ldots,m_{n})$, with $a\in P$ and $m_{i}\in M$, then one has the identification

\begin{figure}[!h]
\begin{center}
\includegraphics[scale=0.4]{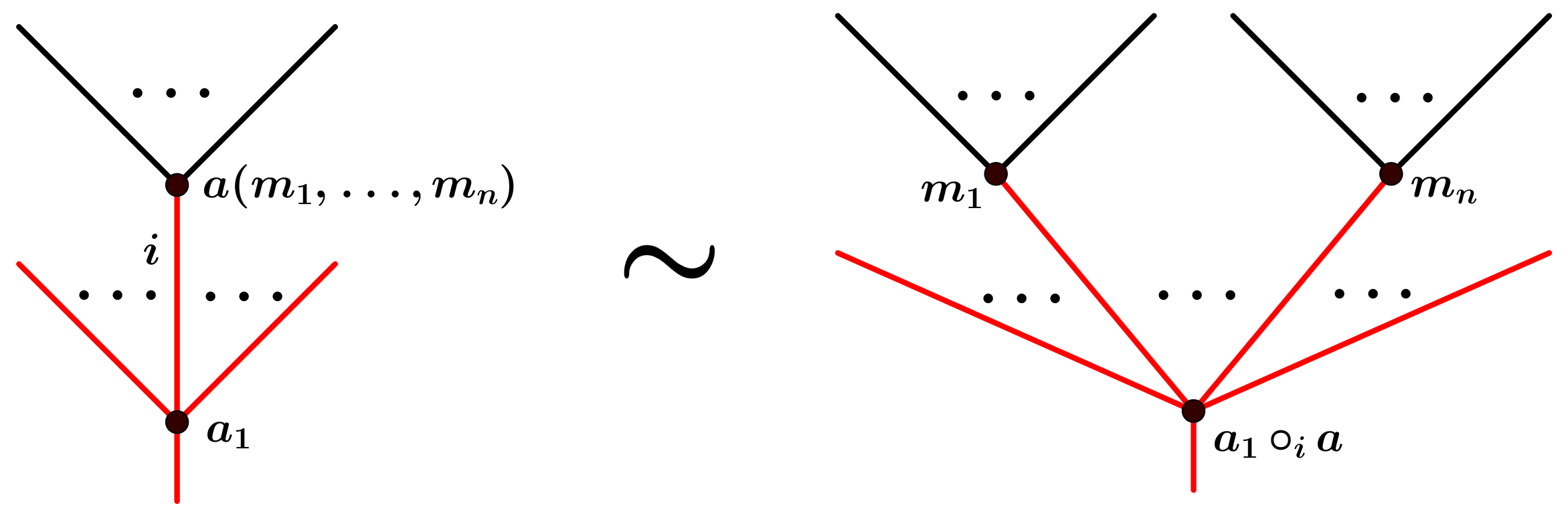}\vspace{-5pt}
\caption{Illustration of the additional relation.}\label{F3}\vspace{-10pt}
\end{center}
\end{figure}

\noindent By convention, a point $a\in\mathcal{L}(M;P;Q)(c,\ldots,c;c)$ is interpreted as a corolla whose edges are coloured by $c$ and the vertex is labelled by $a\in Q(n)$. Similarly, a point $a\in\mathcal{L}(M;P;Q)(c,\ldots,c;o)$ is interpreted as a corolla whose leaves have colour $c$, the trunk has colour $o$ and the vertex is labelled by $a\in P(n)$. We denote by $[T\,;\,\{a_{v}\}]$ a point in $\mathcal{L}(M;P;Q)$.

From now on, we describe the ($P$-$Q$) bimodule structure of  $\mathcal{L}(M;P;Q)$. Let  $[T\,;\,\{a_{v}\}]$ be a point in $\mathcal{L}(M;P;Q)(I_{c},I_{o};o)$ and $a$ be a point in $Q(n)=\mathcal{L}(M;P;Q)(c,\ldots,c;c)$. The operadic composition $[T\,;\,\{a_{v}\}]\circ_{i} a$, with $i\in I_{c}$, consists in grafting the corolla labelled by $a$ to the $i$-th incoming edge of $T$ and contracting the inner edge so obtained by using the right $Q$-module structure of $M$.   

\begin{figure}[!h]
\begin{center}
\includegraphics[scale=0.6]{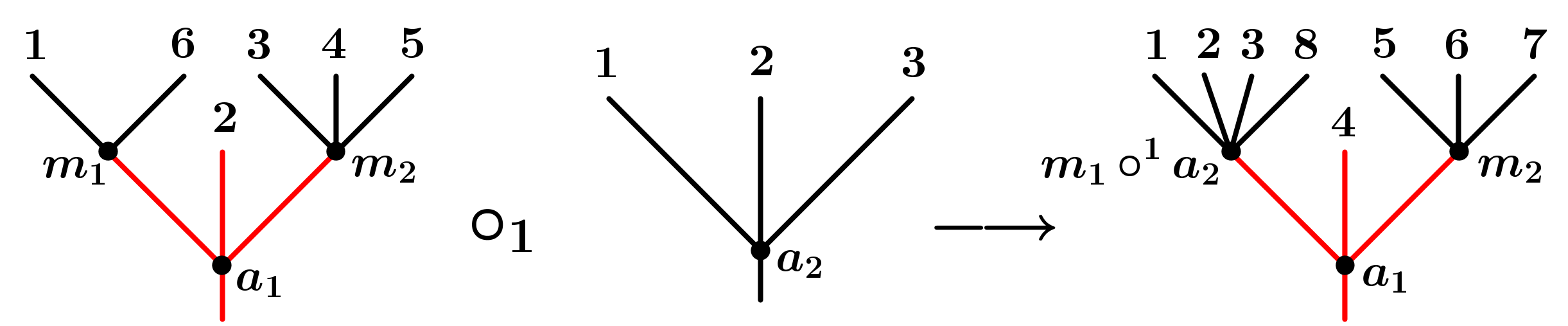}\vspace{-5pt}
\caption{Illustration of the operadic composition.}\vspace{-13pt}
\end{center}
\end{figure}

\noindent  Let $[T\,;\,\{a_{v}\}]$ be a point in $\mathcal{L}(M;P;Q)(I_{c},I_{o};o)$ and $[T'\,;\,\{a'_{v}\}]$ be a point in $\mathcal{L}(M;P;Q)(I'_{c},I'_{o};o)$. The operadic composition $[T\,;\,\{a_{v}\}]\circ_{i} [T'\,;\,\{a'_{v}\}]$, with $i\in I_{o}$, consists in grafting $T'$ to the $i$-th incoming edge of $T$ and contracting the inner edges connecting two vertices indexed by points in $P$ using its operadic structure. If all the incoming edges of the new root are inner edges, then we contract all of them by using the left $P$-module structure of $M$.

\begin{figure}[!h]
\begin{center}
\includegraphics[scale=0.49]{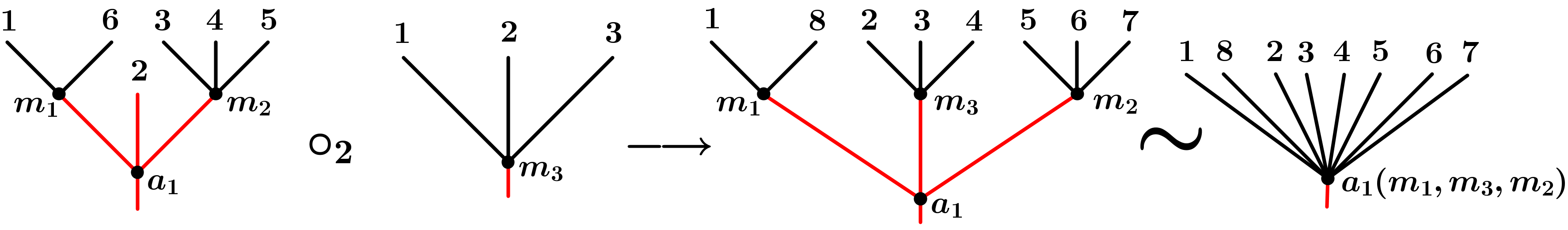}\vspace{-5pt}
\caption{Illustration of the operadic composition.}
\end{center}
\end{figure} \vspace{-15pt}

By construction, $\mathcal{L}(M;P;Q)$ is an $\{o\,;\,c\}$-operad, denoted by $\mathcal{L}(M)$ if there is no ambiguity about the operads $P$ and $Q$. Furthermore, there is an $\{o\,;\,c\}$-operadic map,
$$
\tau_{M}:P\oplus Q \rightarrow \mathcal{L}(M;P;Q),\vspace{-25pt}
$$
\newpage

\noindent induced by the identity map $(\tau_{O})_{c}:Q\rightarrow \mathcal{L}(M;P;Q)_{c}=Q$. The morphism $(\tau_{O})_{o}:P\rightarrow \mathcal{L}(M;P;Q)_{o}$ assigns a point $a\in P(0)$ to $\gamma_{o}(a)\in M(0)$ and a point $a\in P(n)$, with $n>0$, to the corolla labelled by $a$. Consequently, the pair $(\mathcal{L}(M;P;Q)\,;\, \tau_{M})$ is an object in the category $Op[P\,;\,Q]$. Finally, one has a functor
$$
\mathcal{L}(-;P;Q):Bimod_{P\text{-}Q}\longrightarrow Op[P\,;\,Q].
$$
\end{const}

\begin{rmk}\label{F8}
This construction is functorial with respect to the operads $P$ and $Q$. Let $f_{p}:P\rightarrow P'$ and $f_{q}:Q\rightarrow Q'$ be two operadic maps and $f_{m}:M\rightarrow M'$ be a map from a ($P$-$Q$) bimodule $M$ to a ($P'$-$Q'$) bimodule $M'$ such that the following diagrams commute:
\begin{equation}\label{f2}
\xymatrix{
P(n)\,\times\, M(m_{1})\times \cdots \times M(m_{n}) \ar[r]^{\,\,\,\,\,\,\,\,\,\,\,\,\,\,\,\,\,\,\,\,\,\,\,\,\,\,\gamma_{l}} \ar[d]^{f_{p}\times f_{m}\times \cdots \times f_{m}} & M(m_{1}+\cdots m_{n}) \ar[d]^{f_{m}} \\
P'(n)\,\times\, M'(m_{1})\times \cdots \times M'(m_{n}) \ar[r]^{\,\,\,\,\,\,\,\,\,\,\,\,\,\,\,\,\,\,\,\,\,\,\,\,\,\,\gamma_{l}}  & M'(m_{1}+\cdots m_{n})
}
\,\,\,\,\,\,\,\,\,\,\,\,\,\,\,\,
\xymatrix{
M(n)\,\times \, Q(m) \ar[r]^{\circ^{i}} \ar[d]^{f_{m}\times f_{q}} & M(n+m-1) \ar[d]^{f_{m}}\\
M'(n)\,\times \, Q'(m) \ar[r]^{\circ^{i}}  & M'(n+m-1) 
}
\end{equation}
Then, there is a map of $\{o\,;\,c\}$-operads
$$
\begin{array}{rcl}\vspace{3pt}
f:\mathcal{L}(M;P;Q) & \longrightarrow & \mathcal{L}(M';P';Q'); \\ 
 {}[ T\,;\,\{a'_{v}\} ] & \longmapsto & [T\,;\,\{a'_{v}\}],
\end{array} 
\hspace{15pt}\text{with}\hspace{15pt} 
a'_{v}=\left\{
\begin{array}{ll}\vspace{3pt}
f_{p}(a_{v}) & \text{if}\,\,a_{v}\in P, \\ \vspace{3pt}
f_{m}(a_{v}) & \text{if}\,\,a_{v}\in M,\\ 
f_{q}(a_{v}) & \text{if}\,\,a_{v}\in Q.
\end{array} 
\right.
$$ 
\end{rmk}

\begin{pro}
The pair of functors $(\mathcal{L}\,;\,\mathcal{R})$ form an adjoint pair.
\end{pro}

\begin{proof}
Let $(O\,;\,\tau_{O})$ be an object in the category $Op[P\,;\,Q]$ and let $f:M\rightarrow \mathcal{R}(O)$ be a ($P$-$Q$) bimodule map. If the map of $\{o\,;\,c\}$-operads $\tilde{f}:\mathcal{L}(M;P;Q)\rightarrow O$ extends $f$ and satisfy the relation $\tilde{f}\circ \tau_{M} =\tau_{O}$, then the folloing conditions hold:
\begin{itemize}[itemsep=3pt, topsep=7pt]
\item[$i)$] $\tilde{f}:\mathcal{L}(M;P;Q)_{c}(n)=Q(n)\rightarrow O_{c}(n)$ coincides with the map $\tau_{O}:Q(n)\rightarrow O_{c}(n)$,
\item[$ii)$] the restriction of $\tilde{f}:\mathcal{L}(M;P;Q)_{o}(n)\rightarrow O_{o}(n)$ to $P(n)$ coincides with $\tau_{O}:P(n)\rightarrow O_{o}(n)$,
\item[$iii)$] $\tilde{f}:\mathcal{L}(M;P;Q)(c,\ldots,c;o)=M(n)\rightarrow \mathcal{R}(O)(n)$ coincides with the map $f$.
\end{itemize} 

We define the map $\tilde{f}$ on the spaces $\mathcal{L}(M;P;Q)(I_{c},I_{o};o)$ by induction on the number of vertices of the trees $T\in \Psi(I_{c},I_{o})$. Let $[T\,;\,\{a_{v}\}]$ be a point in $\mathcal{L}(M;P;Q)(I_{c},I_{o};o)$. If $T$ has only one vertex, then $|I_{c}|=0$ or $|I_{o}|=0$ since all adjacent edges (i.e edges with the same target vertex) have the same colour. In these cases, the map $\tilde{f}$ is defined by the conditions $(ii)$ and $(iii)$ above.   

Assume the map $\tilde{f}$ has been defined for the trees having at most $k\geq 1$ vertices. Let  $[(T\,;\,\sigma)\,;\,\{a_{v}\}]$ be a point in $\mathcal{L}(M;P;Q)(I_{c},I_{o};o)$ where $T$ has $k+1$ vertices and $\sigma$ is the permutation indexing the leaves of $T$. By construction, $[(T\,;\,id)\,;\,\{a_{v}\}]$ has a decomposition on the form  $[(T_{1}\,;\,id)\,;\,\{a_{v}\}\setminus \{a_{0}\}]\circ_{i}[(T_{2}\,;\,id)\,;\,a_{0}]$ where $T_{2}$ is the corolla whose leaves have colour $c$ and the vertex is labelled by $a_{0}\in M$. Since $\tilde{f}$ has to preserve the operadic structure, then one has
$$
\tilde{f}\big( \, [(T\,;\,\sigma)\,;\,\{a_{v}\}]\,\big) \,=\,\big(\, \tilde{f}\big( \,  [(T_{1}\,;\,id)\,;\,\{a_{v}\}\setminus \{a_{0}\}]\, \big) \,\circ_{i}\, f(a_{0})\,\big)\cdot\sigma,
$$
where $T_{1}$ is a tree with $k$ vertices. Due to the operadic and the bimodule axioms, the map $\tilde{f}$ doesn't depend on the choice of the decomposition and the uniqueness follows from the construction.
\end{proof}

The categories $Bimod_{P\text{-}Q}$ has a cofibrantly generated model structure whereas $Op[P\,;\,Q]$ inherits a model category structure from the category of  coloured operads. A map $f:(O\,;\,\tau_{o})\rightarrow (O'\,;\,\tau_{O'})$ is a weak equivalence (resp. fibration or cofibration) if the  operadic map $f:O\rightarrow O'$ is a weak equivalence (resp. fibration or cofibration). In both cases, every object is fibrant. The following proposition claims that the adjunction has good properties with respect to the model category structures. 

\begin{pro}
The pair $(\mathcal{L}\,;\,\mathcal{R})$ is a Quillen adjunction. 
\end{pro}

\begin{proof}
It is sufficient to show that $\mathcal{R}$ preserves fibrations and acyclic fibrations. Let $f:(O\,;\,\tau_{O})\rightarrow (O'\,;\,\tau_{O'})$ be a fibration in the category $Op[P\,;\,Q]$, that is, $f:O\rightarrow O'$ is a fibration in the category of $\{o\,;\,c\}$-operads. We recall that the model category structure on $Operad_{\{o\,;\,c\}}$ is obtained from the adjunction
$$
\mathcal{F}:Seq(\{o\,;\,c\})\leftrightarrows Operad_{\{o\,;\,c\}}:\mathcal{U},
$$ 
where $\mathcal{U}$ is the forgetful functor and $\mathcal{F}$ is the free operad functor from pointed $\{o\,;\,c\}$-sequences (see  \cite{Berger03} or \cite{Ducoulombier14} for more details). As a consequence, the operad map $f:O\rightarrow O'$ is a fibration if and only if $\mathcal{U}(f)$ is a fibration. In particular, the map $f_{M}:=\{\mathcal{R}(O)(n)=O(c,\ldots, c;o)\rightarrow \mathcal{R}(O')(n)=O'(c,\ldots, c;o)\}$ is a fibration in the category of sequences. Furthermore, the model category structure on ($P$-$Q$) bimodules is obtained from the adjunction
$$
\mathcal{F}_{B}:Seq_{P_{0}}\leftrightarrows  Bimod_{P\text{-}Q}:\mathcal{U},
$$
where $\mathcal{F}_{B}$ is the free bimodule functor. So, a map in $Bimod_{P\text{-}Q}$ is a fibration if and only if it is a fibration in the category of  sequences. Consequently, $\mathcal{R}(f)=f_{M}$ is a fibration.
\end{proof}

\begin{rmk}
Let $O$ be an operad. Contrary to the general case, $\mathcal{L}(O\,;\,O\,;\,O)=\mathcal{L}(O)$ has a simple description. Since $\gamma:O(0)\rightarrow O(0)$ is the identity map, the axiom $(iii)$ of Construction \ref{b5} implies that there is no univalent vertices other than the root in $\mathcal{L}(O)$.  Furthermore, the $O$-bimodule $O$ has a distinguished point $\ast_{O}\in O(1)$. So, each point $x\in O(n)$, with $n\geq 1$, can be expressed as follows:
\begin{equation}\label{E8}
x=\gamma_{l}(x\,;\, \ast_{O},\ldots, \ast_{O}).
\end{equation}
As a consequence of the relation illustrated in Figure \ref{F3}, $\mathcal{L}(O)$ is the two-coloured operad given by
$$
\mathcal{L}(O)(I_{c}\,,\,I_{o}\,;\,o)\cong \mathcal{L}(O)(|I_{c}|+|I_{o}|\,;\,c)\cong O(|I_{c}|+|I_{o}|).
$$  
\end{rmk}

\begin{lmm}\label{F1}
Let $O$ be an operad. Let $\iota':O\rightarrow \mathcal{B}(O)$ be the map of sequences sending $x\in O(n)$ to $[T\,;\,1\,;\,\{a_{v}\}]$ where $T$ has a root indexed by the pair $(x\,;\,1)$ whereas the other vertices are bivalent pearl labelled by $\ast_{O}\in O(1)$. Then, the map $\iota'$ is a deformation retract and we denote by $\mu'$ the homotopy inverse.
\end{lmm}

\begin{proof}
We start by bringing the parameters above the section to $0$. By using the identification (\ref{E8}), we deduce that $\mathcal{B}(O)$ is homotopy equivalent to the sub-sequence formed by points without vertices above the section and such that the pearls are labelled by $\ast_{O}\in O(1)$. Thereafter, we bring the parameters below the section to $1$.
\begin{figure}[!h]
\begin{center}
\includegraphics[scale=0.5]{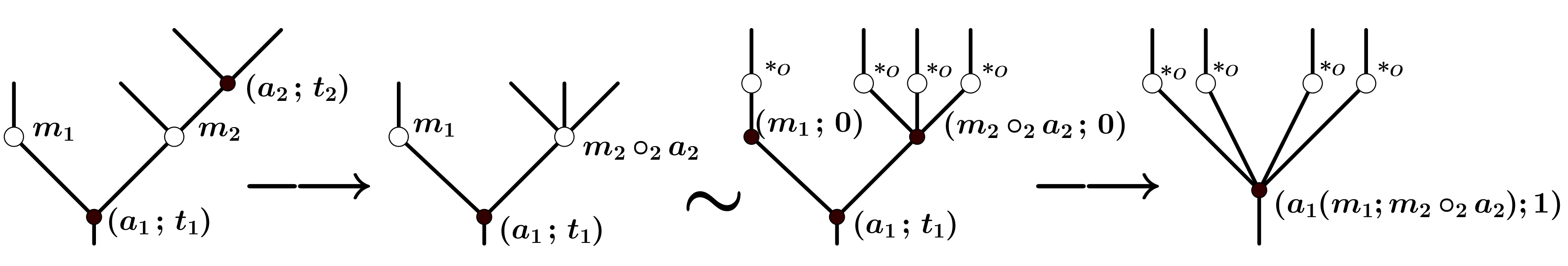}
\caption{Illustration of the homotopy.}\vspace{-15pt}
\end{center}
\end{figure}
\end{proof}

\begin{thm}\label{d3}
Let $O$ be a well pointed operad and $O'$ be an object in the category $Op[O\,;\,O]$. If  $O$ is $\Sigma$-cofibrant, then $\mathcal{L}(\mathcal{B}(O);O;O)$ is a cofibrant replacement of $\mathcal{L}(O;O;O)$ in the category $Op[O\,;\,O]$. Furthermore, the following weak equivalence holds:
\begin{equation}\label{d2}
Bimod_{O}^{h}\big(\,O\,;\,\mathcal{R}(O')\,\big)  \simeq  Op[O\,;\,O]^{h}\big(\,\mathcal{L}(O;O;O)\,;\,O'\,\big).
\end{equation}
\end{thm}

\begin{proof}
Theorem \ref{B7} implies that $\mathcal{B}(O)$ is a cofibrant replacement of $O$ in the category $Bimod_{O}$. Moreover, $\mathcal{L}$ is a left adjoint functor in a Quillen adjunction, hence $\mathcal{L}$ preserves cofibrations and cofibrant objects. As a consequence, $\mathcal{L}(\mathcal{B}(O);O;O)$ is cofibrant in the category $Op[O\,;\,O]$. 

Moreover, a point $[T\,;\,\{a_{v}\}]$ in $\mathcal{L}(\mathcal{B}(O);O;O)(I_{c},I_{o};o)$ is given by a tree $T\in \Psi(I_{c};I_{o})$ such that $a_{r}\in O(|r|)$ and $a_{v}\in \mathcal{B}(O)(|v|)$ for $v\neq r$. So, the map $\mu'$ introduced in Lemma \ref{F1} together with axiom $(iii)$ of Construction \ref{b5} induce a weak equivalence of $\{o\,;\,c\}$-operads
$$
\mathcal{L}(\mathcal{B}(O);O;O) \longrightarrow \mathcal{L}(O;O;O).
$$
Finally, $\mathcal{L}(\mathcal{B}(O);O;O)$ is a cofibrant replacement of $\mathcal{L}(O;O;O)$ in $Op[O\,;\,O]$. Since every object is fibrant in the categories considered, there are the following weak equivalences:
$$
\begin{array}{rcl}\vspace{3pt}
Bimod_{O}^{h}(O\,;\,\mathcal{R}(O')) & \simeq & Bimod_{O}(\mathcal{B}(O)\,;\,\mathcal{R}(O')), \\ 
Op[O\,;\,O]^{h}(\mathcal{L}(O;O;O)\,;\,O') & \simeq & Op[O\,;\,O](\mathcal{L}(\mathcal{B}(O);O;O)\,;\,O').
\end{array} 
$$
The weak equivalence (\ref{d2}) arises from the adjunction $(\mathcal{L}\,;\,\mathcal{R})$ which induces a homeomorphism (see \cite{Kelly04}) 
$$
Bimod_{O}\big(\,\mathcal{B}(O)\,;\,\mathcal{R}(O')\,\big) \cong Op[O\,;\,O]\big(\,\mathcal{L}(\mathcal{B}(O);O;O)\,;\,O'\,\big).\vspace{-15pt}
$$
\end{proof}\vspace{-10pt}

\subsection{Cofibrant replacement in the category \texorpdfstring{$Op[O\,;\,\emptyset]$}{Lg}}\label{h1}

The category $Op[O\,;\,\emptyset]$ is a special case of Definition \ref{d7}. The objects are pairs $(O'\,;\,f_{o})$ in which $O'$ is an $\{o\,;\,c\}$-operad and $f_{o}:O\rightarrow O'_{o}$ is a map of operads. Consequently, the objects $(O'\,;\,\tau_{O'})$ in $Op[O\,;\,O]$ can be seen in the category $Op[O\,;\,\emptyset]$ by taking the restriction $(\tau_{O'})_{o}:O\rightarrow O'_{o}$. In particular $\mathcal{L}(O;O;O)$ and  $\mathcal{L}(\mathcal{B}(O);O;O)$ are objects in $Op[O\,;\,\emptyset]$. However,  $\mathcal{L}(\mathcal{B}(O);O;O)$ is not necessarily cofibrant in $Op[O\,;\,\emptyset]$ since $O$ is not necessarily cofibrant as an operad. To solve this issue, we change slightly Construction \ref{b7} by using the Boardman-Vogt resolution introduced in Section \ref{C4}.

\begin{const}\label{h4}
Let $P$ and $Q$ be two operads. From a ($P$-$Q$) bimodule $M$, we build the ($P$-$\mathcal{BV}(Q)$) bimodule $\mathcal{B}_{\emptyset}(M)$. The points are equivalence classes $[T\,;\,\{t_{u}\}\,;\,\{a_{v}\}]$ in which $T\in \textbf{stree}$ and $\{a_{v}\}$ is a family of points labelling the vertices in the same way as in Construction \ref{b7}. The family $\{t_{u}\}$  of real numbers in the interval $[0\,,\,1]$ indexes the vertices below the section and the inner edges above the section, with the condition $t_{s(e)}\leq t_{t(e)}$ for $e$ an inner edge below the section. In other words, $\mathcal{B}_{\emptyset}(M)$ is the quotient of the sub-sequence 
$$
\left.\underset{T\in \text{stree}}{\coprod}\,\,\,\underset{v\in V^{d}(T)}{\prod}\,\big[\,P(|v|)\times I\,\big]\times\underset{v\in V^{p}(T)}{\prod}\,M(|v|)\times\underset{ v\in V^{u}(T)}{\prod}\,\big[\,Q(|v|)\times I\big]\right/ \sim\,,
$$
coming from the restriction on the families $\{t_{u}\}$. The equivalence relation is generated by axioms $(ii)$, $(iii)$  $(iv.b)$ and $(v.b)$ of Construction \ref{b7} as well as the following relations:

\begin{itemize}[itemsep=-10pt, topsep=3pt, leftmargin=*]
\item[$i')$] If a vertex is labelled by a distinguished point $\ast_{P}\in P(1)$ or $\ast_{Q}\in Q(1)$, then\vspace{-5pt}
\begin{figure}[!h]
\begin{center}
\includegraphics[scale=0.23]{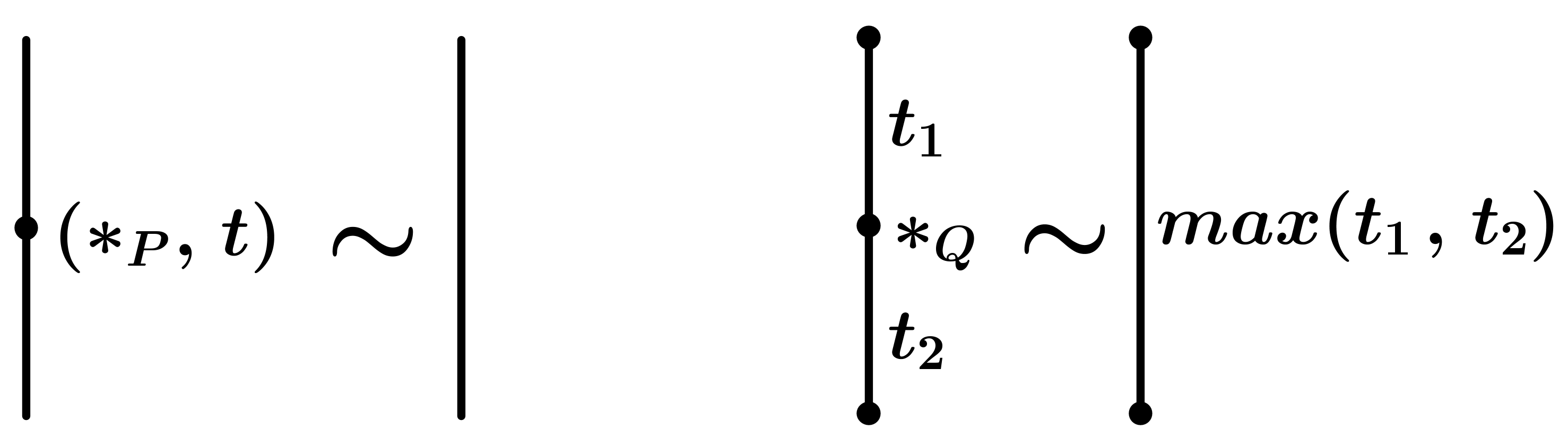}\vspace{-5pt}
\end{center}
\end{figure}
\item[$v)$] If an inner edge above the section indexed by $0$, then we contract it using the operadic structure of $Q$ or the right $Q$-bimodule structure of $M$:\vspace{-5pt}

\begin{figure}[!h]
\begin{center}
\includegraphics[scale=0.43]{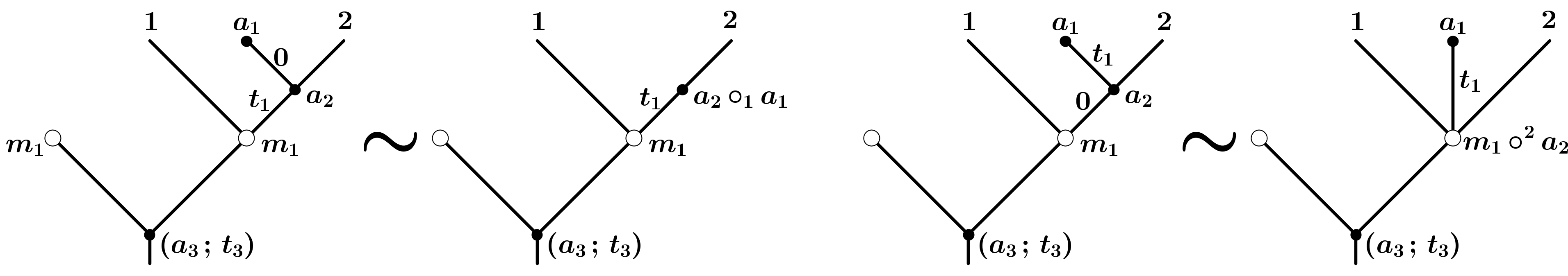}\vspace{-5pt}
\caption{Illustration of the relation $(v)$.}\vspace{-15pt}
\end{center}
\end{figure}
\end{itemize}

\noindent  The left $P$-bimodule structure on $\mathcal{B}_{\emptyset}(M)$ is similar to Construction \ref{b7}. Let $[T\,;\,\{t_{u}\}\,;\,\{a_{v}\}]$ be a point in $\mathcal{B}_{\emptyset}(M)(n)$ and $[T'\,;\,\{t'_{e}\}\,;\,\{a'_{v}\}]$ be a point in $\mathcal{BV}(Q)(m)$. The composition  $[T\,;\,\{t_{u}\}\,;\,\{a_{v}\}]\circ^{i}[T'\,;\,\{t'_{e}\}\,;\,\{a'_{v}\}]$ consists in grafting $T'$ to the $i$-th incoming edge of $T$ and indexing the new inner edge by $1$.
\end{const}

From now on, we introduce a filtration of the resolution $\mathcal{B}_{\emptyset}(M)$. Similarly to the bimodule case (see Section \ref{B6}), a point in $\mathcal{B}_{\emptyset}(M)$ is said to be prime if the real numbers indexing the set of inner edges above the section and the vertices below the section are strictly smaller than $1$. Besides, a point is said to be composite if one of its parameters is $1$ and such a point can be decomposed into prime components. A prime point is in the $k$-th filtration term $\mathcal{B}_{\emptyset}(M)_{k}$ if the number of its geometrical inputs (which is the number of leaves plus the number of univalent vertices above the section) is smaller than $k$. Then, a composite point is in the $k$-th filtration term if its prime components are in $\mathcal{B}_{\emptyset}(M)_{k}$. For each $k\geq 0$, $\mathcal{B}_{\emptyset}(M)_{k}$ is a ($P$-$\mathcal{BV}(Q)$) bimodule and the family $\{\mathcal{B}_{\emptyset}(M)_{k}\}$ gives rise a filtration of $\mathcal{B}_{\emptyset}(M)$, 
\begin{equation}\label{E4}
\xymatrix{
P_{0}\ar[r] & \mathcal{B}_{\emptyset}(M)_{0}\ar[r] & \mathcal{B}_{\emptyset}(M)_{1} \ar[r] &  \cdots \ar[r] & \mathcal{B}_{\emptyset}(M)_{k-1} \ar[r] & \mathcal{B}_{\emptyset}(M)_{k} \ar[r] & \cdots \ar[r] & \mathcal{B}_{\emptyset}(M).
}
\end{equation}  

\begin{lmm}\label{e1}
Let $P$ and $Q$ be two well pointed operads. Let $M$ be a ($P$-$Q$) bimodule  such that the map $\gamma:P(0)\rightarrow M(0)$ is a cofibration. If $P$, $Q$ and $M$ are $\Sigma$-cofibrant, then $\mathcal{B}_{\emptyset}(M)$ and $T_{k}(\mathcal{B}_{\emptyset}(M)_{k})$ are cofibrant replacements of $M$ and $T_{k}(M)$ in the category of $Bimod_{P\text{-}\mathcal{BV}(Q)}$ and $T_{k}Bimod_{P\text{-}\mathcal{BV}(Q)}$ respectively.
\end{lmm}

\begin{proof}
The proof is similar to the one of Theorem \ref{B7}. In order to fix the notation and to introduce the tower of fibrations associated to the space $Bimod_{P\text{-}\mathcal{BV}(Q)}(\mathcal{B}_{\emptyset}(M)\,;\,M')$, we show that the filtration (\ref{E4}) is composed of cofibrations. For this purpose, we consider another filtration according to the number of vertices. We recall that $\textbf{\text{stree}[k\,;\,l]}$ is the set trees with section having exactly $k$ geometrical inputs and $l$ vertices. Then, the sequence $W_{k}[l]$ is the quotient of the sub-sequence
\begin{equation}
\left.\underset{T\in \textbf{\text{stree}[k\,;\,l]}}{\coprod}\,\,\,\underset{v\in V^{d}(T)}{\prod}\,\big[\,P(|v|)\times I\,\big]\times\underset{v\in V^{p}(T)}{\prod}\,M(|v|)\times\underset{ v\in V^{u}(T)}{\prod}\,\big[\,Q(|v|)\times I\big]\right/ \sim\,,
\end{equation}
coming from the restriction on the real numbers indexing the vertices below the section. The equivalent relation is generated by the compatibility with the symmetric group axioms of Construction \ref{b7}. The sequence $\partial W_{k}[l]$ is formed by points in $W_{k}[l]$ satisfying one of the following conditions:
\begin{itemize}
\item[$\blacktriangleright$] there is a vertex below the section indexed by $0$ or $1$,
\item[$\blacktriangleright$] there are two consecutive vertices below the section indexed by the same real number,
\item[$\blacktriangleright$] there is an inner edge above the section indexed by $0$ or $1$,
\item[$\blacktriangleright$] there is a univalent pearl labelled by a point on the form $\gamma(x)$ with $x\in P(0)$,
\item[$\blacktriangleright$] there is a bivalent vertex labelled by a distinguished point $\ast_{P}\in P(1)$ or $\ast_{Q}\in Q(1)$.
\end{itemize}
For $(k\,;\,l)\neq (0\,;\,0)$, the sequences $W_{k}[l]$ and $\partial W_{k}[l]$ are not objects in the category $Seq_{P_{0}}$. So, we denote by $\tilde{W}_{k}[l]$ and $\partial \tilde{W}_{k}[l]$ the sequences obtained as follows:
$$
\tilde{W}_{k}[l](n):=
\left\{
\begin{array}{ll}\vspace{4pt}
W_{k}[l](0)\sqcup P(0) & \text{if } n=0, \\ 
W_{k}[l](n) & \text{otherwise},
\end{array} 
\right.
\hspace{15pt}\text{and}\hspace{15pt}
\partial\tilde{W}_{k}[l](n):=
\left\{
\begin{array}{ll}\vspace{4pt}
\partial W_{k}[l](0)\sqcup P(0) & \text{if } n=0, \\ 
\partial W_{k}[l](n) & \text{otherwise}.
\end{array} 
\right.
$$ 
Then, we consider the following pushout diagrams: 
$$
\xymatrix@R=20pt{
\mathcal{F}_{B}(P_{0}) \ar[r] \ar@{=}[d] & \mathcal{F}_{B}(\tilde{W}_{0}[1]) \ar@{=}[d] \\
P_{0} \ar[r] & \mathcal{B}_{\emptyset}(M)_{0}[1] 
}
\hspace{30pt}
\xymatrix@R=20pt{
\mathcal{F}_{B}(\partial \tilde{W}_{k}[l]) \ar[r] \ar[d] & \mathcal{F}_{B}(\tilde{W}_{k}[l]) \ar[d] \\
\mathcal{B}_{\emptyset}(M)_{k}[l-1] \ar[r] & \mathcal{B}_{\emptyset}(M)_{k}[l] 
}
$$ 
Similarly to the proof of Theorem \ref{B7}, we can show that the inclusion from $\partial W_{k}[l]$ to $W_{k}[l]$ is a $\Sigma$-cofibration. Since the pushout diagrams preserve the cofibrations, the map $\mathcal{B}_{\emptyset}(M)_{k}[l-1]\rightarrow \mathcal{B}_{\emptyset}(M)_{k}[l]$ is a cofibration in the category of ($P$-$\mathcal{BV}(Q)$) bimodules. Furthermore, the limit of the sequences $\mathcal{B}_{\emptyset}(M)_{k}[l]$ is $\mathcal{B}_{\emptyset}(M)_{k}$. Thus, the inclusion $\mathcal{B}_{\emptyset}(M)_{k-1}\rightarrow\mathcal{B}_{\emptyset}(M)_{k}$ is a cofibration in the category of ($P$-$\mathcal{BV}(Q)$) bimodules. As a consequence, for $(k\,;\,l)\neq (0\,;\,0)$, the vertical maps of the the following pullback diagrams are fibrations:
$$
\xymatrix@R=20pt{
Bimod_{P\text{-}\mathcal{BV}(Q)}(\mathcal{B}_{\emptyset}(M)_{k}[l]\,;\, M') \ar[r] \ar[d] & Seq(W_{k}[l]\,;\, O')\ar[d] \\
Bimod_{P\text{-}\mathcal{BV}(Q)}(\mathcal{B}_{\emptyset}(M)_{k}[l-1]\,;\, M') \ar[r] & Seq( \partial W_{k}[l] \,;\, O')
}
$$
Furthermore, if $g\in Bimod_{P\text{-}\mathcal{BV}(Q)}(\mathcal{B}_{\emptyset}(M)_{k}[l-1]\,;\, M')$, then the fiber over $g$ is homeomorphic to the mapping space from $W_{k}[l]$ to $M'$ such that the restriction to $\partial W_{k}[l]$ coincides with the map induced by $g$:
\begin{equation}\label{E5}
Seq^{g}\big(\,(W_{k}[l]\,,\, \partial W_{k}[l])\,;\, M'\,\big).
\end{equation}
\end{proof}

\begin{notat}
Let $O$ be an operad. We denote by $\mathcal{BV}_{\emptyset}(O)$ the $\{o\,;\,c\}$-sequence
$$
\mathcal{BV}_{\emptyset}(O):=\mathcal{L}(\mathcal{B}_{\emptyset}(O)\,;\,O\,;\,\mathcal{BV}(O)).
$$
\end{notat}

\begin{lmm}\label{F0}
Let $O$ be an operad. Let $\iota'':O\rightarrow \mathcal{B}_{\emptyset}(O)$ be the map of sequences sending $x\in O(n)$ to $[T\,;\,1\,;\,\{a_{v}\}]$ where $T$ has a root indexed by the pair $(x\,;\,1)$ whereas the other vertices are bivalent pearl labelled by $\ast_{O}\in O(1)$. Then, the map $\iota''$ is a deformation retract and we denote by $\mu''$ the homotopy inverse.
\end{lmm}

\begin{proof}
The proof is similar to the one's of Lemma \ref{F1}.
\end{proof}

\begin{pro}\label{d4}
Let $O$ be a well pointed operad. If $O$ is $\Sigma$-cofibrant, then $\mathcal{BV}_{\emptyset}(O)$ is a cofibrant replacement of $\mathcal{L}(O;O;O)$  in the category $Op[O\,;\,\emptyset]$.
\end{pro}

\begin{proof}
From Lemma \ref{e1}, $\mathcal{B}_{\emptyset}(O)$ is a cofibrant replacement of $O$ in the category of ($O$-$\mathcal{BV}(O)$) bimodules. Since the functor $\mathcal{L}$ preserves  cofibrant objects, $\mathcal{BV}_{\emptyset}(O)$ is cofibrant in the category $Op[O\,;\,\mathcal{BV}(O)]$. Consequently, $\mathcal{BV}_{\emptyset}(O)$ is also cofibrant in $Op[O\,;\,\emptyset]$ because the following maps are cofibrations:
$$
O \oplus \emptyset \longrightarrow O\oplus \mathcal{BV}(O) \longrightarrow \mathcal{BV}_{\emptyset}(O).
$$
A point in $\mathcal{BV}_{\emptyset}(O)$ is given by a tree in $\Psi$ vertices of which are indexed by points in $\mathcal{BV}(O)$, $\mathcal{B}_{\emptyset}(O)$ or $O$. So, the homotopy retracts $\mu''$ (introduced in Lemma \ref{F0})  and $\mu:\mathcal{BV}(O)\rightarrow O$ (see (\ref{d8})) induce a weak equivalence from $\mathcal{BV}_{\emptyset}(O)$ to $\mathcal{L}(O;O;O)$. 
\end{proof}

\subsection{Relative delooping between operad and bimodule mapping spaces}\label{F2}

As shown in the previous section, the $\{o\,;\,c\}$-sequence $\mathcal{BV}_{\emptyset}(O)$ is an object in the category $Op[O\,;\,\emptyset]$. Its restriction to the colour $c$ coincides with the operad $\mathcal{BV}(O)$. Hence, one has a continuous map coming from the restriction to the colour $c$,
\begin{equation}\label{e2}
h:Op[O\,;\,\emptyset]\big(\,\mathcal{BV}_{\emptyset}(O) \,;\,O'\,\big)  \longrightarrow  Operad\big( \,\mathcal{BV}(O) \,;\,O'_{c}\,\big).
\end{equation}
A model for the relative loop spaces (\ref{e3}) is given by the homotopy fiber of the maps (\ref{e2}) over the composite $(\tau_{O'})_{c}\circ \mu:\mathcal{BV}(Q)\rightarrow O'_{c}$. In the following definition, we give an explicit description of the homotopy fiber. 

\begin{defi}\label{e5}
A point in the homotopy fiber of (\ref{e2}) over $(\tau_{O'})_{c}\circ \mu$ is a family of continuous maps:
$$
\begin{array}{cllll} \vspace{7pt}
f[n\,;\,c] & :\,\mathcal{BV}(O)(n)\times [0\,,\,1] & \longrightarrow & O'(c,\ldots,c;c), & \text{for} \,\,n\geq 0, \\ 
f[I_{c},I_{o}\,;\,o] & :\,\mathcal{BV}_{\emptyset}(O)(I_{c},I_{o};o)\times \{1\} & \longrightarrow & O'(I_{c},I_{o};o), & \text{for}\,\, |I_{c}|\geq 0 \,\,\text{and}\,\,|I_{o}|\geq 0,
\end{array} 
$$
satisfying relations coming from the operadic structure:
\begin{itemize}[leftmargin=*, itemsep=7pt]
\item[$\blacktriangleright$] $f[n+m-1\,;\,c](x\circ_{i}y\,;\,t)=f[n\,;\,c](x\,;\,t)\circ_{i}f[m\,;\,c](y\,;\,t),$\,\, for  $x\in \mathcal{BV}(O)(n),$  and  $y\in \mathcal{BV}(O)(m)$,
\item[$\blacktriangleright$] $f[I''_{c},I''_{o}\,;\,o](x\circ_{i}y\,;\,1)=f[I_{c},I_{o}\,;\,o](x\,;\,1)\circ_{i}f[I'_{c},I'_{o}\,;\,o](y\,;\,1),$\hspace{1pt} for  $x\in \mathcal{BV}_{\emptyset}(O)(I_{c},I_{o}\,;\,o)$  and  $y\in \mathcal{BV}_{\emptyset}(O)(I'_{c},I'_{o}\,;\,o)$,
\item[$\blacktriangleright$] $f[I''_{c},I''_{o}\,;\,o](x\circ_{i}y\,;\,1)=f[I_{c},I_{o}\,;\,o](x\,;\,1)\circ_{i}f[I'_{c},\emptyset\,;\,c](y\,;\,1),$\hspace{1pt} for  $x\in \mathcal{BV}_{\emptyset}(O)(I_{c},I_{o}\,;\,o)$  and  $y\in \mathcal{BV}_{\emptyset}(O)(I'_{c},\emptyset\,;\,c)$,
\end{itemize}\vspace{4pt}
and relations coming from the based point:
\begin{itemize}[leftmargin=*]
\item[$\blacktriangleright$] $f[n\,;\,c](x\,;\,0)=(\tau_{O'})_{c}\circ\mu (x),$ \hspace{1pt} for \hspace{4pt} $x\in \mathcal{BV}(O)(n)$.
\end{itemize}
\end{defi}

In order to prove Theorem \ref{d9}, we use a method introduced by the author in \cite{Ducoulombier14} for the cosimplicial case. We consider the topological space $Op[O\,;\,\mathcal{BV}(O)](\mathcal{BV}_{\emptyset}(O)\,;\,O')$ as an intermediate space between the relative loop space and $Op[O\,;\,O]^{h}(\mathcal{L}(O)\,;\,O')$. More precisely, $Op[O\,;\,\mathcal{BV}(O)](\mathcal{BV}_{\emptyset}(O)\,;\,O')$ is the subspace of $Op[O\,;\,\emptyset](\mathcal{BV}_{\emptyset}(O)\,;\,O')$ formed by maps $f$ satisfying the relation: 
\begin{equation}\label{e4}
f(x) = (\tau_{O'})_{c}\circ \mu(x) \,\,\,\,\,\,\,\,\,\forall x\in \mathcal{BV}(O).
\end{equation}

\begin{pro}\label{d5}
Under the assumptions of Theorem \ref{d9}, the following weak equivalence holds:
$$
Op[O\,;\,\mathcal{BV}(O)](\,\mathcal{BV}_{\emptyset}(O)\,;\,O'\,)  \simeq  Op[O\,;\,O](\mathcal{L}(\,\mathcal{B}(O);O;O)\,;\,O'\,).
$$
\end{pro}

\begin{proof}
Let us notice that this proposition is inspired by constructions introduced by Fresse in \cite[Chapter 7]{Fresse09} and the author in \cite[Proposition 4.4]{Ducoulombier14}.  From Corollary \ref{d3}, to obtain the result it is sufficient to build a weak equivalence
\begin{equation}\label{h9}
\xi:  Bimod_{O}\big(\,\mathcal{B}(O)\,;\,\mathcal{R}(O')\,\big)\longrightarrow Bimod_{O\text{-}\mathcal{BV}(O)}\big(\,\mathcal{B}_{\emptyset}(O)\,;\,\mathcal{R}(O')\,\big).
\end{equation}
Firstly, we build a set map $i:\mathcal{B}(O)\rightarrow \mathcal{B}_{\emptyset}(O)$ sending  a point $[T\,;\,\{t_{v}\}\,;\,\{a_{v}\}]$ to  $[T\,;\,\{t'_{u}\}\,;\,\{a_{v}\}]$ where the tree with section, the indexation of vertices below the section and the family $\{a_{v}\}$ labelling the vertices are still the same. If $e$ is an inner edge above the section, then this edge is indexed by the real number $t'_{e}\in [0\,,\,1]$ defined as follows: 
$$
t'_{e}=\left\{
\begin{array}{cc}\vspace{3pt}
(t_{t(e)}-t_{s(e)})/(t_{t(e)}-1) & \text{if} \,\, t_{t(e)}<1, \\ 
1 & \text{if} \,\, t_{t(e)}=1.
\end{array} 
\right.
$$ 
By convention, we assume that the pearls are indexed by $0$. The function $i$ so obtained doesn't depend on the choice of the point in the equivalence class. Unfortunately, the function $i$ is not a continuous map. Indeed, if $e$ is an inner edge connecting two vertices in the set $V^{u}(T)$, then $t'_{e}$ is not well defined as $t_{t(e)}$ approaches $1$.\vspace{-10pt}

\begin{center}
\begin{figure}[!h]
\begin{center}
\includegraphics[scale=0.5]{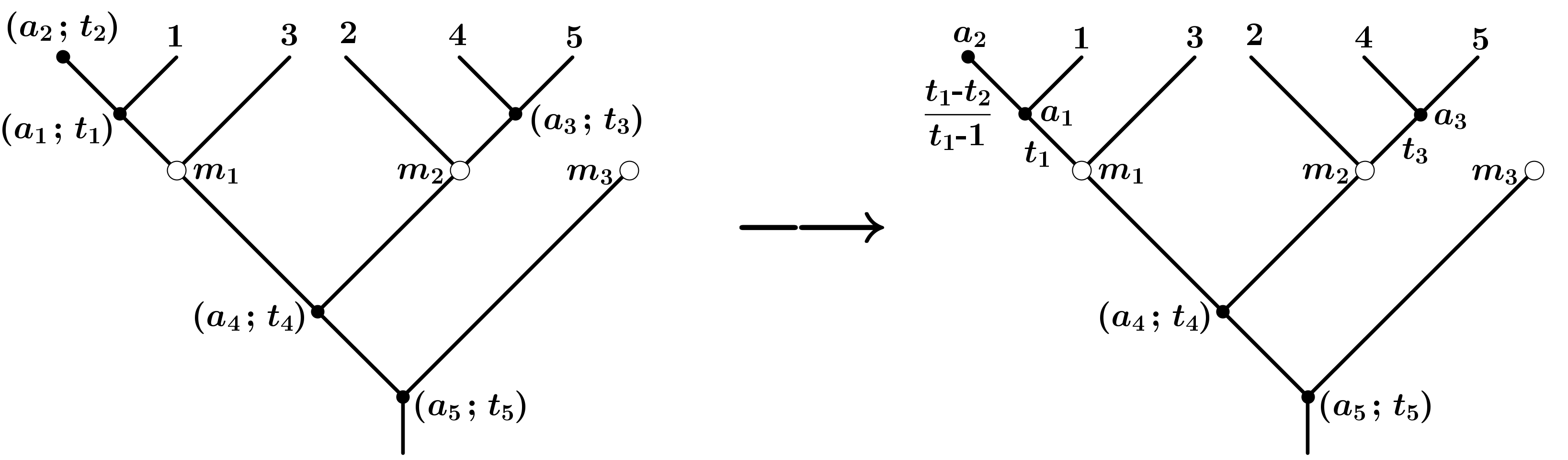}\\
\caption{Illustration of the map $i:\mathcal{B}(O)\rightarrow \mathcal{B}_{\emptyset}(O)$.}\vspace{-10pt}
\end{center}
\end{figure}
\end{center}

\noindent To solve this issue, we introduce an equivalence relation on the sequence $\mathcal{B}_{\emptyset}(O)$. Let $\sim$ be the equivalence relation generated by $x\sim x'$ if and only if there exist $[T\,;\,\{t^{1}_{e}\}\,;\,\{a_{v}\}]$, $[T\,;\,\{t^{2}_{e}\}\,;\,\{a_{v}\}]\in\mathcal{BV}(O)$ and $y\in\mathcal{B}_{\emptyset}(O)$ such that:
\begin{equation}\label{E7}
x=y \circ^{i}[T\,;\,\{t^{1}_{e}\}\,;\,\{a_{v}\}] \hspace{30pt} \text{and} \hspace{30pt} x'=y \circ^{i}[T\,;\,\{t^{2}_{e}\}\,;\,\{a_{v}\}].
\end{equation}
We denote by $\mathcal{B}_{\emptyset}(O)/\!\!\!\sim\!\!(n)$ the quotient space $\mathcal{B}_{\emptyset}(O)(n)/\!\!\sim$. The ($O$-$\mathcal{BV}(O)$) bimodule structure on $\mathcal{B}_{\emptyset}(O)$ induces a $O$-bimodule structure on the sequence $\mathcal{B}_{\emptyset}(O)/\!\!\sim$\,:\vspace{5pt}
\begin{itemize}[itemsep=4pt]
\item[$\blacktriangleright$] $\circ^{i} \,:\, \mathcal{B}_{\emptyset}(O)/\!\!\!\sim\!\!(n) \times O(m) \longrightarrow \mathcal{B}_{\emptyset}(O)/\!\!\!\sim\!\!(n+m-1);$\\
$\phantom{.}\hspace{36pt}(x\,;\,y)\hspace{28pt}\longmapsto \hspace{28pt} x\,\circ^{i}\,y,$
\item[$\blacktriangleright$] $\gamma_{l} \,:\, O(n)\times \mathcal{B}_{\emptyset}(O)/\!\!\!\sim\!\!(m_{1})\times \cdots \times \mathcal{B}_{\emptyset}(O)/\!\!\!\sim\!\!(m_{n})\longrightarrow \mathcal{B}_{\emptyset}(O)/\!\!\!\sim\!\!(m_{1}+\cdots + m_{n});$\\
$\phantom{.}\hspace{54pt}(x\,,\,y_{1}\,,\ldots,y_{n})\hspace{60pt}\longmapsto \hspace{28pt} \gamma_{l}(x;y_{1}\,,\ldots,y_{n}).$
\end{itemize}
The $O$-bimodule axioms are satisfied thanks to the equivalence relation. The left $\mathcal{BV}(O)$-bimodule structure on $O'$ arises from the map of operads $\mu:\mathcal{BV}(O)\rightarrow O$ (see (\ref{d8})). As a consequence, each bimodule map $f\in Bimod_{O\text{-}\mathcal{BV}(O)}(\mathcal{B}_{\emptyset}(O)\,;\,\mathcal{R}(O'))$ preserves  equivalence classes and induces a map $\tilde{f}\in Bimod_{O}(\mathcal{B}_{\emptyset}(O)/\!\!\!\sim\,;\,\mathcal{R}(O'))$. By using the universal property of the quotient, one has a continuous bijection 
$$
\xi_{1}: Bimod_{O}(\mathcal{B}_{\emptyset}(O)/\!\!\!\sim\,;\,\mathcal{R}(O'))\longrightarrow Bimod_{O\text{-}\mathcal{BV}(O)}(\mathcal{B}_{\emptyset}(O)\,;\,\mathcal{R}(O')).
$$
Let $e$ be an inner edge above the section connecting two vertices in $V^{u}(T)$. If $t_{t(e)}$ is equal to $1$, then all the parameters $t'_{e}$ are identified in the quotient space. So, the function $i:\mathcal{B}(O)\rightarrow  \mathcal{B}_{\emptyset}(O)/\!\!\!\sim$ is a continuous map. By construction, the map $i$ is a homeomorphism and preserves the $O$-bimodule structures.
Consequently, we obtain the following homeomorphism coming from the composition with $i$:
$$
\xi_{2}: Bimod_{O}(\mathcal{B}_{\emptyset}(O)/\!\!\!\sim\,;\,\mathcal{R}(O')) \longrightarrow Bimod_{O}(\mathcal{B}(O)\,;\,\mathcal{R}(O')).
$$
 
It is not obvious that $\xi_{1}$ is a homeomorphism since our cofibrant replacements are not necessarily finite $CW$-complexes in each arity. So, we will prove that $\xi=\xi_{1}\circ \xi_{2}^{-1}$ is a weak equivalence by using towers of fibrations associated to the filtrations (\ref{A8}) and (\ref{E4}). Indeed, $\xi$ induces a morphism between the following towers of fibrations:
$$
\xymatrix{
Bimod_{O}(\mathcal{B}_{0}(O)[1]\,;\,\mathcal{R}(O')) \ar[d]^{\xi^{0\,;\,1}}& \ar[l] \cdots  & \ar[l] Bimod_{O}(\mathcal{B}_{k}(O)[l]\,;\,\mathcal{R}(O'))\ar[d]^{\xi^{k\,;\,l}} &\ar[l] \cdots \\
 Bimod_{O\text{-}\mathcal{BV}(O)}(\mathcal{B}_{\emptyset}(O)_{0}[1]\,;\,\mathcal{R}(O')) & \ar[l]\cdots  & \ar[l] Bimod_{O\text{-}\mathcal{BV}(O)}(\mathcal{B}_{\emptyset}(O)_{k}[l]\,;\,\mathcal{R}(O')) &\ar[l] \cdots
}
$$
Consequently, $\xi$ is a weak equivalence if the maps $\xi^{k\,;\,l}$ are weak equivalences. Since the points in $\mathcal{B}_{0}(O)[1]$ and $\mathcal{B}_{\emptyset}(O)_{0}[1]$ are indexed by trees with section without vertices above the section, the map $\xi^{0\,;\,1}$ is the identity map which is a weak equivalence. Assume that $\xi^{k\,;\,l-1}$ is a weak equivalence and $g$ is a point in $Bimod_{O}(\mathcal{B}_{k}(O)[l-1]\,;\,\mathcal{R}(O'))$. We consider the following diagram:
$$
\xymatrix{
Bimod_{O}(\mathcal{B}_{k}(O)[l-1]\,;\,\mathcal{R}(O')) \ar[d]^{\xi^{k\,;\,l-1}}_{\simeq} & \ar[l] Bimod_{O}(\mathcal{B}_{k}(O)[l]\,;\,\mathcal{R}(O'))\ar[d]^{\xi^{k\,;\,l}} &\ar[l] F_{1} \ar[d]^{\xi^{g}} \\
 Bimod_{O\text{-}\mathcal{BV}(O)}(\mathcal{B}_{\emptyset}(O)_{k}[l-1]\,;\,\mathcal{R}(O')) & \ar[l] Bimod_{O\text{-}\mathcal{BV}(O)}(\mathcal{B}_{\emptyset}(O)_{k}[l]\,;\,\mathcal{R}(O')) &\ar[l] F_{2}
}
$$
where $F_{1}$ is the fiber over $g$ and $F_{2}$ is the fiber over $\xi^{k\,;\,l-1}(g)$. Since the left horizontal maps are fibrations, $\xi^{k\,;\,l}$ is a weak equivalence if the map between the fiber is a weak equivalence. By using the identifications (\ref{E6}) and (\ref{E5}), one has
$$
\xymatrix@R=15pt{
F_{1} \ar[dd]^{\xi^{g}} & \ar[l]_{\hspace{-60pt}\cong}  Seq^{g}\big( \, (X_{k}[l]\,,\, \partial X_{k}[l]) \,;\, \mathcal{R}(O')  \,\big)\\
& Seq^{\tilde{g}}\big(\,(W_{k}[l]/\!\sim \,,\, \partial W_{k}[l]/\!\sim)\,;\, \mathcal{R}(O')\,\big) \ar[u]^{\xi_{2}^{\ast}}_{\cong} \ar[d]_{\xi_{1}^{\ast}}\\
F_{2} & \ar[l]_{\hspace{-60pt}\cong} Seq^{\xi^{k\,;\,l-1}(g)}\big(\,(W_{k}[l]\,,\, \partial W_{k}[l])\,;\, \mathcal{R}(O')\,\big)
}
$$
where $\sim$ is the equivalence relation generated by (\ref{E7}). Since we are in the fiber over $\xi^{k\,;\,l-1}(g)$ the induced map $\tilde{g}$ is well defined. In the same way as $X_{k}[l]$ in the proof of Theorem \ref{B7}, $W_{k}[l]$ can also be expressed as follows:
$$
W_{k}[l]:=\underset{[T\,;\,;V^{p}(T)]}{\coprod} (H'(T)\times \underline{M}(T)) \underset{Aut(T\,;\,V^{p}(T))}{\times} \Sigma_{|T|},
$$
where $H'(T)$ is the space of parametrization of the vertices below the section and the inner edges above the section of $T$ by real numbers satisfying the restriction introduced in Construction \ref{h4}. Thus, the quotient map $W_{k}[l]\rightarrow W_{k}[l]/\!\sim$ is proper since its restriction to $\underline{M}(T)$ is the identity map, $H'(T)$ is a finite $CW$-complex and $\textbf{stree}^{p}\textbf{[k\,;\,l]}$ is a finite set. In that case, $\xi_{1}^{\ast}$ is a homeomorphism and $\xi^{k\,;\,l}$ is a weak equivalence.
\end{proof}

\begin{pro}\label{d6}
Under the assumptions of Theorem \ref{d9}, the following weak equivalence holds:
$$
Op[O\,;\,\mathcal{BV}(O)](\,\mathcal{BV}_{\emptyset}(O)\,;\,O'\,)  \simeq  \Omega\big(\,\,Operad(\,\mathcal{BV}(O)\,;\,O'_{c}\,)\,\,;\,\, Op[O\,;\,\emptyset](\,\mathcal{BV}_{\emptyset}(O)\,\,;\,\,O'\,)\,\,\big).
$$ 
\end{pro}

\begin{proof}
Let us recall that $Op[O\,;\,\mathcal{BV}(O)](\,\mathcal{BV}_{\emptyset}(O)\,;\,O'\,)$ is the subspace of $Op[O\,;\,\emptyset](\,\mathcal{BV}_{\emptyset}(O)\,\,;\,\,O'\,)$ formed by maps $f$ satisfying the additional condition (\ref{e4}). In particular, the space $Op[O\,;\,\mathcal{BV}(O)](\,\mathcal{BV}_{\emptyset}(O)\,;\,O'\,)$ is a subspace of the homotopy fiber (\ref{e2}), described in Definition \ref{e5}, where the inclusion is given by
$$
\begin{array}{ccl}\vspace{10pt}
\iota  :\,Op[O\,;\,\mathcal{BV}(O)](\,\mathcal{BV}_{\emptyset}(O)\,;\,O'\,) & \longrightarrow & \Omega\big(\,\,Operad(\,\mathcal{BV}(O)\,;\,O'_{c}\,)\,\,;\,\, Op[O\,;\,\emptyset](\,\mathcal{BV}_{\emptyset}(O)\,\,;\,\,O'\,)\,\,\big); \\ 
  f & \longmapsto & \left\{
 \begin{array}{cl}\vspace{3pt}
 \tilde{f}(n\,;\,c)(x\,;\,t)  =  (\tau_{O'})_{c}\circ\mu (x) & \text{for}\,\, x\in \mathcal{BV}(O)(n), \\ 
 \tilde{f}(I_{c},I_{o}\,;\,o)(x\,;\,1) =  f(I_{c},I_{o}\,;\,o)(x) & \text{for}\,\, x\in \mathcal{BV}_{\emptyset}(O)(I_{c},I_{o}\,;\,o).
 \end{array} 
 \right.
\end{array} 
$$
Let us prove that $\iota$ is a weak equivalence. As shown in the previous sections, the operad $\mathcal{BV}(O)$ and the bimodule $\mathcal{B}_{\emptyset}(O)$ have a filtration according to the number of geometrical inputs and the number of vertices. So, we will define a similar filtration for the operad $\mathcal{BV}_{\emptyset}(O)$ in order to build towers of fibrations. For this purpose, we introduce the $\{o\,;\,c\}$-sequences $V_{k}[l]$ and $\partial V_{k}[l]$ defined as follows:
$$
\left\{
\begin{array}{lcl}\vspace{4pt}
V_{k}[l](\emptyset\,,\,I_{o}\,;\,o) & = & O(|I_{o}|), \text{ if } |I_{o}|\geq 1, \\ \vspace{4pt}
V_{k}[l](I_{c}\,,\,\emptyset\,;\,o) & = & W'_{k}[l](|I_{c}|), \\ 
V_{k}[l](I_{c}\,,\,\emptyset\,;\,c) & = & Y_{k}[l](|I_{c}|),
\end{array} 
\right.
\hspace{15pt}\text{and}\hspace{15pt}
\left\{
\begin{array}{lcl}\vspace{4pt}
\partial V_{k}[l](\emptyset \,,\,I_{o}\,;\,o) & = & O(|I_{o}|),\text{ if } |I_{o}|\geq 1,\\ \vspace{4pt}
\partial V_{k}[l](I_{c}\,,\,\emptyset\,;\,o) & = & \partial W'_{k}[l](|I_{c}|), \\ 
\partial V_{k}[l](I_{c}\,,\,\emptyset\,;\,c) & = & \partial Y_{k}[l](|I_{c}|),
\end{array} 
\right.
$$
where $Y_{k}[l]$ and $\partial Y_{k}[l]$ are the sequences introduced in the proof of Theorem \ref{B9} whereas $W'_{k}[l]$ and $\partial W'_{k}[l]$ are the sub-sequences of $W_{k}[l]$ and $\partial W_{k}[l]$, introduced in the proof of Theorem \ref{e1}, formed by points indexed by trees with section without univalent pearl other than the root. We consider these sub-sequences because $\mathcal{B}_{\emptyset}(O)$ doesn't have univalent pearl due to axiom $(iii)$ of Construction \ref{b7}.

For $(k\,;\,l)\neq (1\,;\,1)$, the $\{o\,;\,c\}$-sequences  $V_{k}[l]$ and $\partial V_{k}[l]$ are not pointed. So, we denote by $\tilde{V}_{k}[l]$ and $\partial \tilde{V}_{k}[l]$ the $\{o\,;\,c\}$-sequences obtained from $V_{k}[l]$ and $\partial V_{k}[l]$ by taking $\tilde{Y}_{k}[l]$ and $\partial \tilde{Y}_{k}[l]$ instead of $Y_{k}[l]$ and $\partial Y_{k}[l]$ respectively. Then, we consider the following pushout diagrams:
$$
\xymatrix{
\mathcal{F}(O\oplus \emptyset) \ar[r]\ar@{=}[d] & \mathcal{F}(\tilde{V}_{1}[1])\ar[d] \\
O\oplus \emptyset \ar[r] & \mathcal{BV}_{\emptyset}(O)_{1}[1] 
}
\hspace{30pt}
\xymatrix{
\mathcal{F}(\partial \tilde{V}_{k}[l]) \ar[r]\ar[d] & \mathcal{F}(\tilde{V}_{k}[l])\ar[d] \\
\mathcal{BV}_{\emptyset}(O)_{k}[l-1] \ar[r] & \mathcal{BV}_{\emptyset}(O)_{k}[l] 
}
$$
By construction, $\{\mathcal{BV}_{\emptyset}(O)_{k}[l]\}$ gives rise a filtration of $\mathcal{BV}_{\emptyset}(O)$ in which $\mathcal{BV}_{\emptyset}(O)_{k}[l]$ is a two-coloured operad endowed with a map from $O\oplus \mathcal{BV}_{k}(O)[l]$. 

Now, we are able to introduce the towers of fibrations showing that the map $\iota$ is a weak equivalence. For the space $Op[O\,;\,\mathcal{BV}(O)](\,\mathcal{BV}_{\emptyset}(O)\,;\,O'\,)$, the additional relation (\ref{e4}) implies that we have to consider the $\{o\,;\,c\}$-sequence
$$
\left\{
\begin{array}{lcl}\vspace{4pt}
\partial V^{1}_{k}[l](\emptyset \,,\,I_{o}\,;\,o) & = & O(|I_{o}|)\\ \vspace{4pt}
\partial V^{1}_{k}[l](I_{c}\,,\,\emptyset\,;\,o) & = & \partial W'_{k}[l](|I_{c}|), \\ 
\partial V^{1}_{k}[l](I_{c}\,,\,\emptyset\,;\,c) & = & Y_{k}[l](|I_{c}|).
\end{array} 
\right.
$$
Since the inclusion from $\partial V^{1}_{k}[l]$ to $V_{k}[l]$ is a $\Sigma$-cofibration, the vertical maps of the following pullback diagram are fibrations:
\begin{equation}\label{F4}
\xymatrix{
Op[O\,;\,\mathcal{BV}_{k}(O)[l]](\,\mathcal{BV}_{\emptyset}(O)_{k}[l]\,;\,O'\,)\ar[d]\ar[r] & Seq(\,V_{k}[l]\,;\,O'\,)\ar[d]\\
Op[O\,;\,\mathcal{BV}_{k}(O)[l-1]](\,\mathcal{BV}_{\emptyset}(O)_{k}[l-1]\,;\,O'\,) \ar[r] & Seq(\,\partial V^{1}_{k}[l]\,;\,O'\,)
}
\end{equation}
For the space $\Omega (\,Operad(\mathcal{BV}(O)\,;\,O'_{c}\,)\,;\, Op[O\,;\,\emptyset](\,\mathcal{BV}_{\emptyset}(O)\,\,;\,\,O'\,)\,)$, the relations introduced in Definition \ref{e5} implies that we have to consider the $\{o\,;\,c\}$-sequences
$$
\left\{\hspace{-2pt}
\begin{array}{l}\vspace{4pt}
V^{2}_{k}[l](\emptyset\,,\,I_{o}\,;\,o)  =  O(|I_{o}|) \\ \vspace{4pt}
V^{2}_{k}[l](I_{c}\,,\,\emptyset\,;\,o)  =  W'_{k}[l](|I_{c}|), \\ 
V^{2}_{k}[l](I_{c}\,,\,\emptyset\,;\,c)  =  Y_{k}[l](|I_{c}|)\times [0\,,\,1],
\end{array} 
\right.
\hspace{-1pt}\text{and}\hspace{5pt}
\left\{\hspace{-2pt}
\begin{array}{l}\vspace{4pt}
\partial V^{2}_{k}[l](\emptyset \,,\,I_{o}\,;\,o)  =  O(|I_{o}|)\\ \vspace{4pt}
\partial V^{2}_{k}[l](I_{c}\,,\,\emptyset\,;\,o)  =  \partial W'_{k}[l](|I_{c}|), \\ 
\partial V^{2}_{k}[l](I_{c}\,,\,\emptyset\,;\,c)  =  (Y_{k}[l](|I_{c}|)\times \{0\})\hspace{-10pt} \underset{\partial Y_{k}[l](|I_{c}|)\times \{0\}}{\displaystyle \coprod}\hspace{-10pt} (\partial Y_{k}[l](|I_{c}|)\times [0\,,\,1]).
\end{array} 
\right.
$$
Since the inclusion from $\partial V^{2}_{k}[l]$ to $V^{2}_{k}[l]$ is a $\Sigma$-cofibration, the vertical maps of the following pullback diagram are fibrations:
\begin{equation}\label{F5}
\xymatrix{
\Omega (\,Operad(\mathcal{BV}_{k}(O)[l]\,;\,O'_{c}\,)\,;\, Op[O\,;\,\emptyset](\,\mathcal{BV}_{\emptyset}(O)_{k}[l]\,\,;\,\,O'\,)\,)\ar[d]\ar[r] & Seq(\,V^{2}_{k}[l]\,;\,O'\,)\ar[d]\\
\Omega (\,Operad(\mathcal{BV}_{k}(O)[l-1]\,;\,O'_{c}\,)\,;\, Op[O\,;\,\emptyset](\,\mathcal{BV}_{\emptyset}(O)_{k}[l-1]\,\,;\,\,O'\,)\,)\ar[r] & Seq(\,\partial V^{2}_{k}[l]\,;\,O'\,)
}
\end{equation}
The map $\iota$ induces a morphism between the two towers of fibrations:
$$
\xymatrix@R=15pt{
\vdots\ar[d] & \vdots\ar[d] \\
Op[O\,;\,\mathcal{BV}_{k}(O)[l]](\,\mathcal{BV}_{\emptyset}(O)_{k}[l]\,;\,O'\,) \ar[r]^{\hspace{-40pt}\iota_{k\,;\,l}}\ar[d] & \Omega (\,Operad(\mathcal{BV}_{k}(O)[l]\,;\,O'_{c}\,)\,;\, Op[O\,;\,\emptyset](\,\mathcal{BV}_{\emptyset}(O)_{k}[l]\,\,;\,\,O'\,)\,)\ar[d]\\
\vdots\ar[d] & \vdots\ar[d] \\
Op[O\,;\,\mathcal{BV}_{1}(O)[1]](\,\mathcal{BV}_{\emptyset}(O)_{1}[1]\,;\,O'\,) \ar[r]^{\hspace{-40pt}\iota_{1\,;\,1}} & \Omega (\,Operad(\mathcal{BV}_{1}(O)[1]\,;\,O'_{c}\,)\,;\, Op[O\,;\,\emptyset](\,\mathcal{BV}_{\emptyset}(O)_{1}[1]\,\,;\,\,O'\,)\,)
}
$$
Since the vertical maps are fibrations, in order to prove that the map $\iota$ is a weak equivalence, it is sufficient to show by induction that the maps $\iota_{k\,;\,l}$ are weak equivalences. By using the notation of Definition \ref{e5}, a point in the relative loop space $\Omega (\,Operad(\mathcal{BV}_{1}(O)[1]\,;\,O'_{c}\,)\,;\, Op[O\,;\,\emptyset](\,\mathcal{BV}_{\emptyset}(O)_{1}[1]\,\,;\,\,O'\,)\,)$ is determined by a family $\{f\}$ of continuous maps 
$$
f[c\,;\,o]:O(1)\times \{1\}\longrightarrow O'(c\,;\,o) \hspace{15pt}\text{and}\hspace{15pt}
\left\{
\begin{array}{l}\vspace{4pt}
f[0\,;\,c]:O(0)\times [0\,,\,1]\longrightarrow O'(\,;\,c), \\ 
f[1\,;\,c]:O(0)\times [0\,,\,1]\longrightarrow O'(1\,;\,c),
\end{array} 
\right.
$$
satisfying the relations $f[0\,;\,c](x\,;\,0)=\eta(x)$ and $f[1\,;\,c](x\,;\,0)=\eta(x)$. On the other hand, the image of $\iota_{1\,;\,1}$ is formed by maps satisfying also the conditions $f[0\,;\,c](x\,;\,t)=\eta(x)$ and $f[1\,;\,c](x\,;\,t)=\eta(x)$ with $t\in [0\,,\,1]$. So, $\iota_{1\,;\,1}$ is a homotopy equivalence and the homotopy $H$ is defined as follows:
$$
\left\{
\begin{array}{l}\vspace{4pt}
H(\{f\}\,;\,u)[c\,;\,o](x\,;\,1)=f[c\,;\,o](x\,;\,1),  \\ \vspace{4pt}
H(\{f\}\,;\,u)[0\,;\,c](x\,;\,t)=f[0\,;\,c](x\,;\,(1-u)t),  \\ 
H(\{f\}\,;\,u)[1\,;\,c](x\,;\,t)=f[1\,;\,c](x\,;\,(1-u)t),
\end{array} 
\right.
$$
with $u\in [0\,,\,1]$ and $\{f\}$ a point in the relative loop space.

Assume that the map $\iota_{k\,;\,l-1}$ is a weak equivalence and $g$ is a point in the space $\Omega (\,Operad(\mathcal{BV}_{k}(O)[l-1]\,;\,O'_{c}\,)\,;\, Op[O\,;\,\emptyset](\,\mathcal{BV}_{\emptyset}(O)_{k}[l-1]\,\,;\,\,O'\,)\,)$. Let $F_{1}$ and $F_{2}$ be the fiber over $g$ and $\iota_{k\,;\,l-1}(g)$ respectively. So, the map $\iota_{k\,;\,l}$ is a weak equivalence if the map between the fibers $\iota_{g}$ is a weak equivalence. By using the pullback diagrams \ref{F4} and \ref{F5}, one has 
$$
\xymatrix{
F_{1} \ar[d]^{\iota_{g}}  & Seq^{g}((V_{k}[l]\,;\,\partial V^{1}_{k}[l])\,;\,O')\ar[d]^{\alpha}\ar[l]_{\hspace{-50pt}\cong}\\
F_{2} & Seq^{\iota_{k\,;\,l-1}(g)}((V^{2}_{k}[l]\,;\,\partial V^{2}_{k}[l])\,;\,O')\ar[l]_{\hspace{-50pt}\cong}
}
$$
The image of $\alpha$ is the subspace formed by maps $\{f\}$ satisfying the relation $f[|I_{c}|\,;\,c](x\,;\,t)=f[|I_{c}|\,;\,c](x\,;\,0)$ for $x\in V_{k}[l](I_{c}\,,\,\emptyset\,;\,c)$ and $t\in [0\,,\,1]$. Since $F_{2}$ is the fiber over $\iota_{k\,;\,l-1}(g)$, a point $\{f\}\in  Seq^{\iota_{k\,;\,l-1}(g)}((V^{2}_{k}[l]\,;\,\partial V^{2}_{k}[l])\,;\,O')$ satisfies the relation 
$$
f[|I_{c}|\,;\,c](x\,;\,t)=f[|I_{c}|\,;\,c](x\,;\,0), \hspace{15pt}\text{for } x\in \partial V_{k}[l](I_{c}\,,\,\emptyset\,;\,c).
$$
So, $\alpha$ is a homotopy equivalence and the homotopy $H'$ below is well defined:
$$
\left\{
\begin{array}{l}\vspace{4pt}
H'(\{f\}\,;\,u)[I_{c}\,,\,I_{o}\,;\,c](x\,;\,1)=f[I_{c}\,,\,I_{o}\,;\,c](x\,;\,1),  \\ 
H'(\{f\}\,;\,u)[|I_{c}|\,;\,c](x\,;\,1)=f[|I_{c}|\,;\,c](x\,;\,(1-u)t),
\end{array} 
\right.
$$
with $u\in [0\,,\,1]$ and $\{f\}$ a point in the space $Seq^{\iota_{k\,;\,l-1}(g)}((V^{2}_{k}[l]\,;\,\partial V^{2}_{k}[l])\,;\,O')$. Thus proves that the map $\iota$ is a weak equivalence. 
\end{proof}

The next theorem is the main result of this paper. It is a direct consequence of Theorem \ref{d9} and Theorem \ref{H1}. Indeed, the weak equivalence in \ref{H1} identifies explicit loop spaces from maps between operads. Similarly, we identify relative loop spaces in Theorem \ref{d9} which are compatible with the loops spaces in the sense that they form typical algebras over the Swiss-Cheese operad $\mathcal{SC}_{1}$.

\begin{thm}\label{f3}
Let $O$ be a well pointed $\Sigma$-cofibrant operad. Let $\eta:\mathcal{L}(O)\rightarrow O'$ be a map in the category $Op[O\,;\,O]$ in which the spaces $O(1)$ and $O'_{c}(1)$ are contractible. Then, the pair of topological spaces
$$
\big(\,Bimod^{h}_{O}(O\,;\,O'_{c})\,\,;\,\,Bimod^{h}_{O}(O\,;\,\mathcal{R}(O'))\,\big)
$$ 
is weakly equivalent to the $\mathcal{SC}_{1}$-algebra
$$
\big(\,\, \Omega\big(\,Operad^{h}(O\,;\,O'_{c})\,\big)\,\,;\,\,
\Omega\big(\,\,Operad^{h}(\,O\,;\,O'_{c}\,)\,\,;\,\, Op[O\,;\,\emptyset]^{h}(\,\mathcal{L}(O;O;O)\,\,;\,\,O'\,)\,\,\big)\,\,\big).
$$
\end{thm}

\begin{cor}\label{f6}
Let $\eta_{1}:O\rightarrow O'$ be a map of operads and $\eta_{2}:O'\rightarrow M$ be a map of $O'$-bimodules. Assume that $O$ is a well pointed $\Sigma$-cofibrant operad and the spaces $O(1)$ and $O'(1)$ are contractible. Then, the pair of topological spaces
$$
\big(\, Bimod^{h}_{O}(O\,;\, O')\,\,;\,\,  Bimod^{h}_{O}(O\,;\, M)\,\big)
$$  
is weakly equivalent to the $\mathcal{SC}_{1}$-algebra
$$
\left(\,\,
\Omega\big(\,  Operad^{h}(O\,;\,O')\,\big)\,\,;\,\,\Omega\big(\,\, Operad^{h}(\,O\,;\,O'\,)\,\,;\,\, Op[O\,;\,\emptyset]^{h}(\,\mathcal{L}(O;O;O)\,\,;\,\,\mathcal{L}(M;O';O')\,)\,\,\big)\,\, 
\right).
$$
\end{cor}

\begin{proof}
As shown in Remark \ref{F8}, the maps $\eta_{1}:O\rightarrow O'$ and $\eta_{2}:O'\rightarrow M$ induce a map of $\{o\,;\,c\}$-operads $f:\mathcal{L}(O;O;O)\rightarrow\mathcal{L}(M;O';O')$. So, the corollary is a consequence of Theorem \ref{f3} applied to the map $f$.
\end{proof}

\subsection{Truncated version of the relative delooping}\label{H2}

In what follows, we denote by $T_{k}Op[P\,;\,Q]$ the category of $k$-truncated $\{o\,;\,c\}$-operads under $T_{k}(P\oplus Q)$. This section is devoted to the proof of the following theorem which is a truncated version of Theorem \ref{d9}:

\begin{thm}\label{G3}
Let $O$ be a well pointed operad and $\eta:\mathcal{L}(O)\rightarrow O'$ be a map in the category $Op[O\,;\,O]$. If $O$ is $\Sigma$-cofibrant, then the following weak equivalence holds:
\begin{equation*}
T_{k}Bimod_{O}^{h}(T_{k}(O)\,;\,T_{k}(\mathcal{R}(O')))  \simeq  \Omega\big(\, T_{k}Operad^{h}(T_{k}(O)\,;\,T_{k}(O'_{c}))\,;\,T_{k}Op[O\,;\,\emptyset]^{h}(T_{k}(\mathcal{L}(O))\,;\,T_{k}(O'))\,\big).
\end{equation*}
\end{thm}

For this purpose, we need a cofibrant replacement of $T_{k}(\mathcal{L}(O;O;O))$ in the category $T_{k}Op[O\,;\,\mathcal{BV}_{k}(O)]$. So, we change slightly the filtration of $\mathcal{B}_{\emptyset}(O)$ introduced in Section \ref{h1}. The definition of prime points and composite points in $\mathcal{B}_{\emptyset}(O)$ are still the same. Nevertheless, we consider two kinds of prime component. First, there are prime components in $\mathcal{B}_{\emptyset}(O)$ considered in Section \ref{h1}. Then, there are prime components in $\mathcal{BV}(O)$ coming from sub-trees above the section whose trunks are indexed by $1$. So, a prime point is said to be in the $n$-th filtration term $\mathcal{B}^{k}_{\emptyset}(O)_{n}$ if the number of its geometrical inputs is smaller than $n$. Besides, a composite point is said to be in the $n$-th filtration term if its prime components are in $\mathcal{B}^{k}_{\emptyset}(O)_{n}$ or $\mathcal{BV}_{k}(O)$. Finally, $\mathcal{B}^{k}_{\emptyset}(O)_{n}$ is a ($O\text{-}\mathcal{BV}_{k}(O)$) bimodule and we denote by $\mathcal{BV}_{\emptyset}^{k}(O)$ the following $\{o\,;\,c\}$-operad:
$$
\mathcal{BV}_{\emptyset}^{k}(O):=\mathcal{L}(\mathcal{B}^{k}_{\emptyset}(O)_{k}\,;\,O\,;\,\mathcal{BV}_{k}(O)).
$$

\begin{lmm}\label{G2}
$T_{k}(\mathcal{B}_{\emptyset}^{k}(O)_{k})$ is a cofibrant replacement of $T_{k}(O)$ in the category of $k$ truncated ($O$-$\mathcal{BV}_{k}(O)$) bimodules.
\end{lmm}

\begin{proof}
Similarly to the proof of Lemma \ref{e1}, $\mathcal{B}^{k}_{\emptyset}(O)_{k}$ is cofibrant in the category of ($O$-$\mathcal{BV}_{k}(O)$) bimodules. So, $T_{k}(\mathcal{B}_{\emptyset}^{k}(O)_{k})$ is cofibrant in the category of $k$ truncated ($O$-$\mathcal{BV}_{k}(O)$) bimodules. Furthermore, the restriction of the homotopy inverse introduced in Lemma \ref{F0} gives rise a map $\mu'': \mathcal{B}^{k}_{\emptyset}(O)_{k} \rightarrow O$. In what follows, we prove that $T_{k}(\mu'')$ is a weak equivalence.

Let $T_{k}(\mathcal{B}_{\emptyset}^{k}(O)_{k})[l]$ be the sub-sequence formed by points having at most $l$ vertices above the section. Due to the axiom $(iii)$ of Construction \ref{b7}, $T_{k}(\mathcal{B}_{\emptyset}^{k}(O)_{k})[l]$ doesn't have univalent pearl. From the identification (\ref{E8}), $T_{k}(\mathcal{B}_{\emptyset}^{k}(O)_{k})[0]$ is equivalent to its sub-sequence formed by points whose pearls are labelled by the unit $\ast_{O}\in O(1)$. Consequently,  $T_{k}(\mathcal{B}_{\emptyset}^{k}(O)_{k})[0]$ is homotopy equivalent to $T_{k}(O)$ and the homotopy consists in bringing the real numbers indexing the vertices below the section to $1$.

According to the notation introduced in Section  \ref{B6}, let $\textbf{stree}_{k\,;\,l}$ be the set of planar trees with section having at most $k$ leaves and exactly $l$ vertices above the section. Similarly to the proof of Theorem \ref{B7}, on has the pushout diagram 
\begin{equation}\label{G1}
\xymatrix{
\underset{[T\,;\,V^{p}(T)]}{\coprod} (H_{2}\times \underline{M})^{-}(T) \underset{Aut(T\,;\,V^{p}(T))}{\times} \Sigma_{|T|}  \ar[r]\ar[d] & \underset{[T\,;\,V^{p}(T)]}{\coprod} (H_{2}(T)\times \underline{M}(T)) \underset{Aut(T\,;\,V^{p}(T))}{\times} \Sigma_{|T|}\ar[d] \\
T_{k}(\mathcal{B}_{\emptyset}^{k}(O)_{k})[l-1] \ar[r] & T_{k}(\mathcal{B}_{\emptyset}^{k}(O)_{k})[l] 
}
\end{equation}
where the disjoint union is along the isomorphism classes of $\textbf{stree}_{k\,;\,l}$. The space $H_{2}(T)$ is the space of parametrizations of the vertices below the section and the inner edges above the section by real numbers satisfying the restriction coming from the construction of $T_{k}(\mathcal{B}_{\emptyset}^{k}(O)_{k})$. Similarly, $H_{2}^{-}(T)$ is the subspace of $H_{2}(T)$ having at least one inner edge above the section indexed by $0$.

Similar arguments used in the proof of Theorem \ref{B7} imply that the horizontal maps of Diagram (\ref{G1}) are acyclic $\Sigma$-cofibrations if the inclusion from $H_{2}^{-}(T)$ to $H_{2}(T)$ is an acyclic cofibration. Since $H_{2}^{-}(T)\rightarrow H_{2}(T)$ is an inclusion of $CW$-complex, this map is a cofibration. Moreover, $H_{2}(T)$ is obviously contractible and the homotopy consists in bringing the parameters to $1$. For the space $H_{2}^{-}(T)$, there are two cases to consider. First, there is no univalent vertex above the section. In that case, $H_{2}^{-}(T)$ is contractible and the homotopy is the same as the previous one. 

In the second case, there is at least one univalent vertex above the section. Hence, the homotopy consists in bringing the real number indexing the output edge of the univalent vertex to $0$. This homotopy is well define since  univalent vertices doesn't change the number of geometrical inputs. Thereafter, we bring all the other vertices to $1$. 
Finally, we proved that the map $T_{k}(\mathcal{B}_{\emptyset}^{k}(O)_{k})[l-1]\rightarrow T_{k}(\mathcal{B}_{\emptyset}^{k}(O)_{k})[l]$ is an acyclic $\Sigma$-cofibration. In particular, one has 
$$
\xymatrix{
T_{k}(\mathcal{B}_{\emptyset}^{k}(O)_{k})[0] \ar[r]^{\simeq} \ar@/^-1pc/[rr]_{\simeq} & T_{k}(\mathcal{B}_{\emptyset}^{k}(O)_{k}) \ar[r]^{\hspace{10pt}T_{k}(\mu''_{k})} & T_{k}(O).
}
$$
From the \textbf{2to3} axiom in model category theory, $T_{k}(\mu''_{k})$ is a weak equivalence in the category of $k$ truncated sequences. Finally, it is also a weak equivalence in the category of $k$ truncated ($O$-$\mathcal{BV}_{k}(O)$) bimodules. 
\end{proof}

\begin{lmm}\label{G4}
$T_{k}(\mathcal{BV}_{\emptyset}^{k}(O))$ is a cofibrant replacement of $T_{k}(\mathcal{L}(O))$ in the category $T_{k}Op[O\,;\,\emptyset]$.
\end{lmm}

\begin{proof}
Since the functor $\mathcal{L}$ preserves cofibrant objects, $\mathcal{BV}_{\emptyset}^{k}(O)$ is cofibrant in the categories $Op[O\,;\,\mathcal{BV}_{k}(O)]$ and $Op[O\,;\,\emptyset]$. Consequently, $T_{k}(\mathcal{BV}_{\emptyset}^{k}(O))$ is also cofibrant in $T_{k}Op[O\,;\,\emptyset]$. Furthermore, the restriction of the homotopy inverse introduced in Lemma \ref{F0} gives rise a map $\mu'': \mathcal{B}^{k}_{\emptyset}(O)_{k} \rightarrow O$. There is also a map $\mu_{k}:\mathcal{BV}_{k}(O) \rightarrow O$ which sends the real numbers indexing the inner edges to $0$. As shown in Remark \ref{F8}, the maps $\mu''$ and $\mu_{k}$ induce an $\{o\,;\,c\}$-operadic map $\beta: \mathcal{BV}_{\emptyset}^{k}(O)\rightarrow \mathcal{L}(O)$. Similarly to the proof of Lemma \ref{G2}, the reader can check that the map $T_{k}(\beta)$ is a weak equivalence. 
\end{proof}

\begin{proof}[Proof of Theorem \ref{G3}]
Since the arguments are the same as in the previous sections, we only give the main steps of the proof. One has the following weak equivalences:
$$
\begin{array}{rclc}\vspace{4pt}
T_{k}Bimod_{O}^{h}(T_{k}(O)\,;\,T_{k}(\mathcal{R}(O'))) & \simeq & Bimod_{O}(\mathcal{B}_{k}(O)\,;\,\mathcal{R}(O')), & (1) \\ \vspace{4pt}
 & \simeq & Op[O\,;\,O](\mathcal{L}(\mathcal{B}_{k}(O);O;O)\,;\,O'), & (2)\\ \vspace{4pt}
 & \simeq & Op[O\,;\,\mathcal{BV}_{k}(O)](\mathcal{BV}_{\emptyset}^{k}(O)\,;\,O'), & (3)\\ \vspace{4pt}
 & \simeq & \Omega \big(\, Operad(\mathcal{BV}_{k}(O)\,;\,O_{c})\,\,;\,\, Op[O\,;\,\emptyset](\mathcal{BV}_{\emptyset}^{k}(O)\,;\,O')\,\big), & (4)\\ 
 & \simeq & \Omega\big(\, T_{k}Operad^{h}(T_{k}(O)\,;\,T_{k}(O'_{c}))\,;\,T_{k}Op[O\,;\,\emptyset]^{h}(T_{k}(\mathcal{L}(O))\,;\,T_{k}(O'))\,\big).& (5)
\end{array} 
$$ 
Steps $(1)$ and $(5)$ are consequences of the cofibrant replacements introduced in Theorem \ref{B7} and Lemma\ref{G4} respectively. Step $(2)$ is a consequence of the adjunction $(\mathcal{L}\,;\,\mathcal{R})$. Finally, the proofs of steps $(3)$ and $(4)$ are equivalent to the proofs of Proposition \ref{d5} and Proposition \ref{d6} respectively.  
\end{proof}

\begin{thm}\label{F9}
Let $O$ be a well pointed $\Sigma$-cofibrant operad. Let $\eta:\mathcal{L}(O)\rightarrow O'$ be a map in the category $Op[O\,;\,O]$ in which the spaces $O(1)$ and $O'(1)$ are contractible. Then, the pair of topological spaces
$$
\big(\,T_{k}Bimod^{h}_{O}(T_{k}(O)\,;\,T_{k}(O'_{c}))\,\,;\,\,T_{k}Bimod^{h}_{O}(T_{k}(O)\,;\,T_{k}(\mathcal{R}(O')))\,\big)
$$ 
is weakly equivalent to the $\mathcal{SC}_{1}$-algebra
$$
\big(\,\, \Omega\big(\,T_{k}Operad^{h}(T_{k}(O)\,;\,T_{k}(O'_{c}))\,\big)\,\,;\,\,
\Omega\big(\,\,T_{k}Operad^{h}(\,T_{k}(O)\,;\,T_{k}(O'_{c})\,)\,\,;\,\, T_{k}Op[O\,;\,\emptyset]^{h}(\,T_{k}(\mathcal{L}(O))\,\,;\,\,T_{k}(O')\,)\,\,\big)\,\,\big).
$$
\end{thm}

\begin{cor}
Let $\eta_{1}:O\rightarrow O'$ be a map of operads and $\eta_{2}:O'\rightarrow M$ be a map of $O'$-bimodules. Assume that $O$ is a well pointed $\Sigma$-cofibrant operad and the spaces $O(1)$ and $O'(1)$ are contractible. Then, the pair of topological spaces
$$
\big(\, T_{k}Bimod^{h}_{O}(T_{k}(O)\,;\, T_{k}(O'))\,\,;\,\,  T_{k}Bimod^{h}_{O}(T_{k}(O)\,;\, T_{k}(M))\,\big)
$$  
is weakly equivalent to the $\mathcal{SC}_{1}$-algebra
$$
\left(\,
\Omega\big(\,  T_{k}Operad^{h}(T_{k}(O)\,;\,T_{k}(O'))\,\big)\,\,;\,\,\Omega\big(\,\, T_{k}Operad^{h}(\,T_{k}(O)\,;\,T_{k}(O')\,)\,;\, T_{k}Op[O\,;\,\emptyset]^{h}(\,T_{k}(\mathcal{L}(O))\,;\,T_{k}(\mathcal{L}(M;O';O')))\,\big)\, 
\right).
$$
\end{cor}

\subsection{Generalization using Dwyer-Hess' conjecture}\label{i4}

In this section we use the previous theorems in order to identify algebras over the Swiss-Cheese operad $\mathcal{SC}_{d+1}$ from maps between $\{o\,;\,c\}$-operads. For this purpose, we will assume a conjecture introduced by Dwyer and Hess. A version of this conjecture was proved by Boavida de Brito and Weiss \cite{Weiss15} in the particular case $M=\mathcal{C}_{n}$ with $n>d$. In the incoming paper \cite{Ducoulombier16.2}, Turchin and the author prove the following conjecture: 

\begin{conj}\label{f5}
Let $\eta:\mathcal{C}_{d}\rightarrow M$ be a map of $\mathcal{C}_{d}$-bimodules with $M(0)\simeq \ast$. one has
$$
\begin{array}{rcl}\vspace{4pt}
Ibimod_{\mathcal{C}_{d}}^{h}(\mathcal{C}_{d}\,;\, M) & \simeq & \Omega^{d}\big(Bimod_{\mathcal{C}_{d}}^{h}(\mathcal{C}_{d}\,;\,M)\big) \\ 
T_{k}Ibimod_{\mathcal{C}_{d}}^{h}(T_{k}(\mathcal{C}_{d})\,;\, T_{k}(M)) & \simeq & \Omega^{d}\big(T_{k}Bimod_{\mathcal{C}_{d}}^{h}(T_{k}(\mathcal{C}_{d})\,;\,T_{k}(M))\big)
\end{array} 
$$
\end{conj}

In what follows, the theorems and corollaries are direct consequences of the statements introduced in the previous sections together with the Dwyer-Hess' conjecture. We denote by $\mathcal{CC}_{d}$ the two-coloured operad defined as follows:
$$
\mathcal{CC}_{d}=\mathcal{L}(\mathcal{C}_{d}\,;\,\mathcal{C}_{d}\,;\,\mathcal{C}_{d}).
$$

\begin{thm}\label{F6}
Assume that Conjecture \ref{f5} is true. Let $\eta:\mathcal{CC}_{d}\rightarrow O$ be a map in the category $Op[\mathcal{C}_{d}\,;\,\mathcal{C}_{d}]$, with $O(\,;\,c)\simeq O(c\,;\,c)\simeq\ast$ and $O(\,;\,o)\simeq \ast$. Then, the pair of topological spaces
$$
\big(\,Ibimod^{h}_{\mathcal{C}_{d}}(\mathcal{C}_{d}\,;\,O_{c})\,\,;\,\,Ibimod^{h}_{\mathcal{C}_{d}}(\mathcal{C}_{d}\,;\,\mathcal{R}(O))\,\big)
$$
is weakly equivalent to the $\mathcal{SC}_{d+1}$-algebra 
$$
\left(\,\,
\Omega^{d+1}\big( Operad^{h}(\mathcal{C}_{d}\,;\,O_{c})\big)\,\,;\,\,\Omega^{d+1}\big(\,\,Operad^{h}(\,\mathcal{C}_{d}\,;\,O_{c}\,)\,\,;\,\, Op[\mathcal{C}_{d}\,;\,\emptyset]^{h}(\,\mathcal{CC}_{d}\,\,;\,\,O\,)\,\,\big)\,\, 
\right).
$$
\end{thm}

\begin{cor}\label{G6}
Let $\eta_{1}:\mathcal{C}_{d}\rightarrow O$ be a map of operads and $\eta_{2}:O\rightarrow M$ be a map of $O$-bimodules. Assume that the spaces $O(0)$, $O(1)$ and $M(0)$ are contractible. Then, the pair of topological spaces
$$
\big(\, Ibimod^{h}_{\mathcal{C}_{d}}(\mathcal{C}_{d}\,;\, O)\,\,;\,\,  Ibimod^{h}_{\mathcal{C}_{d}}(\mathcal{C}_{d}\,;\, M)\,\big)
$$  
is weakly equivalent to the $\mathcal{SC}_{d+1}$-algebra
$$
\left(\,
\Omega^{d+1}\big(\,  Operad^{h}(\mathcal{C}_{d}\,;\,O)\,\big)\,\,;\,\,\Omega^{d+1}\big(\,\, Operad^{h}(\,\mathcal{C}_{d}\,;\,O\,)\,;\, Op[\mathcal{C}_{d}\,;\,\emptyset]^{h}(\,\mathcal{CC}_{d}\,;\,\mathcal{L}(M;O;O))\,\big)\, 
\right).
$$
\end{cor}

\begin{thm}
Assume that Conjecture \ref{f5} is true. Let $\eta:\mathcal{CC}_{d}\rightarrow O$ be a map in the category $Op[\mathcal{C}_{d}\,;\,\mathcal{C}_{d}]$, with $O(\,;\,c)\simeq O(c\,;\,c)\simeq\ast$ and $O(\,;\,o)\simeq\ast$. Then, the pair of topological spaces
$$
\big(\,T_{k}Ibimod^{h}_{\mathcal{C}_{d}}(T_{k}(\mathcal{C}_{d})\,;\,T_{k}(O_{c}))\,\,;\,\,T_{k}Ibimod^{h}_{\mathcal{C}_{d}}(T_{k}(\mathcal{C}_{d})\,;\,T_{k}(\mathcal{R}(O)))\,\big)
$$
is weakly equivalent to the $\mathcal{SC}_{d+1}$-algebra 
$$
\left(\,\,
\Omega^{d+1}\big( T_{k}Operad^{h}(T_{k}(\mathcal{C}_{d})\,;\,T_{k}(O_{c}))\big)\,\,;\,\,\Omega^{d+1}\big(\,\,T_{k}Operad^{h}(\,T_{k}(\mathcal{C}_{d})\,;\,T_{k}(O_{c})\,)\,\,;\,\, T_{k}Op[\mathcal{C}_{d}\,;\,\emptyset]^{h}(\,T_{k}(\mathcal{CC}_{d})\,\,;\,\,T_{k}(O)\,)\,\,\big)\,\, 
\right).
$$
\end{thm}

\begin{cor}\label{G7}
Let $\eta_{1}:\mathcal{C}_{d}\rightarrow O$ be a map of operads and $\eta_{2}:O\rightarrow M$ be a map of $O$-bimodules. Assume that the spaces $O(0)$, $O(1)$ and $M(0)$ are contractible. Then, the pair of topological spaces
$$
\big(\, T_{k}Ibimod^{h}_{\mathcal{C}_{d}}(T_{k}(\mathcal{C}_{d})\,;\, T_{k}(O))\,\,;\,\,  T_{k}Ibimod^{h}_{\mathcal{C}_{d}}(T_{k}(\mathcal{C}_{d})\,;\, T_{k}(M))\,\big)
$$  
is weakly equivalent to the $\mathcal{SC}_{d+1}$-algebra
$$
\left(
\Omega^{d+1}\big(  T_{k}Operad^{h}(T_{k}(\mathcal{C}_{d})\,;\,T_{k}(O))\big)\,;\,\Omega^{d+1}\big( T_{k}Operad^{h}(T_{k}(\mathcal{C}_{d})\,;\,T_{k}(O))\,;\, T_{k}Op[\mathcal{C}_{d}\,;\,\emptyset]^{h}(T_{k}(\mathcal{CC}_{d})\,;\,T_{k}(\mathcal{L}(M;O;O)))\big) 
\right).
$$
\end{cor}\vspace{-10pt}

\section{Application to the spaces of embeddings and $(l)$-immersions}\label{F7}

In \cite{Sinha06}, Sinha proves that the space of long embeddings $\overline{Emb}_{c}(\mathbb{R}^{1}\,;\,\mathbb{R}^{n})$ is weakly equivalent to the homotopy totalization of the Kontsevich operad. As a consequence, $\overline{Emb}_{c}(\mathbb{R}^{1}\,;\,\mathbb{R}^{n})$ is weakly equivalent to an explicit double loop space. For the space of long knots in higher dimension, defined by the homotopy fiber
\begin{equation}\label{f7}
\overline{Emb}_{c}(\mathbb{R}^{d}\,;\,\mathbb{R}^{n}):=hofib\big(\,Emb_{c}(\mathbb{R}^{d}\,;\,\mathbb{R}^{n})\,\longrightarrow Imm_{c}(\mathbb{R}^{d}\,;\,\mathbb{R}^{n})\,\big),
\end{equation}
there is no cosimplicial replacement known. However, Arone and Turchin develop in \cite{Arone14} a machinery , based on the Goodwillie calculus \cite{Weiss99}, in order to identify embedding spaces with spaces of infinitesimal bimodules over the little cubes operad. In this section, we recall the notion of embedding calculus as well as Arone and Turchin's results. Then, we give an application to the space of $(l)$-immersions assuming the Dwyer-Hess' conjecture.

\newpage

\subsection{The Goodwillie Calculus}  

Let $M$ be a smooth manifold of dimension $d$. Let $\mathcal{O}(M)$ be the poset of open subsets of $M$. For $k\in\mathbb{N}$, $\mathcal{O}_{k}(M)$ is formed by elements in $\mathcal{O}(M)$  homeomorphic to a disjoint union of at most $k$ open disks in $\mathbb{R}^{d}$. From a contravariant functor $F:\mathcal{O}(M)\rightarrow Top$, Weiss introduces in \cite{Weiss99} the functor $T_{k}F:\mathcal{O}(M)\rightarrow Top$ as follows:
$$
T_{k}F(U):=\underset{\hspace{15pt}V\in \mathcal{O}_{k}(U)}{holim}\,\,F(V).
$$
The functor $T_{k}F$ is called the \textit{$k$-th polynomial approximation} of $F$. The restriction to $\mathcal{O}_{k-1}(M)$ induces a natural transformation $T_{k}F\rightarrow T_{k-1}F$. There is a tower, called the \textit{Taylor tower}, associated to the functor $F$:
$$
\xymatrix{
 & F \ar[d] \ar[dr] \ar[dl] \ar[drr] & & & \\
T_{0}F & T_{1}F \ar[l] & T_{2}F \ar[l] & T_{3}F \ar[l] &\cdots \ar[l] 
}
$$
In good cases, the homotopy limit of the Taylor tower, denoted by $T_{\infty}F$, is a functor equivalent to $F$ (i.e. $\forall U\in \mathcal{O}(M)$, $T_{\infty}F(U)\simeq F(U)$) and one says that the Taylor tower converges. 

Taylor towers have been used by Goodwillie \cite{Goodwillie90, Goodwillie91, Goodwillie03} and Weiss \cite{Weiss99.2, Weiss99} in order to study embedding spaces. The idea is to build a contravariant functor
$$
Emb(-\,;\,N):\mathcal{O}(M)\rightarrow Top,
$$
where $N$ is a smooth manifold of dimension $n$. In a similar way, one can study the space of immersions $Imm(M\,;\,N)$ as well as the space of long embeddings defined as follows:
$$
\overline{Emb}(M;N):= hofib\big(\,Emb(M\,;\,N)\,\longrightarrow Imm(M\,;\,N)\,\big).
$$  

\begin{thm}{\cite{Weiss14}}
Let $M$ and $N$ be two smooth manifolds of dimension $d$ and $n$ respectively. If $n-d-2>0$, then the Taylor towers associated to the functors $Emb(-\,;\,N)$, $Imm(-\,;\,N)$ and $\overline{Emb}(-\,;\,N)$ converge.
\end{thm}

If one considers the embedding spaces and immersion spaces with compact support, then one has to change slightly the construction of the Taylor towers. Let $M$ be the complementary of a compact subspace of $\mathbb{R}^{d}$. Let $\mathcal{O}'(M)$ be the poset of open subsets of $M$ of the form $V\cup W$ with $W$ the complementary of a closed disk, $V\in \mathcal{O}(M)$ and $V\cap W=\emptyset$. For $k\in \mathbb{N}$, $\mathcal{O}'_{k}(M)$ is formed by points $V\cup W$ with $V\in \mathcal{O}_{k}(M)$. From a contravariant functor $F:\mathcal{O}(M)\rightarrow Top$, the $k$-th polynomial approximation functor of $F$, denoted $T_{k}F:\mathcal{O}'(M)\rightarrow Top$ by abuse of notation, is defined as follows:
$$
T_{k}F(U):=\underset{\hspace{15pt}V\in \mathcal{O}'_{k}(U)}{holim}\,\,F(V).
$$

\begin{thm}{\cite{Weiss14}}\label{f8}
Let $M$ be the complementary of a compact subspace of $\mathbb{R}^{d}$. If $n-d-2>0$, then the Taylor towers associated to the functors $Emb_{c}(-\,;\,\mathbb{R}^{n})$, $Imm_{c}(-\,;\,\mathbb{R}^{n})$ and $\overline{Emb}_{c}(-\,;\,\mathbb{R}^{n})$ converge. In particular, the Taylor tower associated to the space of long knots in higher dimension (\ref{f7}) converges.
\end{thm}

\subsection{Connections between Taylor towers and infinitesimal bimodules}

A standard isomorphism in $\mathbb{R}^{d}$ arises from an affine embedding preserving the direction of the axes. Let $A=\cup_{s\in S}\, A_{s}$ be a disjoint union of open subsets of $\mathbb{R}^{d}$ indexed by a set $S$ and let $M$ be a subspace of $\mathbb{R}^{d}$. A map $f:A\rightarrow M$ is called a \textit{standard embedding} if $f$ is an embedding such that, for each $s\in S$, the composite map $A_{s}\rightarrow M \hookrightarrow \mathbb{R}^{d}$ is an inclusion in $\mathbb{R}^{d}$ followed by to a standard isomorphism. We denote the space of standard embeddings from $A$ to $M$ by $sEmb(A\,;\,M)$.

\begin{defi}
Let $M$ be an open subset of $\mathbb{R}^{d}$. The sequence $sEmb(-\,;\,M)$ is given by
$$
sEmb(-\,;\,M)(k)=sEmb(\,\underset{i=1}{\overset{k}{\sqcup}}\,\mathcal{C}^{d}\,;\,M),
$$
where $\mathcal{C}^{d}$ is the unit little cube of dimension $d$. For $M=\mathcal{C}^{d}$, the sequence $sEmb(-\,;\,\mathcal{C}^{d})$ is the little cubes operad $\mathcal{C}_{d}$. In general, $sEmb(-\,;\,M)$ is only a right $\mathcal{C}_{d}$-module whose operations
$$
\circ^{i}:\,sEmb(-\,;\,M)(n)\times\mathcal{C}_{d}(m)\rightarrow sEmb(-\,;\,M)(n+m-1) 
$$
are induced by the composition with the standard embeddings of $\mathcal{C}_{d}$. 
\end{defi}

\begin{thm}{\cite[Theorem 5.10]{Arone14}}\label{f9}
Let $M$ be an open subset of $\mathbb{R}^{d}$ with $d<n$. There are the weak equivalences:
$$
\left\{
\begin{array}{rcl}
T_{\infty}\overline{Emb}(M\,;\,\mathbb{R}^{n}) & \simeq & Rmod^{h}_{\mathcal{C}_{d}}(sEmb(-\,;\,M)\,;\,\mathcal{C}_{n}), \\ 
T_{k}\overline{Emb}(M\,;\,\mathbb{R}^{n}) & \simeq & T_{k}Rmod^{h}_{\mathcal{C}_{d}}(T_{k}(sEmb(-\,;\,M))\,;\,T_{k}(\mathcal{C}_{n})),
\end{array} 
\right.
$$
where $Rmod^{h}_{\mathcal{C}_{d}}$ and $Rmod^{h}_{\mathcal{C}_{d}}$ are the categories of right modules and $k$ truncated right modules over $\mathcal{C}_{d}$ respectively. In the particular case $M=\mathcal{C}_{d}$, one has
$$
\left\{
\begin{array}{rcl}
T_{\infty}\overline{Emb}_{c}(\mathbb{R}^{d}\,;\,\mathbb{R}^{n}) & \simeq & Ibimod_{\mathcal{C}_{d}}^{h}(\mathcal{C}_{d}\,;\,\mathcal{C}_{n}) ,\\ 
T_{k}\overline{Emb}_{c}(\mathbb{R}^{d}\,;\,\mathbb{R}^{n}) & \simeq & T_{k}Ibimod_{\mathcal{C}_{d}}^{h}(T_{k}(\mathcal{C}_{d})\,;\,T_{k}(\mathcal{C}_{n})).
\end{array} 
\right.
$$
Finally, if $M$ is the complementary of a compact subset of $\mathbb{R}^{d}$, then there are the weak equivalences
$$
\left\{
\begin{array}{rcl}
T_{\infty}\overline{Emb}_{c}(M\,;\,\mathbb{R}^{n}) & \simeq & Ibimod^{h}_{\mathcal{C}_{d}}(sEmb(-\,;\,M)\,;\,\mathcal{C}_{n}) ,\\ 
T_{k}\overline{Emb}_{c}(M\,;\,\mathbb{R}^{n}) & \simeq & T_{k}Ibimod^{h}_{\mathcal{C}_{d}}(T_{k}(sEmb(-\,;\,M))\,;\,T_{k}(\mathcal{C}_{n})).
\end{array} 
\right.
$$
\end{thm}

\begin{cor}\label{g4}
Assume Conjecture \ref{f5} is true. If $n-d-2>0$, then one has
$$
\overline{Emb}_{c}(\mathbb{R}^{d}\,;\,\mathbb{R}^{n})\simeq \Omega^{d+1}Operad^{h}(\mathcal{C}_{d}\,;\,\mathcal{C}_{n})\hspace{15pt}\text{and}\hspace{15pt} T_{k}\overline{Emb}_{c}(\mathbb{R}^{d}\,;\,\mathbb{R}^{n})\simeq \Omega^{d+1}( T_{k}Operad^{h}(T_{k}(\mathcal{C}_{d})\,;\,T_{k}(\mathcal{C}_{n})).
$$
\end{cor}

\begin{proof}
It is a consequence of Conjecture \ref{f5} together with Theorem \ref{f8} and \ref{f9}.
\end{proof}

\subsection{The space of \texorpdfstring{$(l)$}{Lg}-immersions}

The space of $(l)$-immersions, denoted by $Imm_{c}^{(k)}(\mathbb{R}^{d}\,;\,\mathbb{R}^{n})$, is the subspace of $Imm_{c}(\mathbb{R}^{d}\,;\,\mathbb{R}^{n})$ of immersions $f$ such that for each subset of $l$ distinct elements $K\subset \mathbb{R}^{d}$, the restriction $f_{|K}$ is non constant. In particular, the space of $(2)$-immersions is the embedding space $Emb_{c}(\mathbb{R}^{d}\,;\,\mathbb{R}^{n})$. The spaces of $(l)$-immersions give rise a filtration of the inclusion from the space of embeddings to the space of immersions: 
$$
\xymatrix{
Emb_{c}(\mathbb{R}^{d}\,;\,\mathbb{R}^{n}) \ar[r] & Imm_{c}^{(3)}(\mathbb{R}^{d}\,;\,\mathbb{R}^{n}) \ar[r] & 
\cdots \ar[r] & Imm_{c}^{(l)}(\mathbb{R}^{d}\,;\,\mathbb{R}^{n}) \ar[r] & \cdots \ar[r] & Imm_{c}(\mathbb{R}^{d}\,;\,\mathbb{R}^{n}) 
}
$$
The space of long $(l)$-immersions is defined by the following homotopy fiber:
$$
\overline{Im}m^{(l)}_{c}(\mathbb{R}^{d}\,;\,\mathbb{R}^{n}):=hofib\big(\,Imm^{(l)}_{c}(\mathbb{R}^{d}\,;\,\mathbb{R}^{n})\longrightarrow Imm_{c}(\mathbb{R}^{d}\,;\,\mathbb{R}^{n})\,\big).
$$
\begin{center}
\begin{figure}[!h]
\begin{center}
\includegraphics[scale=0.5]{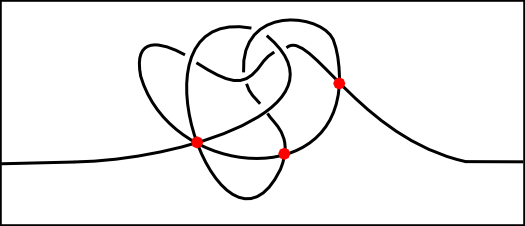}
\caption{Illustration of a point in $\overline{Im}m^{(4)}_{c}(\mathbb{R}^{1}\,;\,\mathbb{R}^{3})$.}
\end{center}
\end{figure}
\end{center}
Contrary to the space of long knots in higher dimension, we don't know if the Taylor tower associated to the space of long $(l)$-immersions converges. Nevertheless, Dobrinskaya and Turchin have been able to identify the homotopy limit of the Taylor tower with a space of infinitesimal bimodule maps by using the non-$(l)$-overlapping little cubes bimodule. 

\begin{thm}{\cite[Theorem 11.2]{Turchin14}} \label{g5}
If $n>d$, then on has the following weak equivalences:
$$
T_{k}\overline{Im}m_{c}^{(l)}(\mathbb{R}^{d}\,;\,\mathbb{R}^{n})\simeq T_{k}Ibimod_{\mathcal{C}_{d}}^{h}(T_{k}(\mathcal{C}_{d})\,;\,T_{k}(\mathcal{C}_{n}^{(k)}))
\hspace{15pt}\text{and}\hspace{15pt}
T_{\infty}\overline{Im}m_{c}^{(l)}(\mathbb{R}^{d}\,;\,\mathbb{R}^{n})\simeq Ibimod_{\mathcal{C}_{d}}^{h}(\mathcal{C}_{d}\,;\,\mathcal{C}_{n}^{(k)}).
$$
\end{thm}

As shown in Section \ref{G5}, the non-$(l)$-overlapping little cubes bimodule is not an operad. Nevertheless, the inclusion $\eta_{2}:\mathcal{C}_{n}\rightarrow \mathcal{C}_{l}^{(k)}$ preserves the $\mathcal{C}_{n}$-bimodule structures. So, Corollaries \ref{G6} and \ref{G7}, applied to the bimodule map $\eta_{2}$ and the operadic map $\eta_{1}:\mathcal{C}_{d}\rightarrow \mathcal{C}_{n}$, imply the following theorem:

\begin{thm}
Assume that Dwyer-Hess' conjecture \ref{f5} is true. If $n-d-2>0$, then the pair of spaces 
$$
(\overline{Emb}_{c}(\mathbb{R}^{d}\,;\,\mathbb{R}^{n})\,;\,T_{\infty}\overline{Im}m_{c}^{(l)}(\mathbb{R}^{d}\,;\,\mathbb{R}^{n}))$$ 
is weakly equivalent to the $\mathcal{SC}_{d+1}$-algebra
$$
\left(\,\,
\Omega^{d+1}Operad^{h}(\mathcal{C}_{d}\,;\,\mathcal{C}_{n})\,\,;\,\,\Omega^{d+1}\big(\,\,Operad^{h}(\,\mathcal{C}_{d}\,;\,\mathcal{C}_{n}\,)\,\,;\,\, Op[\mathcal{C}_{d}\,;\,\emptyset]^{h}(\,\mathcal{CC}_{d}\,\,;\,\,\mathcal{L}(\mathcal{C}_{n}^{(l)}\,;\,\mathcal{C}_{n}\,;\,\mathcal{C}_{n})\,)\,\,\big)\,\, 
\right).
$$
Similarly, the pair of $k$ polynomial approximation
$$
(T_{k}\overline{Emb}_{c}(\mathbb{R}^{d}\,;\,\mathbb{R}^{n})\,;\,T_{k}\overline{Im}m_{c}^{(l)}(\mathbb{R}^{d}\,;\,\mathbb{R}^{n}))
$$
is weakly equivalent to the $\mathcal{SC}_{d+1}$-algebra
$$
\left(
\Omega^{d+1}\big(T_{k}Operad^{h}(T_{k}(\mathcal{C}_{d});T_{k}(\mathcal{C}_{n}))\big);\Omega^{d+1}\big(T_{k}Operad^{h}(T_{k}(\mathcal{C}_{d});T_{k}(\mathcal{C}_{n})); T_{k}Op[\mathcal{C}_{d};\emptyset]^{h}(T_{k}(\mathcal{CC}_{d});T_{k}(\mathcal{L}(\mathcal{C}_{n}^{(l)};\mathcal{C}_{n};\mathcal{C}_{n})))\big) 
\right).
$$
\end{thm}

\begin{merci}
I would like to thank Muriel Livernet for her help in preparing this paper and Eric Hoffbeck for many helpful comments. I also wish to express my gratitude to Victor Turchin for suggesting the application to the space of $(l)$-immersions and ideas which are at the origin of the second version. I am also grateful to the Max Planck Institute for its hospitality and its financial support.
\end{merci}

\bibliographystyle{amsplain}
\bibliography{bibliography}

\end{document}